\documentclass[11pt]{article}
\usepackage[utf8]{inputenc}

\usepackage{setspace}
\usepackage[a4paper, margin=1in]{geometry}

\title{Optimization methods for the capacitated refueling\\ station location problem with routing}
\author{Nicholas Nordlund$^*$, Leandros Tassiulas$^*$ and Jan-Hendrik Lange$^\dagger$\\
$^*$Yale University\\
$^\dagger$Amazon}

\PassOptionsToPackage{hyphens}{url}
\usepackage[breaklinks]{hyperref}
    \hypersetup{colorlinks=true, linkcolor=black, citecolor=black, urlcolor=black}

\usepackage{microtype}
\usepackage{xspace} 

\usepackage[shortlabels, inline]{enumitem} 

\usepackage{booktabs}
\usepackage[retainorgcmds]{IEEEtrantools}
\usepackage{multirow}

\usepackage{algorithm}
\usepackage{algorithmicx}
\usepackage{algpseudocode}

\usepackage{tikz}

\usepackage{natbib}

\usepackage{amssymb}
\usepackage{amsmath}
\usepackage{amsfonts} 
\usepackage{bbm}

\usepackage{dsfont}

\usepackage{multicol}

\usepackage{graphicx}
\usepackage{subcaption}
\usepackage{adjustbox}
\usepackage{amsthm}

\providecommand{\R}{\mathbb{R}}

\providecommand{\A}{\mathcal{A}}

\providecommand{\G}{\mathcal{G}}

\providecommand{\V}{\mathcal{V}}
\renewcommand{\P}{\mathcal{P}}

\providecommand{\T}{\mathcal{T}}

\providecommand{\J}{\mathcal{J}}

\providecommand{\terminals}{\mathcal T}
\providecommand{\stations}{\mathcal R}

\providecommand{\odpairs}{\mathcal Q}

\providecommand{\rangemax}{R^{\mathrm{max}}}
\providecommand{\rangeorig}{R^{\mathrm{orig}}}
\providecommand{\rangedest}{R^{\mathrm{dest}}}

\providecommand{\capa}{\kappa}

\providecommand{\od}{\textsc{o-d}\xspace}
\providecommand{\afv}{\textsc{afv}\xspace}
\providecommand{\afvs}{\textsc{afv}s\xspace}

\providecommand{\evs}{\textsc{ev}s\xspace}

\providecommand{\frlm}{\textsc{frlm}\xspace}
\providecommand{\dfrlm}{\textsc{dfrlm}\xspace}
\providecommand{\rslpr}{\textsc{rslp-r}\xspace}
\providecommand{\rslprc}{\textsc{crslp-r}\xspace}

\theoremstyle{plain}

\newtheorem{prop}{Proposition}

\theoremstyle{definition}
\newtheorem{defn}{Definition}

\newtheorem{assump}{Assumption}

\theoremstyle{remark}
\newtheorem*{rem}{Remark}

\begin{document}

\maketitle

\begin{abstract}
The energy transition in transportation benefits from demand-based models to determine the optimal placement of refueling stations for alternative fuel vehicles such as battery electric trucks.
    A formulation known as the refueling station location problem with routing (\rslpr) is concerned with minimizing the number of stations necessary to cover a set of origin-destination trips such that the transit time does not exceed a given threshold.
    In this paper we extend the \rslpr by station capacities to limit the number of vehicles that can be refueled at individual stations.
    The solution to the capacitated \rslpr (\rslprc) avoids congestion of refueling stations by satisfying capacity constraints.
    We devise two optimization methods to deal with the increased difficulty to solve the \rslprc.
    The first method extends a prior branch-and-cut approach and the second method is a branch-cut-and-price algorithm based on variables associated with feasible routes.
    We evaluate both our methods on instances from the literature as well as a newly constructed network and find that the relative performance of the algorithms depends on the strictness of the capacity constraints.
    Furthermore, we show some runtime improvements over prior work on uncapacitated instances.
\end{abstract}


\section{Introduction}

In 2019, transportation accounted for 29\% of greenhouse gas emissions in the USA. Since 1990, emissions from the transportation sector have increased more in absolute terms than any other sector, with emissions from medium and heavy-duty trucks up 92\% \citep{EPA2022}.
In response, governments and corporations have begun to transition away from internal combustion vehicles towards ``greener'' alternatives. These \textit{alternative fuel vehicles} (\afvs) include electric vehicles (\evs), compressed natural gas vehicles, hydrogen fuel cell vehicles, and others.

Despite the environmental benefits, several challenges accompany the transition to \afvs. This includes longer refueling times (especially for \evs), limited vehicle ranges, and the sparse distribution of refueling infrastructure. 
As a result, \afvs in logistics take longer to traverse the same routes as diesel trucks, leading to shipping delays for companies like Amazon, Fedex, and UPS.
Here, several factors contribute to \afvs' increased transit times. 
Firstly, the longer refueling times directly add to transit times unless they are scheduled when the vehicle is idle for separate reasons.
Secondly, fewer refueling stations mean \afvs have to travel further out of their way to refuel and may even need to queue for limited refueling capacity at the station.

Currently, the lack of refueling infrastructure is a limiting factor on the commercial viability of heavy-duty \afvs for logistics companies.
The placement of new infrastructure requires careful planning, however.
Building refueling stations is expensive; there are limited numbers of locations suitable for new refueling stations; and logistics companies can only tolerate a certain amount of deviation compared to the transit times of diesel trucks.
    Therefore, the optimal placement of refueling stations for \afvs becomes an important aspect of the energy transition in transportation.

Prior works have presented formulations for the problem of placing new refueling stations, see \citep{kchaou2021charging} for a recent survey.
One particular formulation is the Refueling Station Location Problem with Routing (\rslpr) introduced by \cite{Yildiz2016}. 
The \rslpr\ considers the placement of refueling stations in order to enable connectivity between a number of origin-destination (\od) pairs.
Range-limited vehicles transport demand between \od\ pairs and use refueling stations to complete their journeys.
The problem has a routing aspect as the connecting paths are allowed to deviate to a given degree from the shortest possible path.
This yields a model that is flexible in regards to vehicle routes, but can still be solved for practical instances \citep{Arslan2019, Goepfert2019}.
However, the formulation does not reflect the capacity required at refueling stations as it allows the assignment of an unlimited amount of \od flows to any particular location.
In reality, the capacity of a station may be constrained by a variety of factors such as available yard space, the number of installed refueling ports or (in the case of \evs) the total available power supply provided by the grid.

To address this shortcoming, we introduce a capacitated version of the \rslpr, abbreviated by \rslprc, which extends the problem by capacity constraints at refueling stations.
In this version of the problem, demand flows between \od\ pairs are routed through a set of refueling stations and cannot exceed pre-defined capacities at individual stations.
Thus, the extended model allows to avert congestion and delays at upcoming refueling stations when planning their locations.
This adds an important feature to the \rslpr while maintaining its routing flexibility.

To solve the \rslprc, we devise two accompanying optimization methods.
The first is an extension of the branch-and-cut approach by \cite{Arslan2019}.
The second is a branch-cut-and-price algorithm based on variables associated with feasible \od routes.
We evaluate our methods on instances from the literature as well as a newly constructed network and compare which algorithm performs better depending on the instance parameters.
Furthermore, we modify the integer separation routine by \cite{Arslan2019} and demonstrate associated runtime improvements on uncapacitated instances.

The paper is structured as follows.
First we outline the context of related work in Section~\ref{sec:related-work}.
In Section~\ref{sec:model} we give a mathematical formulation of the \rslprc.
Then we detail the optimization algorithms to solve the \rslprc in Section~\ref{sec:solution}.
In Section~\ref{sec:experiments} we evaluate the presented methods in numerical experiments and discuss our results.


\section{Related Work}
\label{sec:related-work}

The optimal placement of refueling units in a road network so as to support \afvs has gained considerable research interest in recent years.
In the survey paper by \citet{kchaou2021charging}, the author thoroughly reviews the rapidly expanding body of work on refueling station location models.
As any attempt to match their level of detail is futile, we focus here only on the works most closely related to ours.

\citet{Kuby2005} introduce the Flow Refueling Location Model (\frlm), which places a number of refueling stations in a network optimally so as to maximize coverage of specified \od pairs w.r.t.\ range limitations of \afvs.
An \od pair is covered if sufficient refueling stations are available such that an \afv can transit from origin to destination along a unique \od path without running out of fuel.
They solve the model as a mixed-integer program by pre-generating all combinations of stations that cover individual \od pairs.
While they consider only a single path per \od pair, the number of combinations to generate can still increase exponentially with the number of nodes along the path \citep{Capar2012}.

In order to alleviate this issue, \cite{Capar2013} propose a model that formulates the \frlm in terms of covering all arcs along the \od path.
Their formulation significantly reduces the necessary preprocessing and solution time.

\cite{MirHassani2013} reformulate the \frlm based on an \emph{expanded network graph}, whose arcs represent range-feasible segments of the \od paths in the underlying network.
More generally, an expanded network graph can be used as an explicit representation of all relevant connections in the base network.
Many subsequent works apply this concept, albeit sometimes only implicitly.

\cite{Upchurch2009} develop a capacitated version of the \frlm in which the number of vehicles that can be refueled at a particular station is limited by the station capacity.

\cite{Kim2012} augment the \frlm by a set of \emph{deviation paths} for each \od pair, resulting in the Deviation Flow Refueling Location Model (\dfrlm).
Deviation paths may be longer than the shortest path up to some tolerance and a demand is (at least partially) covered if any deviation path is covered.
Since the number of deviation paths grows exponentially with increasing deviation from the shortest path length, pre-generating all deviation paths quickly becomes intractable.
A heuristic method to solve the \dfrlm is proposed by \cite{Kim2013}. Their algorithm greedily adds stations to the network, which is represented by an expanded network graph.

\cite{Hosseini2017b} extend the model of \cite{MirHassani2013} by fuel consumption that is proportional to distance and corresponding station refueling capacities.
They develop a heuristic based on Lagrangean relaxation to solve the model.
A capacitated version of the \dfrlm based on the compact formulation of \cite{Capar2013} is devised by \cite{Hosseini2017}, who also propose a heuristic algorithm to solve it.

\cite{Yildiz2016} combine the ideas of deviation paths and path segments to obtain a formulation without the requirement to pre-compute paths, which they call Refueling Station Location Problem with Routing (\rslpr).
Demands are covered by routing vehicles along any path whose length deviates from the shortest path length no more than a specified bound.
To solve the \rslpr they propose a branch-and-price method that prices out variables corresponding to arcs in an expanded network graph.

More efficient branch-and-cut approaches to the \rslpr are developed by \cite{Arslan2019,Goepfert2019}.
The authors formulate the \rslpr without any arc variables based on inequalities associated with \od pair separating station sets.
Violated inequalities, which can be identified by combinatorial algorithms applied to an expanded network graph, are added iteratively to an LP relaxation.
The resulting method greatly improves the running time of prior approaches.

\cite{Kinay2021} consider a variant of the problem that removes the path deviation bound and instead minimizes a weighted sum of station cost and deviation cost to cover all \od pairs.
This allows them to devise an algorithm based on Benders' decomposition, where the subproblem is solvable in polynomial time.


In this paper we consider a capacitated version of the \rslpr that limits the number of \od pairs supported by individual stations.
Due to the introduction of station capacities, the assignments of \od paths to every \od pair are no longer independent.
This means that the subproblem can be expressed by a 0--1 multi-commodity network flow with an additional transit time constraint.
In order to solve the model, we devise, implement and test both a branch-and-cut and a branch-cut-and-price algorithm.


\cite{Hellsten2021} study the transit time constrained fixed-charge multi-commodity network design problem.
Mathematically, this problem differs from the capacitated \rslpr only in the objective function and the missing integer constraint on the path flows.
In other words, contrary to \cite{Hellsten2021} we do not allow the flow of an individual \od pair to split over several paths, which increases the difficulty of solving the problem.


\section{Capacitated Refueling Station Problem with Routing}
\label{sec:model}

In this section we formulate an extension of the \rslpr with station capacities.
We refer to the resulting problem as Capacitated Refueling Station Location Problem with Routing (\rslprc).
In Section~\ref{sec:math-model} we state the problem mathematically and in Section~\ref{sec:model-extension} we discuss variations of the model that we omit from the main problem statement for the sake of simplicity.
We present two basic integer programming formulations of the problem in Section~\ref{sec:cut-formulation} and Section~\ref{sec:path-formulation}.
The former is based on \od pair separating node cuts and the latter is based on variables associated with \od paths.
We augment the basic formulations with further valid inequalities in Section~\ref{sec:valid-inequalities}.

\subsection{Problem Statement}
\label{sec:math-model}

The goal of the \rslprc is to determine an optimal set of refueling stations such that a limited-range vehicle can perform a given set of origin-destination (\od) trips without running out of fuel.
Accordingly, let $\rangemax$ denote the \emph{maximum range} that the vehicle can travel before it needs to refuel at a refueling station.
In addition, we assume the vehicle starts an \od trip from the origin with an initially available range denoted by $\rangeorig$ and needs to arrive at the destination with remaining range denoted by $\rangedest$.
In order to be able to concatenate \od trips with one another, we require that $\rangeorig \leq \rangedest$, i.e.\ the remaining range at the end of each trip must cover the starting range of a potential subsequent trip.
A commonly used assumption is $\rangeorig = \rangedest = \rangemax/2$, the so-called ``half-capacity requirement''.

\begin{table}[!ht]
\center
\caption
{Summary of the main notation used in the paper\label{tab:notation}.}
{
\begin{tabular}{rl}
\toprule
\textbf{Symbol} & \textbf{Description} \\
\midrule
$\rangemax \in \R$ & Maximum vehicle range \\
$\rangeorig \in \R$ & Available range at origins \\
$\rangedest \in \R$ & Required range at destinations \\
\midrule
$\G = (\V,\A)$ & Directed refueling network graph \\
$\T \subset \V$ & Terminal sites (origin/destination nodes) \\
$\stations \subset \V$ & Refueling stations (transit nodes) \\
$\ell_a \geq 0 $ & Distance of arc $a \in \A$ \\
$\tau_a \geq 0$ & Transit time of arc $a \in \A$ \\
\midrule
$q \in \odpairs$ & \od pair indexes \\
$s_q, t_q$ & Origin, destination node for \od\ pair $q$ \\
$f_q > 0$ & Flow demand for \od\ pair $q$ \\
$u_q > 0$ & Transit time upper bound for \od\ pair $q$ \\
\midrule
$\P_q$ & Set of time-feasible paths for \od\ pair $q$ \\
$\P = \bigcup_{q \in \odpairs} \P_q$ & Set of all time-feasible paths \\
$\delta_v^P \in \{0,1\}$ & Indicator for path $P$ using refueling station $v$\\
$\G_q = (\V_q, \A_q)$ & Time-feasible subgraph for \od\ pair $q$ \\
$\Gamma_q$ & Time-separators for \od pair $q$ \\
\midrule
$c_v \in \mathbb{N}^0$ & Cost to build refueling unit at $v \in \stations$ \\
$\capa_v \in \mathbb{N}^0$ & Refueling capacity at $v \in \stations$ \\
\midrule
$x_v \in \{0,1\}$ & Build decision at station $v \in \stations$\\
$y_P \in \{0,1\}$ & Coverage variable for path $P \in \P$\\
$z_v^q \in \{0, 1\}$ & Usage of station $v$ by \od\ pair $q$ \\
\midrule
$\sigma_q \geq 0$ & Dual variables for coverage \eqref{eq:path-master-coverage} \\
$\mu_v \geq 0$ & Dual variables for capacity \eqref{eq:path-master-capacity} \\
$\pi_v^q \geq 0$ & Dual variables for strong linking \eqref{eq:path-master-strong-link} \\
$\lambda^j_v \geq 0$ & Dual variables for LCIs \eqref{eq:path-master-lci} \\
$\beta_v \geq 0$ & Dual variables for station bounds \eqref{eq:path-station-bound} \\
\bottomrule
\end{tabular}
}
\end{table}

Similar to \cite{Arslan2019} we adopt an \emph{expanded network} view of the problem, which means we abstract the relevant network of refueling stations from the underlying road network.
More precisely, we consider a directed network graph $\G = (\V, \A)$ that consists of two types of nodes, the \emph{terminal sites} $\terminals \subset \V$ and the \emph{refueling stations} $\stations \subset \V$.
Clearly, for any pair of nodes $u, v \in \V$ there is a fastest connection between $u$ and $v$ in the underlying road network.
We denote by $\tau_{uv} \geq 0$ the \emph{transit time} of that connection and by $\ell_{uv} \geq 0$ its \emph{distance}.

We assume that the arcs of the network graph $\G = (\V, \A)$ represent precisely all fastest connections that satisfy the range constraints of the vehicle, as outlined in Definition~\ref{def:range-feasible} and Assumption~\ref{assump:expanded-network}.

\begin{defn}
\label{def:range-feasible}
A connection in the underlying road network between two nodes $u, v \in \V$  is considered \emph{range-feasible} if its distance $\ell_{uv}$ satisfies the following range constraint:
\begin{align*}
\ell_{uv} \leq
\begin{cases}
\rangemax & \text{if } u \in \stations \text{ and } v \in \stations \\
\rangemax - \rangedest & \text{if } u \in \stations \text{ and } v \in \terminals \\
\rangeorig & \text{if } u \in \terminals \text{ and } v \in \stations \\
\rangeorig - \rangedest & \text{if } u \in \terminals \text{ and } v \in \terminals.
\end{cases}
\end{align*}
Note that, if we assume $\rangeorig \leq \rangedest $, then no direct connections between terminal sites are considered range-feasible, which means an intermediate refueling stop is mandatory for any \od trip between terminal sites.
\end{defn}

\begin{assump}[Expanded network graph]
\label{assump:expanded-network}
In the directed network graph $\G = (\V, \A)$ with $\V = \stations \cup \terminals$, the set of arcs $\A$ represent all range-feasible fastest connections in the underlying road network.
\end{assump}

Let $\odpairs$ be an index set of \od pairs.
With every $q \in \odpairs$ we associate an \emph{origin} node $s_q \in \terminals$ and a \emph{destination} node $t_q \in \terminals$, a \emph{flow demand} $f_q > 0$ and an \emph{upper bound} $u_q > 0$ on the allowed transit time between $s_q$ and $t_q$.

In the \rslpr we consider for each \od pair $q$ the set of time-feasible (and range-feasible) routes from $s_q$ to $t_q$.
Any such route is represented by a path in the network graph whose associated transit time satisfies an upper bound.

\begin{defn}[Time-feasible path]
For $q \in \odpairs$ a path $P$ in $\G$ from $s_q$ to $t_q$ is called \emph{time-feasible} if
\begin{align}
\sum_{a \in P} \tau_a \leq u_q, \label{eq:time-feasible-path}
\end{align}
i.e.\ its associated transit time is at most the upper bound for $q$.
Let $\P_q$ denote the set of all time-feasible paths for $q \in \odpairs$ and put $\P = \bigcup_{q \in \odpairs} \P_q$.
\end{defn}

It is convenient to define for each $q \in \odpairs$ an \od pair subgraph that contains all nodes and arcs which appear in at least one time-feasible path for \od pair $q$.
Performing computations on the \od pair subgraphs instead of $\G$ has several benefits for optimization and is the basis for the algorithms presented in Section~\ref{sec:solution}. 

\begin{defn}[\od\ pair subgraph] Let $\hat\tau_{ij}$ denote the shortest path distance from vertex $i$ to $j$ within $\G$ w.r.t.\ the transit time function $\tau$. In particular, we have $\hat\tau_{ij} = \tau_{ij}$ for $ij \in \A$.
For each $q \in \odpairs$, we define the subgraph $\G_q = (\V_q, \A_q)$ via
\begin{align*}
\V_q & = \{i \in \V \mid \hat\tau_{s_q i} + \hat\tau_{it_q} \leq u_q\}, \\
\A_q & = \{ ij \in \A \mid \hat\tau_{s_qi} + \tau_{ij} + \hat\tau_{jt_q} \leq u_q \}.
\end{align*}
\end{defn}

\paragraph{Capacitated Refueling Station Location Problem with Routing}
For any subset of refueling stations $S \subseteq \stations$, let $\G_q(S) \subseteq \G_q$ denote the subgraph obtained from $\G_q$ by removing all refueling stations not contained in $S$.
With each refueling station $v \in \stations$ we associate a build cost $c_v \geq 0$ and a station capacity $\capa_v \geq 0$.
The \emph{Capacitated Refueling Station Location Problem with Routing} (\rslprc) is to determine a minimal cost set of refueling stations $S \subseteq \stations$ such that for all \od pairs there are time-feasible paths in $\G(S)$ whose combined flow demands meet the station capacity constraints:
\begin{align*}
\min_{\substack{S \subseteq \stations, \\ P_q \in \P_q}} \quad & \sum_{v \in S} c_v \tag{\rslprc} \label{eq:problem} \\
\text{s.t.} \quad
& P_q \subseteq \G_q(S) \qquad \forall q \in \odpairs \\
& \sum_{\{q \in \odpairs \mid v \in P_q\}} f_q \leq \capa_v \qquad \forall v \in \stations.
\end{align*}

\subsection{Model Variations}
\label{sec:model-extension}

Below we discuss variations of the problem parameters that make the model applicable to a wider range of scenarios.

\paragraph{Transit Time Function}
The transit time $\tau$ may be chosen as any function that is positive definite and satisfies the triangle inequality, also called a \emph{quasimetric}.
This enables to model a variety of different constraints on the \od paths.
For instance, besides pure driving time, $\tau$ may include the necessary break time spent on charging.
Alternatively, the choice $\tau_{uv} = \ell_{uv}$ relates transit time directly to distance, which is the approach taken by prior works.

\paragraph{Vehicle Types}

Different vehicle types may vary w.r.t.\ their associated ranges, transit times or even transit distances (e.g.\ in the case of restricted roads).
In principle, the model can assign a different vehicle type to each \od pair by defining the relevant \od pair network graphs $\G_q$ separately.

\paragraph{Existing Refueling Stations}

Existing refueling stations, i.e.\ stations that are pre-defined, can be modeled by setting the corresponding cost to $c_v = 0$.

\paragraph{Objective Function}

We stated the \rslprc in terms of minimizing the cost of new refueling stations.
Another common variant is to maximize the enabled flow demand under a given cost budget $B$, stated as follows:
\begin{align}
\max_{\substack{Q \subseteq \odpairs, \\ S \subseteq \stations, \\ P_q \in \P_q}} \quad & \sum_{q \in Q} f_q \label{eq:problem-demand} \\
\text{s.t.} \quad
& P_q \subseteq \G_q(S) \qquad \forall q \in Q \nonumber \\
& \sum_{\{q \in Q \mid v \in P_q\}} f_q \leq \capa_v \qquad \forall v \in \stations \nonumber \\
& \sum_{v \in S} c_v \leq B. \nonumber
\end{align}
We note that the methods presented in this paper can be adapted so as to solve \eqref{eq:problem-demand} instead of \eqref{eq:problem} by adding coverage variables for \od pairs and the budget constraint to the otherwise analogous approach.

\paragraph{Infeasible Demand Pairs}

For some \od pairs $q$ there may not exist any time-feasible path, i.e.\ $\P_q = \emptyset$.
Such \od pairs can be removed a-priori from the problem statement.

\subsection{Cut Formulation}
\label{sec:cut-formulation}

In this section we extend the formulation by \cite{Arslan2019, Goepfert2019} with station capacity constraints.

To this end, we similarly introduce the dual notion of time-feasible paths, the \emph{time-separators}, i.e.\ station sets that intersect all time-feasible paths for a given \od pair.
In any time-separator for \od pair $q \in \odpairs$ at least one refueling station must be active in order to establish a time-feasible path for~$q$ that traverses only active refueling stations.

\begin{defn}[Time-separator]
A set of refueling stations $S \subseteq \stations$ is a \emph{time-separator} for \od pair~$q \in \odpairs$ if it intersects all its time-feasible paths, i.e.\
\begin{align*}
    S \cap P \neq \emptyset, \quad \forall P \in \P_q.
\end{align*}
We denote by
\begin{align*}
\Gamma_q = \{ S \subseteq \stations \mid S \text{ time-separator for } q \}
\end{align*}
the set of all time-separators for \od\ pair $q$.
Naturally, any \emph{separator} for $q \in \odpairs$, i.e.\ any set of stations that intersects \emph{all} paths from $s_q$ to $t_q$ in $\G_q$ (instead of only all time-feasible paths) is also a time-separator.
A (time-)separator $S$ for $q$ is called \emph{minimal} if no proper subset of $S$ is also a (time-)separator for $q$.
\end{defn}

\begin{rem}
\cite{Arslan2019, Goepfert2019} refer to time-separators as \emph{$q$-node-cuts} and \emph{\textsc{od}-cuts} respectively.
The term time-separator allows a systematic distinction from the subset of separators.
\end{rem}

We define two types of binary variables for the integer program. First, we have a local station variable $z_v^q$ equal to 1 if \od\ pair $q$ routes \afvs through refueling station $v$ and 0 otherwise. Likewise, we define a global station variable $x_v$ equal to 1 if any \od\ pair uses refueling station $v$. We can think of $x$ variables as decision variables for refueling stations in the global network graph $\G$ and $z$ variables as decision variables for stations in the \od\ pair subgraphs $\G_q$.

This results in the following integer linear program for \rslprc, which we refer to as the \emph{cut formulation} \eqref{eq:cut}:

\begin{align}
\min \quad & \sum_{v \in \stations} c_v x_v \tag{CF} \label{eq:cut} \\
\text{s.t.} \quad
& \sum_{v \in S} z^q_v \geq 1 & \forall q \in \odpairs, \; S \in \Gamma_q \label{eq:cut-coverage} \\
& \sum_{q \in \odpairs} f_q z_v^q \leq \kappa_v x_v & \forall v \in \stations \label{eq:cut-capacity}\\
& x_v,z_v^q \in \{0, 1\} & \forall v \in \stations, \; q \in \odpairs \label{eq:cut-int}
\end{align}

Constraint \eqref{eq:cut-coverage} ensures there exists a time-feasible path for every \od pair $q$ by covering each time-separator with a refueling station.
Constraint \eqref{eq:cut-capacity} ensures the amount of demand that flows through a refueling station does not exceed its capacity.
Lastly, constraint \eqref{eq:cut-int} restricts local and global station variables to binary values. 

The cut formulation contains $\mathcal{O}(|\stations||\odpairs|)$ binary variables, which is a polynomial number, but can still become quite large.
In practice, the $z$ variables can be restricted to the subgraphs $\G_q$, which is manageable as long as the subgraphs are not too large.

Further, the cut formulation contains a large number of constraints, since $|\Gamma_q|$ grows exponentially.
We typically solve integer programs with exponentially many constraints using \emph{branch-and-cut} algorithms.
Rather than specifying every constraint (\ref{eq:cut-coverage}) at the start, we can generate and add them to the model dynamically using delayed row generation. 

\subsection{Path Formulation}
\label{sec:path-formulation}

Alternatively, we present a path-based formulation (\ref{eq:path}) for the \rslprc that is amenable to delayed column generation.
Our formulation is closely related to the work of \cite{Hellsten2021}.
Again, we require two types of binary variables we denote $x_v$ and $y_P$.
The $x_v$ variables indicate active refueling stations as in the cut formulation, whereas the $y_P$ variables are equal to 1 if \od\ pair $q$ uses path $P \in \P_q$ and 0 otherwise.
For convenience, we also define the vertex-path incidence indicators
\begin{align*}
\delta^P_v = \begin{cases} 1 & \text{if } v \in P \\ 0 & \text{else.} \end{cases}
\end{align*}
The \emph{path formulation} of the \rslprc reads
\begin{align}
\min \quad & \sum_{v \in \stations} c_v x_v \tag{PF} \label{eq:path} \\
\text{s.t.} \quad
& \sum_{P \in \P_q} y_P = 1 & \forall q \in \odpairs \label{eq:path-coverage} \\
& \sum_{q \in \odpairs} \sum_{P \in \P_q} f_q\delta_v^P y_P \leq \kappa_v x_v & \forall v \in \stations \label{eq:path-capacity} \\
& x_v, y_P \in \{0, 1\} & \forall v \in \stations, \; \forall P \in \P \label{eq:path-int}
\end{align}

Constraint (\ref{eq:path-coverage}) ensures that exactly one path from the set of all time-feasible paths $\P_q$ for \od\ pair $q$ is selected.
Constraint (\ref{eq:path-capacity}) ensures that the total volume of demand flowing through station~$v$ from all \od\ pairs is less than the capacity $\kappa_v$ of that station if $x_v = 1$ (and 0 otherwise).
Finally, constraint (\ref{eq:path-int}) ensures the path and station variables are binary.

While the cut formulation has a polynomial number of variables and an exponential number of constraints, the path formulation has an exponential number of variables and a polynomial number of constraints. We solve integer programs with exponentially many variables using \emph{branch-and-price} algorithms. While the branch-and-cut generates constraints dynamically, the branch-and-price algorithm we use to solve (\ref{eq:path}) generates variables dynamically using column generation.

\subsection{Valid Inequalities}
\label{sec:valid-inequalities}

The basic formulations introduced above have weak linear programming relaxations, as exhibited by closely related network design problems \citep{gendron1999multicommodity}.
Prior works have proposed a number of valid inequalities to strengthen these formulations.
Such inequalities are valid constraints for all integer solutions but invalid for some fractional solutions.
Since valid inequalities separate integer and fractional solutions, including them in the formulation reduces the domain of the feasible set and strengthens LP relaxations.
Two families of valid inequalities for \eqref{eq:cut} are also known as \emph{strong linking inequalities} and \emph{lifted cover inequalities}. 

We present the valid inequalities in terms of the cut formulation, but they can easily be translated to the path formulation using the following identity between local station variables and path variables:
\begin{align*}
z_v^q = \sum_{P \in \P_q} \delta_v^P y_P 
\end{align*}

\begin{defn}[Strong linking inequality]

These are common cutting planes for network design problems. Strong linking inequalities perform similar roles compared to capacity constraints (\ref{eq:cut-capacity}). 
If any \od\ pair $q$ uses refueling station $v$ with local station variable $z_v^q = 1$, then the global refueling station variable $x_v$ must also equal 1.
For the cut formulation (\ref{eq:cut}), we define the strong linking inequalities:
\begin{align*}
    z_v^q \leq x_v, \; \forall v \in \stations, \; q \in \odpairs. 
\end{align*}
Since there is only a polynomial number of strong linking inequalities they may be included in the initial problem formulation.
This is our default approach for \eqref{eq:cut} and aligns with most of the related prior work.
Because the inclusion of all strong linking inequalities can however make the problem highly degenerate \cite{Hellsten2021}, in \eqref{eq:path} they are added in a lazy manner by inspection for each station \citep{chouman2017commodity}.

\end{defn}

\begin{defn}[Lifted cover inequality (LCI)]

Lifted cover inequalities are valid inequalities derived from knapsack constraints like the capacity constraints (\ref{eq:cut-capacity}).
For any given station $v$, a set $C_v \subseteq \odpairs$ is called a \emph{cover} if $\sum_{q \in C_v} f_q > \kappa_v$, i.e.\ a set of \od\ pairs whose total flow volume exceeds the capacity of a refueling station.
A cover is minimal if no proper subset of $C_v$ is also a cover. For any minimal cover $C_v$, the inequality
\begin{equation*}
    \sum_{q \in C_v} z_v^q \leq |C_v| - 1
\end{equation*}
\noindent is called a \emph{cover inequality} and is always valid for all integer solutions of \eqref{eq:cut}.

Cover inequalities are not generally facet-defining for the knapsack polytope, but they can be strengthened by a lifting procedure.
Indeed, given any minimal cover $C_v$, there exists a facet-defining \emph{lifted cover inequality} (LCI),
\begin{equation}
    \sum_{q \in C_v} z_v^q + \sum_{q \in \odpairs \setminus C_v} \alpha_q z_v^q \leq |C_v| - 1 \label{eq:lci_simple}
\end{equation}

\noindent where $\alpha_q$ is some non-negative integer. Given a minimal cover, we can determine the values for $\alpha_q$ through a process called \emph{lifting}. Given the solution to a LP relaxation, \cite{gu1999lifted} showed that solving the separation problem exactly and finding the most violated LCI is NP-Hard. Instead of solving the separation problem exactly, we can solve it using two lightweight heuristics that perform well in practice -- coefficient independent minimal cover generation introduced by \cite{gu1998lifted} and an improved Balas' lifting procedure from \cite{letchford2019lifted}.

\end{defn}



\section{Optimization Methods}
\label{sec:solution}


In this section we present algorithms to solve the problem cut formulation \eqref{eq:cut} and the path formulation \eqref{eq:path} introduced in Section~\ref{sec:model}.
At a high level, we solve \eqref{eq:cut} using \textit{branch-and-cut} and \eqref{eq:path} using \textit{branch-cut-and-price.} 
In these \textit{branch-and-bound} methods, we construct a search tree that starts with solving an initial LP relaxation as the root node.
For the cut formulation, the root node corresponds to a linear program with only an initial set of coverage constraints.
Likewise, for the path formulation, the root node corresponds to a linear relaxation of (\ref{eq:path}) initialized with a subset of path variables $\P' \subset \P$ that allow a feasible solution.

Subsequently, we enter the main steps of the branch-price-and-cut algorithm summarized as follows:
\begin{enumerate}[i.,itemsep=2pt,parsep=0pt]
\item Select a node from the search tree according to some selection strategy.
\item Solve the node's LP relaxation including relevant branching constraints.
\item Run the pricing algorithm (if applicable) to add columns (variables) and run separation algorithms to add rows (constraints) to all nodes in the search tree and resolve.
\item Branch and bound the solution space given the fractional LP solution.
\item Run heuristic algorithms given the current LP solution to search for integer solutions.
\end{enumerate}

The implementation of branch-cut-and-price methods depends on a number of design choices.
In the following sections, we describe how we initialize the cut and path formulations of the problem and present pricing and separation algorithms that we implement.
Additionally, we describe the branching procedure for the path formulation and the custom heuristics used to obtain primal solutions.

We construct the branch-and-bound tree with CPLEX \citep{CPLEX20.1} in the case of \eqref{eq:cut} and with SCIP \citep{SCIP8.0} in the case of \eqref{eq:path}.
In both algorithms, columns and rows are added globally with the default age limits provided by CPLEX and SCIP respectively. 
For \eqref{eq:cut}, we add time separators for integer solutions as lazy constraints and add all other separators including time separators for fractional solutions, LCIs, and strong linking inequalities as additional cutting planes.
For \eqref{eq:path}, we implement separators using SCIP's constraint handlers and their constraint enforcement callbacks. 

\subsection{Initialization of constraints and variables}

We include all capacity constraints \eqref{eq:cut-capacity} and \eqref{eq:path-capacity} in the initial LP relaxation.
The strong linking constraints are included initially for \eqref{eq:cut}, while we generate them dynamically for \eqref{eq:path}.

\subsubsection{Initialization of time-separator constraints}

In principle we could start with an empty set of time-separator constraints to solve an initial relaxation of \eqref{eq:cut}.
However, in branch-and-cut algorithms it is typically helpful to include a suitable subset of constraints in the root relaxation.
This improves the initial bound and can significantly reduce the separation effort in later iterations.

To this end, we adopt the following approach commonly applied to connectivity problems.
For vertices $u,v \in \V_q$, let~$\Delta(u,v)$ denote the minimum \emph{hop distance} of $v$ from $u$ in $\G_q$.
Now, for each $q \in \odpairs$ and all $1 \leq k \leq \Delta(s_q, t_q) - 1$ we define
\begin{align*}
S_q^k = \{ v \in \V_q \mid \Delta(s_q, v) = k \text{ and } \Delta(v, t_q) = \Delta(s_q,t_q) - k \}.
\end{align*}
It is easy to see that each $S_q^k$ is a minimal separator for $q$ and thus it is also a time-separator, albeit not necessarily a minimal one.
We add time-separator constraints for all $S_q^k$ to the initial LP relaxation of \eqref{eq:cut}.
The sets $S^k_q$ can be determined by basic graph search algorithms.

\subsubsection{Initialization of path variables}

We initialize $\P'$ with an auxiliary rejection variable for each \od pair and equip them with a large cost in the objective.
The auxiliary rejection variables may be selected if no other time-feasible path is found.
From the model perspective they could represent the fallback option of diesel vehicles that transit directly from the origin to the destination site. 
The use of auxiliary variables is known to generate a more useful set of columns for the root node \citep{Hellsten2021}.
Moreover, their inclusion resolves infeasible nodes that may arise as a result of branching decisions.

\subsection{Delayed row generation}

At each node of the search tree we add cutting planes corresponding to inequalities that are violated by the current LP solution.
Cutting planes allow to cut off infeasible and certain fractional solutions from the set of feasible integer solutions, thereby accelerating the branch-and-bound process.
Finding these cutting planes requires solving \emph{separation problems} for the respective classes of inequalities.
In our approach, we generate time-separator inequalities for the cut formulation and additional lifted cover inequalities for both formulations.

\subsubsection{Time-separator constraints in the cut formulation}
\label{sec:int-sep}

Suppose $(\bar x, \bar z)$ is a solution to the LP relaxation at the given node of the branch-and-bound tree.
We need to check if $(\bar x, \bar z)$ satisfies all time-separator constraints, and otherwise determine a time-separator such that constraint \eqref{eq:cut-coverage} is violated.
In other words, we need to find a time-separator that is minimum w.r.t.\ weights $\bar z^q_v$.
Unfortunately, this is an NP-hard problem in general \citep{Baier2010}.
However, if $\bar z^q$ is binary, then the separation problem can be solved efficiently.
This suffices to implement a correct branch-and-cut algorithm.
We present below our approaches for separation of integer and fractional solutions, which in the integer case improves upon the approach of \cite{Arslan2019}, as demonstrated by the comparison in Section~\ref{sec:int-sep-results}.

\paragraph{Separating integer solutions}

In this case, we assume that for $q \in \odpairs$, the vector $\bar z^q$ is binary.
We show how to quickly find violated time-separators of small size.

Consider the \od pair subgraph $\G_q(\bar z) = (\V_q, \A_q(\bar z))$ that includes only those arcs whose tails are active nodes, as defined in
\begin{align*}
    \A_q(\bar z) &= \{uv \in \A_q \mid u = s_q \text{ or } \bar z^{q}_u = 1\}.
\end{align*}

Recall that $\hat \tau_{ij}$ denotes the shortest path distance from $i$ to $j$ in $\G$ w.r.t.\ the arc transit times~$\tau$.
Similarly, we define $\hat \tau_{ij}(\bar z)$ as the shortest path distance from $i$ to $j$ in $\G_q(\bar z)$ w.r.t.\ $\tau$.
In other words, $\hat \tau_{ij}(\bar z)$ represents the shortest transit time if only the active refueling stations can be used.
By convention, we set $\hat \tau_{ij}(\bar z) = \infty$ if $\G_q(\bar z)$ contains no path from $i$ to $j$.

For $\bar z$ and $q \in \odpairs$ consider the sets
\begin{align}
S_q(\bar z) = \{ v \in \V_q \cap \stations \mid \bar z^q_v = 0 \text{ and } \hat \tau_{s_q,v}(\bar z) + \hat \tau_{v,t_q} \leq u_q \}. \label{eq:time-feasible-extension}
\end{align}
The inequality condition in \eqref{eq:time-feasible-extension} states that, for every $v \in S_q(\bar z)$, there exists a path in $\G_q(\bar z)$ from~$s_q$ to $v$ that is shortest w.r.t.\ $\tau$ and, moreover, this path can be extended to a time-feasible path within~$\G_q$.
If $\G_q(\bar z)$ contains no time-feasible path, then the extension requirement guarantees that $S_q(\bar z)$ is a time-separator, while in general it might not be a separator.
See Figure~\ref{fig:time-separator} for an example of such a time-separator.

\begin{prop}
Suppose $\tau_{s_q,t_q}(\bar z) > u_q$, that means $\G_q(\bar z)$ contains no time-feasible path.
Then the set $S_q(\bar z)$ defined in \eqref{eq:time-feasible-extension} is a time-separator for $q \in \odpairs$.
In addition, the coverage constraint \eqref{eq:cut-coverage} w.r.t.\ $\bar z$ and $S_q(\bar z)$ is violated.
\end{prop}

\begin{proof}
Let $P \in \P_q$ be a time-feasible path.
Since $P$ is not contained in $\G_q(\bar z)$, there exists a station node $v \in P$ such that $\bar z^q_v = 0$.
Let $v$ be the first such node on $P$ in the order from $s_q$ to $t_q$.
This implies that the subpath of $P$ from $s_q$ to $v$ is contained in $\G_q(\bar z)$.
Since $P$ is time-feasible, it follows that $v \in S_q(\bar z)$.
Hence, we have shown that $S_q(\bar z)$ is a time-separator for $q$.
The second assertion immediately follows from the fact that $\bar z^q_v = 0$ for all $v \in S_q(\bar z)$.
\end{proof}

The sets $S_q(\bar z)$ can be computed efficiently by applying Dijkstra's algorithm once to each of $\G_q(\bar z)$ and $\G_q$ and subsequent graph search.
This allows to quickly determine violated time-separators of relatively small size.
In order to improve the strength of the associated inequalities further, we may iteratively remove vertices from $S_q(\bar z)$ and check with Dijkstra if the remaining set is still a time-separator.
Our approach compares favorably with that of \cite{Arslan2019}, who apply the same iterative removal process to obtain a minimal time-separator, but start with the trivial (and in general large) time-separator~$\{ v \in \V_q \cap \stations \mid \bar z^q_v = 0 \}$.

\begin{figure}[!t]
\center
{
\begin{minipage}[t]{0.49\linewidth}
\center

\begin{tikzpicture}[scale=1.4]

\tikzstyle{inactive}=[thick,draw,circle,minimum size=8pt,inner sep=0pt]
\tikzstyle{active}=[fill,thick,draw,circle,minimum size=8pt,inner sep=0pt]

\tikzstyle{def}=[thick,->,>=stealth]

\node[inactive,label=left:$s$,rectangle] (s) at (0,0) {};
\node[inactive,label=above:$a$] (a) at (1,1.2) {};
\node[inactive,label=above:$b$] (b) at (1.4,-1) {};
\node[inactive,label=above:$c$] (c) at (2.0,0.2) {};
\node[inactive,label=above:$d$] (d) at (3,1) {};
\node[inactive,label=right:$t$,rectangle] (t) at (3.5,-0.5) {};

\draw (s) edge[def] node[below] {2} (a);
\draw (s) edge[def] node[below] {2} (b);
\draw (a) edge[def] node[left] {1} (c);
\draw (b) edge[def] node[left] {1} (c);
\draw (a) edge[def] node[above] {1} (d);
\draw (d) edge[def] node[right] {2} (t);
\draw (c) edge[def] node[above] {3} (t);
\draw (b) edge[def] node[above] {3} (t);

\end{tikzpicture}
\end{minipage}
\hfill
\begin{minipage}[t]{0.49\linewidth}
\center

\begin{tikzpicture}[scale=1.4]

\tikzstyle{inactive}=[thick,draw,circle,minimum size=8pt,inner sep=0pt]
\tikzstyle{active}=[fill,thick,draw,circle,minimum size=8pt,inner sep=0pt]

\tikzstyle{def}=[thick,->,>=stealth]


\node[active,label=left:$s$,rectangle] (s) at (0,0) {};
\node[active,label=above:$a$] (a) at (1,1.2) {};
\node[inactive,label=above:$b$] (b) at (1.4,-1) {};
\node[inactive,label=above:$c$] (c) at (2.0,0.2) {};
\node[inactive,label=above:$d$] (d) at (3,1) {};
\node[active,label=right:$t$,rectangle] (t) at (3.5,-0.5) {};


\draw (s) edge[def] node[below] {2} (a);
\draw (s) edge[def] node[below] {2} (b);
\draw (a) edge[def] node[left] {1} (c);
\draw (a) edge[def] node[above] {1} (d);

\end{tikzpicture}
\end{minipage}
}
\caption
{Example computation of a time-separator that is not a separator\label{fig:time-separator}. Left: Example graph $\G_q$ with origin $s$, destination $t$ and transit time labels. Right: Active subgraph $\G_q(\bar z)$ where the active nodes $s,t$ and $a$ are filled. For an upper transit time bound of $u = 5$, we compute the time-separator $S_q(\bar z) = \{b,d\}$. Node $c$ is not contained in $S_q(\bar z)$, because there is no time-feasible path that includes $c$.}
\end{figure}
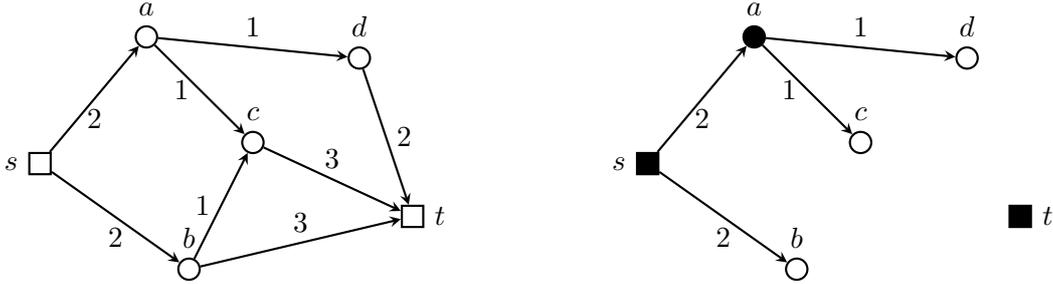

\paragraph{Separating fractional solutions}

In this case the vector $\bar z^q$ is fractional.
We apply the same heuristic approach as \cite{Arslan2019}, which is to find a separator $S$ that is minimum w.r.t.\ the weights $\bar z^q_v$.
The separator $S$ can be found efficiently via max-flow/min-cut algorithms applied to a transformation of the input graph $\G_q$ that splits every refueling node $v$ into two separate nodes for incoming and outgoing arcs and adds a link with associated capacity $\bar z^q_v$.

Similar to the integer case above, we may iteratively remove refueling stations from $S$ until it is a minimal time-separator.
Finally, if $\sum_{v \in S} \bar z_v^g < 1$, then we have identified a violated cut that cuts off the fractional solution~$(\bar x, \bar z)$ from the feasible set.

While the overall solution may contain fractional values, there may exist \od\ pairs such that all global and local refueling station variables for every station in their subgraphs have integer values.
When an \od\ pair subgraph has a locally integer solution, we may call our integer separation routine.

\subsubsection{Separation of lifted cover inequalities}


Suppose $(\bar x, \bar z)$ is a solution to the LP relaxation. First, we identify all saturated refueling stations.
A refueling station is saturated if the current routing assignment completely utilizes all its capacity resources, i.e.\ it holds that $\sum_{q\in \odpairs} f_q \bar z_v^q \geq \kappa_v\bar x_v$.
Next, for every saturated refueling station, we find a minimal cover and use Balas' lifting procedure to generate the lifing coefficients for the LCI, see below for details.
Afterwards, if the current LP solution violates the LCI, then we add this constraint to (\ref{eq:cut}).

We generate a minimal cover $C_v^j$ for refueling station $v$ as follows.
First, we initialize $C_v^j = \emptyset$.
Then, we sort all $\bar z_v^q$ for $q \in \odpairs$ in non-increasing order and add them to $C_v^j$ sequentially until $\sum_{q \in C_v^j} f_q > \kappa_v$.
A cover constructed this way is not necessarily minimal, so we can remove variables from $C_v^j$ until it becomes minimal. 

To generate the LCI coefficients $\alpha_q^j$ for LCI $j$, we use Balas' lifting procedure. Let $S(r)$ denote the sum of the $r$ largest $f_q$ values over the members of the minimal cover $C_v^j$ we just constructed. We further define $S(0) = 0$. According to Balas' lifting procedure, the lifting coefficient $\alpha_q^j$ is the unique integer that satisfies $S(\alpha_q^j) \leq f_q < S(\alpha_q^j + 1)$.

For the determined cover $C_v^j$ and lifting coefficients $\alpha_q^j$, we add the corresponding LCI defined in~\eqref{eq:lci_simple}, which is violated by $(\bar x, \bar z)$.

\subsection{Pricing out path variables}
\label{sec:pricing}

In this section we derive the pricing problem to generate variables during column generation for solving \eqref{eq:path}.
We specify the linear program \eqref{eq:path-master} below as the \emph{master problem} during column generation:
\begin{align}
\min \quad & \sum_{v \in \stations} c_v x_v \tag{MP} \label{eq:path-master} \\
\text{s.t.} \quad
& \sum_{P \in \P_q} y_P = 1 & \forall q \in \odpairs \label{eq:path-master-coverage} \\
& \sum_{q \in \odpairs} \sum_{P \in \P_q} f_q\delta_v^P y_P \leq \kappa_v x_v & \forall v \in \stations \label{eq:path-master-capacity} \\
& \sum_{P \in \P_q} \delta_v^P y_P \leq x_v & \forall v \in \stations, \; \forall q \in \odpairs \label{eq:path-master-strong-link} \\
& \sum_{q \in \odpairs} \sum_{P \in \P_q} \alpha^j_q \delta^P_v y_P \leq |C^j_v| - 1 & \forall v \in \stations, \forall j \in \J_v \label{eq:path-master-lci} \\
& 0 \leq x_v \leq 1 & \forall v \in \stations \label{eq:path-station-bound}\\
& 0 \leq y_P & \forall P \in \P. \label{eq:path-non-neg}
\end{align}
The master problem is a linear relaxation of \eqref{eq:path} obtained by replacing the binary constraints~\eqref{eq:path-int} with variable bounds \eqref{eq:path-station-bound}--\eqref{eq:path-non-neg} and adding the strong linking constraints \eqref{eq:path-master-strong-link} as well as additional lifted cover inequalities \eqref{eq:path-master-lci} indexed by $j \in \J_v$ for all $v \in \stations$.
The \emph{restricted master problem} (RMP) only contains variables for a subset of paths $\P' \subseteq \P$ instead of all possible paths.
During column generation, we continue adding paths to the RMP that improve its objective value.
The problem of finding the path $P \in \P$ that most improves the objective value of the RMP is called the \emph{pricing problem}.
Since the pricing problem can be derived as a separation problem for the dual of the master problem, we formulate this dual problem as \eqref{eq:pricing} below:

\begin{align}
\max \quad & \sum_{q \in \odpairs} \sigma_q + \sum_{v \in \stations} \Big [ \beta_v + \sum_{j \in \J_v} (|C^j_v|-1)\lambda^j_v \Big ]  \tag{MP-D} \label{eq:pricing} \\
\text{s.t.} \quad 
&  \kappa_v \mu_v - \beta_v + \sum_{q \in \odpairs} \pi^q_v \leq c_v & \forall v \in \stations \nonumber \\
& \sum_{v \in \stations} \delta^P_v \Big [ f_q \mu_v + \pi_v^q + \sum_{j \in J_v} \alpha_q^{j} \lambda^j_v  \Big ] - \sigma_q \geq 0 & \forall q \in \odpairs ,\; P \in \P_q \label{eq:dual_path} \\
& \mu_v, \pi^q_v, \lambda^j_v \geq 0 & \forall v \in \stations, \; q \in \odpairs, \; j \in \J_v. \nonumber
\end{align}

Here we introduced dual variables $\sigma_q$ for coverage constraints \eqref{eq:path-master-coverage}, $\mu_v$ for capacity constraints \eqref{eq:path-master-capacity}, $\pi_v^q$ for strong linking constraints \eqref{eq:path-master-strong-link}, $\lambda^j_v$ for LCIs \eqref{eq:path-master-lci} and $\beta_v$ for upper bounds \eqref{eq:path-station-bound}, see also Table~\ref{tab:notation} for reference.
The pricing problem is the separation problem for \eqref{eq:dual_path}, i.e.\ to identify for every $q \in \odpairs$ a time-feasible path~$P \in \P_q$ that minimizes
\begin{align}
\sum_{v \in \stations} \delta^P_v \Big [ f_q \mu_v + \pi_v^q + \sum_{j \in J_v} \alpha_q^{j} \lambda^j_v  \Big ] - \sigma_q. \label{eq:path-value}
\end{align}
If the value of \eqref{eq:path-value} for $P$ is less than 0, then $P$ is added to the RMP.
We repeat the pricing loop by resolving the RMP and the pricing problem until no further paths are found, which means we computed an optimal solution to the RMP.

Path formulations for network design problems are known to suffer from degeneracy \citep{Hellsten2021}. 
This can cause a ``tailing-off effect'' where the pricer makes less progress each iteration as we approach the optimum of the RMP.
When the cost coefficients $c_v$ are integer, we mitigate the tailing-off effect by terminating pricing early, 
which proved necessary for solving \rslprc for large road networks with many \od\ pairs.

Our criteria for early termination of pricing are derived as follows.
During each pricing round we calculate a Lagrangian lower bound to the MP as the sum of the optimum of the RMP and the optima of that round's pricing problems \citep{lubbecke2005selected}.
Now we stop pricing under either of two conditions.
i.\ We stop pricing if the integer part of the current node's solution to the RMP is equal to the integer part of the global lower bound or its parent node's lower bound (if a parent exists).
ii.\ We stop pricing if the Lagrangian bound of the MP and the solution to the RMP have equal integer parts. 
In both cases, the addition of further variables will not improve the value of the integer solution, so we can branch early and reduce the time spent pricing.

The pricing problem itself can be formulated as a \emph{constrained shortest path} (CSP) problem on the subraphs $\G_q$ w.r.t.\ path lengths defined by \eqref{eq:path-value} and the transit time constraint \eqref{eq:time-feasible-path}.
The CSP problem is NP-hard in general \citep{handler1980dual}, but many practical instances can be solved efficiently using dynamic programming approaches \citep{irnich2005shortest, pugliese2013survey}.
Furthermore, simple and very fast heuristics like LARAC based on repeated application of Dijkstra's algorithm are available \citep{xiao2005gen}.
Once LARAC finds no more variables corresponding to constrained shortest paths to add or a condition for early termination is met, we run an exact CSP solver to verify no additional variables should be added.

\subsubsection{Branching for fractional path variables}

When constructing the search tree for \eqref{eq:path} we prioritize branching on station variables $x_v$ first.
Eventually, we may arrive at a node with a fractional LP solution where all station variables $x_v$ are integer.
In this case we branch on fractional path variables $y_P$ using the strategy from \cite{barnhart2000using}, which is summarized as follows:
\begin{enumerate}[i.,itemsep=2pt,parsep=0pt]
\item Take the largest flow $f_q$ among all \od\ pairs whose flow is split.
\item Determine the two paths $P, \; P' \in \P_q$ with the greatest fractional values $y_P, \; y_{P'}$.
\item Identify the refueling station $v$ where the paths diverge before reconvening at the destination site $t_q$.
\item Divide the set of arcs outgoing from $v$ as evenly as possible into two disjoint sets $A_1$ and $A_2$ such that
$A_1$ contains the arc used by path $P$ and $A_2$ contains the arc used by path $P'$. 
\item Create two child nodes, one that forbids arcs in $A_1$ for \od\ pair $q$ and the other that forbids arcs in $A_2$.
\end{enumerate}
This branching rule can be implemented by appropriately deleting arcs of $\G_q$ in the child nodes.
Then the pricing remains a CSP problem in the modified graphs.
Formally, the branching rule fixes~$y_P = 0$ in the RMP for any path $P$ that uses a forbidden arc.

\subsection{Primal heuristics}

We present a primal heuristic for \eqref{eq:cut} that, given a fractional LP solution $(\bar{x}, \bar{z})$ generates an integer solution $(x, z)$ for the \rslprc.
The algorithm is based on constrained shortest paths and summarized in Alg.~\ref{alg:primal}. 
We include pseudocode for the referenced helper methods $\textsc{getCSP}$ (Alg.~\ref{alg:getCSP}) and $\textsc{ResolveInfeasibility}$ (Alg.~\ref{alg:resolve}) in Appendix \ref{app:heuristic}.

Our primal heuristic first initializes a queue containing all \od\ pairs. 
While the queue is not empty, the algorithm will pick the next \od\ pair $q$ and call Alg.~\ref{alg:getCSP} to find a constrained shortest path (CSP) from $s_q$ and $t_q$ in the subgraph $\G_q = (\V_q, \A_q)$ that is minimal with respect to refueling station costs $l_v^q$. 
We determine these refueling station costs $l_v^q$ as follows:

\begin{enumerate}
    \item If the sum of $f_q$ and the demand of all other \od\ pairs previously assigned to refueling station $v$ exceeds the refueling station capacity $\kappa_v$, then set $l_v^q = \infty$.
    \item Otherwise, set $l_v^q = 0$ if the heuristic has already used $v$ for another \od\ pair and $l_v^q = 1 - \bar{z}_v^q$ else.
\end{enumerate}

As the heuristic iterates over \od\ pairs in the queue and determines refueling station costs, the solution to the CSP problem $P_q$ may contain a station with $l_v^q = \infty$. 
This occurs when all time-feasible paths in $\G_q$ contain refueling stations that lack capacity to accommodate $f_q$.
To resolve this infeasiblity, we remove paths from the solution until $P_q$ is feasible in Alg.~\ref{alg:resolve}.
We remove paths and their respective \od\ pairs from the solution in first-in-first-out order.
Any \od\ pairs whose paths are removed from the solution are added back to the queue of \od\ pairs so that the algorithm can generate new paths for them at future iterations. 

Note that the heuristic does not work for \eqref{eq:path}, as it may generate paths whose associated variables have not yet been added to the restricted master problem.
However, since feasibility is not an issue after initialization, we can rely on the internal heuristics used by SCIP.


\begin{algorithm}[h!]
\caption{Primal Heuristic for \eqref{eq:cut}}\label{alg:cap}
\begin{algorithmic}[1]
\Procedure{CSPHeuristic}{$\bar{x}$, $\bar{z}$}
    \State $\textsc{O-DQueue} \gets \{q, \; \forall q \in \odpairs\}$
    \State $\textsc{StationQueue}(v) \gets \{\} \; \forall v \in \stations$
    \State $P_q \gets \{\} \; \forall q \in \odpairs$
    \State $k_v, x_v, z_v^q \gets 0, \; \forall v \in \stations, q \in \odpairs$
    \While{$\textsc{O-DQueue} \; \text{is not empty}$}
        \State $q \gets \textsc{O-DQueue}.\text{dequeue()}$
        
        \State $P_q \gets \textsc{getCSP}(q, k, \bar{z})$ \Comment{(Alg.~\ref{alg:getCSP})}
        \State $\textsc{ResolveInfeasibility}(P_q, k, z, \textsc{StationQueue}, \textsc{O-DQueue})$ \Comment{(Alg.~\ref{alg:resolve})}
        
    \EndWhile
    
\ForAll{$v \in \stations$}
    \If{$k_v > 0$}
        \State $x_v \gets 1$
    \EndIf
\EndFor

\Return{$x_v, z_v^q \; \forall v \in \stations, q \in \odpairs$}
\EndProcedure
\end{algorithmic}
\label{alg:primal}
\end{algorithm}




\section{Numerical Experiments}
\label{sec:experiments}

In this section we present numerical results obtained with our implementations of the methods presented in Section~\ref{sec:solution}.
We first describe the experimental setup and the constructed problem instances and then present the results of our evaluation.

\paragraph{Implementation details}
We implement our algorithm for solving \eqref{eq:cut} with CPLEX \citep{CPLEX20.1} and for \eqref{eq:path} with SCIP \citep{SCIP8.0} via PySCIPOpt \citep{MaherMiltenbergerPedrosoRehfeldtSchwarzSerrano2016} compiled with CPLEX as an LP solver.
Both solvers are run via a Python interface.
To increase the performance of graph operations such as finding shortest paths and constrained shortest paths, we implement a \texttt{C++} extension for Python that takes advantage of the Boost Graph Library \citep{siek2002boost}. 

We implement integer separation routines using lazy constraint callbacks and fractional separation routines using user cut callbacks.
We reduce the overall time spent on separation by limiting the number of separation rounds per branch-and-bound node.
For \eqref{eq:path} we perform only one separation in nodes other than the root.
For \eqref{eq:cut} we perform only one fractional separation round in any node.

Unless specified otherwise we use the default branching rules of CPLEX and SCIP.
All numerical experiments are carried out on an Intel Core i7-7700 CPU machine with 16GB of RAM. 

\subsection{Problem instances}

We construct two different types of problem instances to evaluate our algorithms.
The first we obtain from data of the California road network that was also used by \cite{Arslan2019, Yildiz2016}.
The second we generate ourselves from public road data in Europe and randomly sampled locations.
In order to obtain instances of the \rslprc, we compute expanded network graphs for both types of data.
The Europe network graph is an order of magnitude larger and hence gives rise to more challenging problem instances. 
We summarize properties of both graphs in Table \ref{tab:instances}. Their nodes are plotted in Figure~\ref{fig:maps}.

\begin{figure*}[!ht]
\centering
{
\subfloat[California]{\includegraphics[width=.46\linewidth]{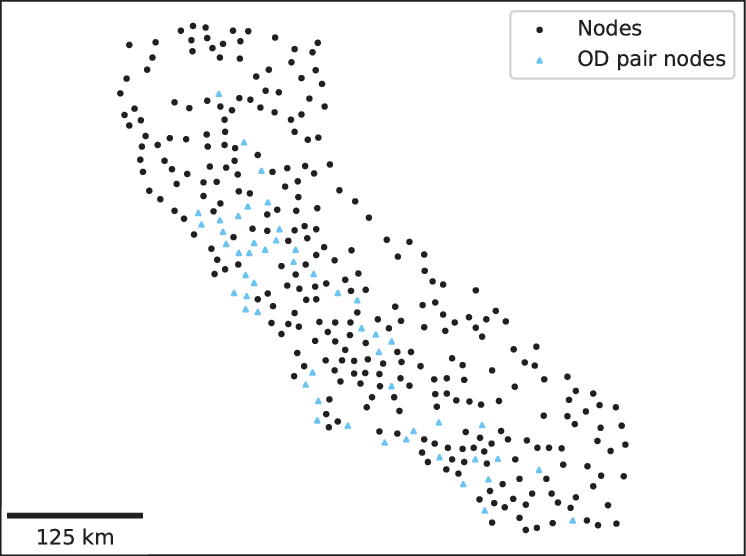} \label{fig:ca_map}}
\hspace{1.5em}
\subfloat[Europe]{\includegraphics[width=.46\linewidth]{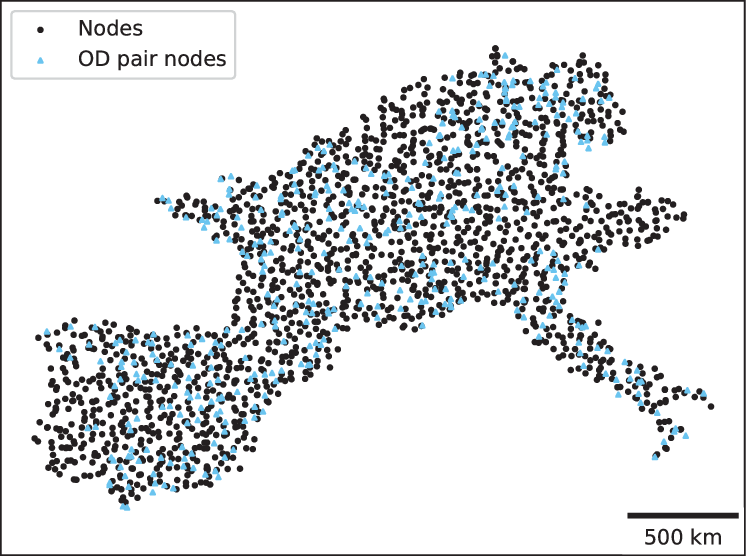}\label{fig:eu_map}}
}
\caption
{Visualization of network nodes for CA and EU instances\label{fig:maps}. \od\ pair nodes are shown as blue triangles, the refueling station candidates are black dots.}
\end{figure*}

\subsubsection{California road network}
This network has 339 nodes corresponding to intersections and major population centers, and 1,234 weighted arcs representing road connections between the nodes.
The corresponding expanded network graph contains nearly 25,000 arcs.
The arc distance in the expanded graph is the distance between their endpoints in the road network.
For simplicity, we assign an arc transit time that is proportional to its distance by multiplying with an average speed of $60$km/h.
The set of all population centers with 50,000 or more people are designated as \od\ pair nodes, and there are 1,167 \od\ pairs in total.

\subsubsection{Europe road network}

For this network, we sample locations uniformly at random in the following European countries: France, Belgium, Germany, Netherlands, Switzerland, Luxembourg, Austria, Spain, Portugal, and Italy.
For refueling stations we used nodes in all countries, while for terminal sites we excluded the Netherlands, Belgium, Austria, Switzerland and Portugal.
Figure~\ref{fig:eu_map} shows the locations of \od\ pair nodes and the candidate locations for refueling stations across Europe.

Then we determine distances and transit times of the expanded network graph by computing fastest routes in a real-world road network.
To this end we use an OpenSourceRoutingMachine (OSRM) \citep{luxen-vetter-2011} in conjunction with OpenStreetMap (OSM) data \citep{OpenStreetMap}, which we filter first down to the main roads to simplify computations.
More precisely, we only include roads labeled with the tags motorway, trunk, primary, secondary, motorway link, trunk link, primary link, and secondary link.
These road connections range from major divided highways down to smaller roads that link towns but do not include links between smaller towns or residential streets.
The resulting road network contains 2,688,299 ways and 20,476,319 nodes.
Due to the filtering, for some arcs of the expanded network no route could be computed. In this case we assign the same transit time and distance as its reverse arc (if it is available).

In this network we also model the time spent on refueling by adding a refueling time that is proportional to distance traveled for every arc that ends in a refueling station.
Specifically, we add 30mins $\cdot\; {\ell_a}/{\rangemax}$ to the transit time of arc $a \in \A$, where 30mins represents the time for a full refueling.

Among the terminal sites, we randomly selected a majority to be destination nodes and the rest to be origin nodes or both.
Then, we generate a set of \od pairs, where most represent regional, i.e.\ short distance demands, and some intermediate as well as long distance demands.

\subsubsection{Parameter settings}
We construct problem instances so as to closely mimic the scenarios tested by \cite{Arslan2019, Yildiz2016}.
Therefore, similar to these works, we determine the upper bound on transit time for each \od\ pair $q$ by adding a percentage $\lambda$ to the shortest path length between the origin and destination sites for a conventional vehicle.
For example, if $\lambda = 15\%$, \afvs cannot deviate more than 15\% from the shortest path from $s_q$ to $t_q$.
We assume the deviation tolerance is uniform across all \od\ pairs.
Also, while our model is capable of handling different building costs and capacities for each refueling station, we assume these are equal for all stations and \od\ pairs ($c_v = 1, \; \kappa_v = \kappa, \; \forall v \in \stations$). 
Further, we assume the half-capacity requirement, i.e.\ $\rangeorig = \rangedest = \rangemax/2$.

In order to obtain a range of different instances we vary three main parameters that affect the computational difficulty: (i) the vehicle's maximum range $\rangemax$, (ii) the transit time deviation tolerance $\lambda$, and (iii) the capacity of the refueling stations $\kappa$.

Different ranges and deviation tolerances affect the sizes of the subgraphs as documented in Table~\ref{tab:subgraphs}.
In general, increasing the deviation tolerance increases the number of nodes and arcs in subgraphs since \afvs have more freedom to travel to far-away refueling stations.
Increasing the range of the \afvs increases the number of arcs in the subgraphs since \afvs can reach more locations from every node.
Decreasing the capacity $\kappa$ increases the difficulty of the routing subproblem, which is to find time-feasible paths for all \od pairs while satisfying the capacity constraint.

\begin{table}[]
\centering
\caption{Statistics of the constructed expanded network graphs\label{tab:instances}.}
{
\setlength{\tabcolsep}{5pt}
\begin{tabular}{lrrrrrrrrrrr}
\toprule
\multicolumn{1}{c}{}        & \multicolumn{3}{c}{$\G = (\V, \A)$}                                                    & \multicolumn{1}{c}{} & \multicolumn{3}{c}{Node degree}                                                   & \multicolumn{1}{c}{} & \multicolumn{3}{c}{\od\ pair distance (km)}                                                 \\ \cmidrule{2-4} \cmidrule{6-8} \cmidrule{10-12} 
Network & \multicolumn{1}{c}{$|\V|$} & \multicolumn{1}{c}{$|\A|$} & \multicolumn{1}{c}{$|\odpairs|$} & \multicolumn{1}{c}{} & \multicolumn{1}{c}{Min} & \multicolumn{1}{c}{Mean} & \multicolumn{1}{c}{Max} & \multicolumn{1}{c}{} & \multicolumn{1}{c}{Min} & \multicolumn{1}{c}{Mean} & \multicolumn{1}{c}{Max} \\
\midrule
\textsc{CA}              & 390                     & 24,684                   & 1,167                     &                      & 7                       & 63.29                    & 127                     &                      & 30.1                    & 153.4                    & 463.5                   \\
\textsc{EU}                   & 2,197                    & 789,448                  & 7,605                     &                      & 12                      & 359.33                   & 704                     &                      & 7.27                   & 551.95                      & 2,940.49                  \\
\bottomrule
\end{tabular}
}
{}
\end{table}

\begin{table}[h!]
\centering
\caption{Size statistics of the origin-destination subgraphs\label{tab:subgraphs}. The table lists the number of nodes $|\V_q|$ and arcs $|\A_q|$ of the subgraphs $\G_q = (\V_q, \A_q)$ for varying maximum range $\rangemax$ and deviation parameter $\lambda$.}
{
\setlength{\tabcolsep}{5pt}
\begin{tabular}{rrrrrrrrrrr}
\toprule
            &     &    &            & \multicolumn{3}{c}{$|\V_q|$} & \multicolumn{1}{c}{} & \multicolumn{3}{c}{$|\A_q|$} \\ \cmidrule(lr){5-7} \cmidrule(lr){9-11} 
Network    & $\rangemax$    & $\lambda$(\%) & \begin{tabular}[c]{@{}c@{}}Time-feasible\\ \od\ pairs\end{tabular} & Min     & Mean      & Max    &                      & Min   & Mean      & Max      \\ \midrule
\textsc{CA} & 100km & 0.1     &   1,167     & 3       & 17.1      & 50     &                      & 12    & 263.1     & 1051     \\
            &     & 5.0     &   1,167      & 3       & 31.9      & 135    &                      & 12    & 1,050.6   & 7,436    \\
            &     & 10.0    &  1,167       & 3       & 44.5      & 179    &                      & 12    & 1,986.9   & 12,288   \\
            &     & 15.0    &  1,167       & 3       & 56.2      & 216    &                      & 12    & 3,007.7   & 16,397   \\
            & 200km & 0.1   &   1,167       & 3       & 17.1      & 50     &                      & 12    & 363.1     & 1,799    \\
            &     & 5.0    &  1,167        & 3       & 31.9      & 135    &                      & 12    & 1,533.8   & 13,590   \\
            &     & 10.0  &  1,167        & 3       & 44.5      & 179    &                      & 12    & 3,015.5   & 23,705   \\
            &     & 15.0  &  1,167         & 3       & 56.2     & 216    &                      & 12    & 4,780.2   & 33,706   \\ \midrule
\textsc{EU} & 650km & 10.0  &  6,761        & 1       & 8.64      & 129    &                      & 2     & 142.84     & 7,068 \\
            &     & 15.0    &  7,381        & 1      & 27.29      & 712    &                       & 2    & 1,828.23     & 199,444\\
            &     & 20.0    &  7,467        & 1      & 50.27      & 1,102  &                       & 2    & 6,040.76     & 408,466\\
            &     & 25.0    &  7,510        & 1      & 72.02      & 1,309  &                       & 2    & 11,222.59    & 513,184 \\
            \bottomrule
\end{tabular}
}
\end{table}

\subsection{Results}

In this section we evaluate our algorithms by comparing solution quality (objective value, optimality gap) and runtime across a number of problem instances.
Additionally, we report the average \emph{utilization} of refueling stations, i.e.\ the total demand flow it services divided by its capacity.
The utilization indicates the restrictiveness of the available capacity and how efficiently it is used in the solution.
Note that we provide plots of all the results in Appendix~\ref{app:plots-ca}.

Moreover, in Section~\ref{sec:int-sep-results} we evaluate our improved integer separation approach for the cut formulation, which we compare against \citep{Arslan2019} on non-capacitated problem instances.

\paragraph{California road network}

Similar to \cite{Arslan2019}, we test a range of deviation tolerances $\lambda = 0.1\%$, $5.0\%$, $10.0\%$, and $15.0\%$ and ranges $\rangemax = 100$km and $200$km for all \od\ pairs. 
For capacity constraints, we choose a range of $\kappa_v$ from $25$ to $250$ in increments of $25$.
In total, we run our solvers on 80 instances of the California road network, each run with a time limit of 10 hours.

The full results are listed in Table~\ref{tab:fullca100} for $\rangemax$ = 100km and Table~\ref{tab:fullca200} for $\rangemax$ = 200km.
For $\lambda = 5.0\%$ and $\rangemax = 100$km we show a visual comparison in Figure~\ref{fig:results_ca}.
Across all instances 79 of them had feasible solutions, among which \eqref{eq:path} solved all 79 to within a $10\%$ optimality gap and \eqref{eq:cut} solved 50.

When the capacity constraints are less restrictive, i.e.\ for higher values of $\kappa$, both formulations find the optimal solution. 
Furthermore, in this case solving the cut formulation is faster and requires exploration of fewer branch-and-bound nodes.
To be precise, across all 40 instances with $\kappa \geq 150$, the average runtime for \eqref{eq:cut} is $23.84\%$ less than for \eqref{eq:path}.

In the more restrictive case, the problem appears to be much harder as we hit the time limit for both formulations.
With the path formulation we are able to find better solutions (fewer stations to accommodate the demand) with lower optimality gaps.
In fact, across all instances where $\kappa_v < 150$, the mean optimality gap is $1.80\%$ for \eqref{eq:path} and $33.65\%$ for \eqref{eq:cut}.
In 31 of those instances \eqref{eq:path} gives a better solution and in only 1 instance the solution obtained from \eqref{eq:cut} is better.
Across all 32 instances where the solutions differ, the solutions from \eqref{eq:path} have $16.52\%$ fewer stations on average.

\begin{figure*}[!ht]
\centering
{
\begin{minipage}{.93\linewidth}
\subfloat[Optimality gap]{\includegraphics[width=.49\linewidth]{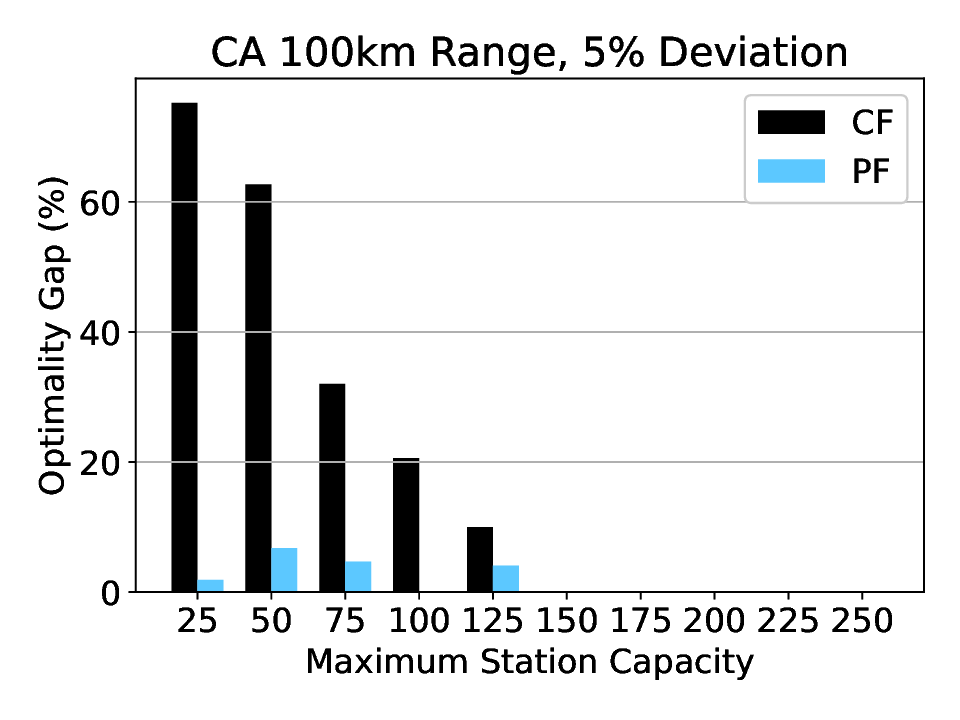} \label{fig:gap}}
\subfloat[Solution time]{\includegraphics[width=.49\linewidth]{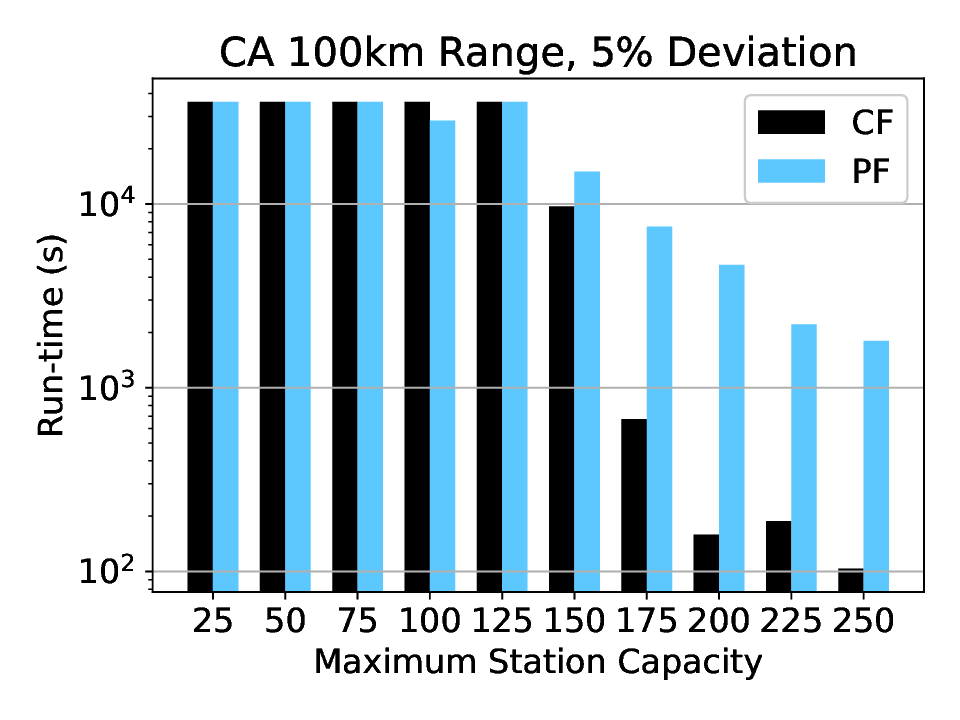}\label{fig:time}}
\newline
\subfloat[Utilization]{\includegraphics[width=.49\linewidth]{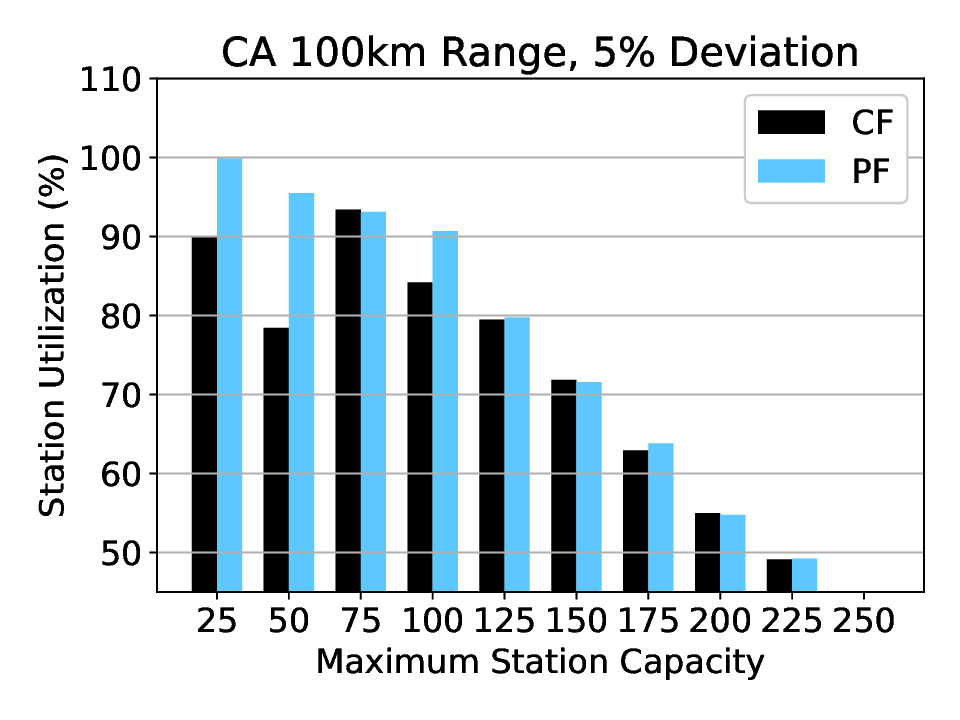} \label{fig:utilization}}
\subfloat[Objective Value]{\includegraphics[width=.49\linewidth]{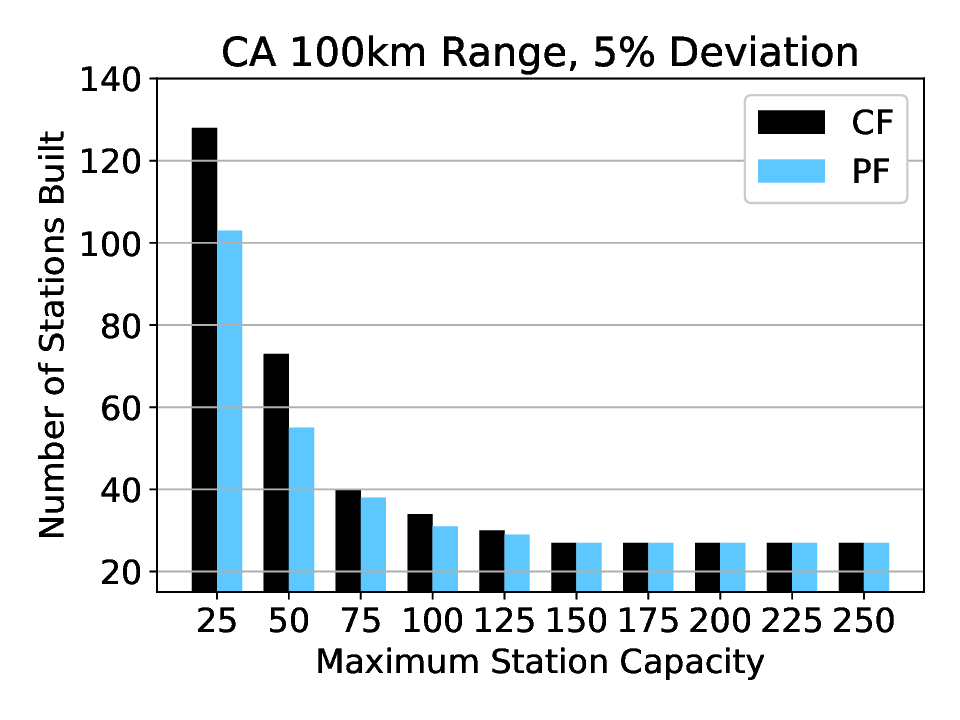}\label{fig:objval}}
\end{minipage}
}
\caption{Results for the CA road network\label{fig:results_ca}. The plot shows results with deviation tolerance $\lambda = 5\%$ and vehicle range $\rangemax = 100$km for \eqref{eq:cut} and \eqref{eq:path}.}
\end{figure*}

\paragraph{Europe road network}

\begin{figure*}[!ht]
\centering
{
\begin{minipage}{.93\linewidth}
\subfloat[Optimality gap]{\includegraphics[width=.49\linewidth]{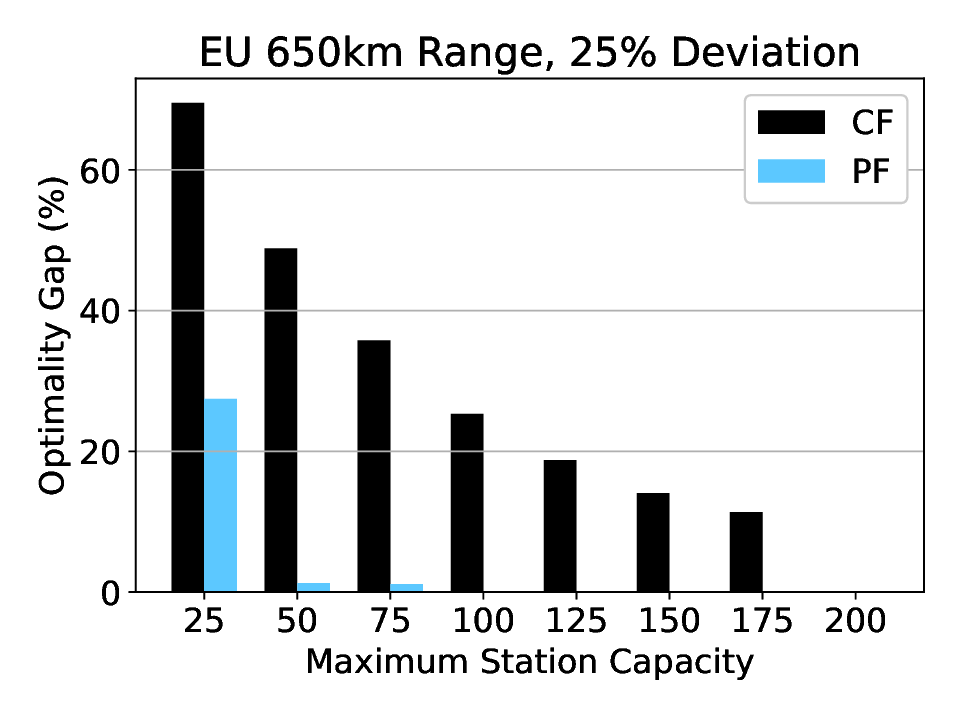} \label{fig:eugap}}
\subfloat[Solution time]{\includegraphics[width=.49\linewidth]{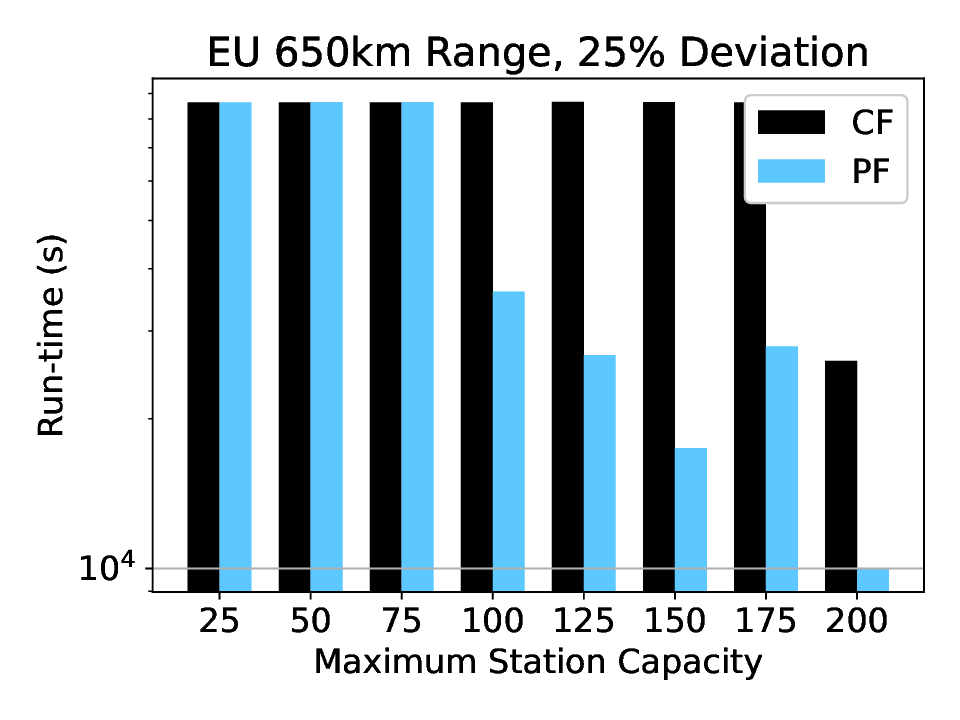}\label{fig:eutime}}
\newline
\subfloat[Utilization]{\includegraphics[width=.49\linewidth]{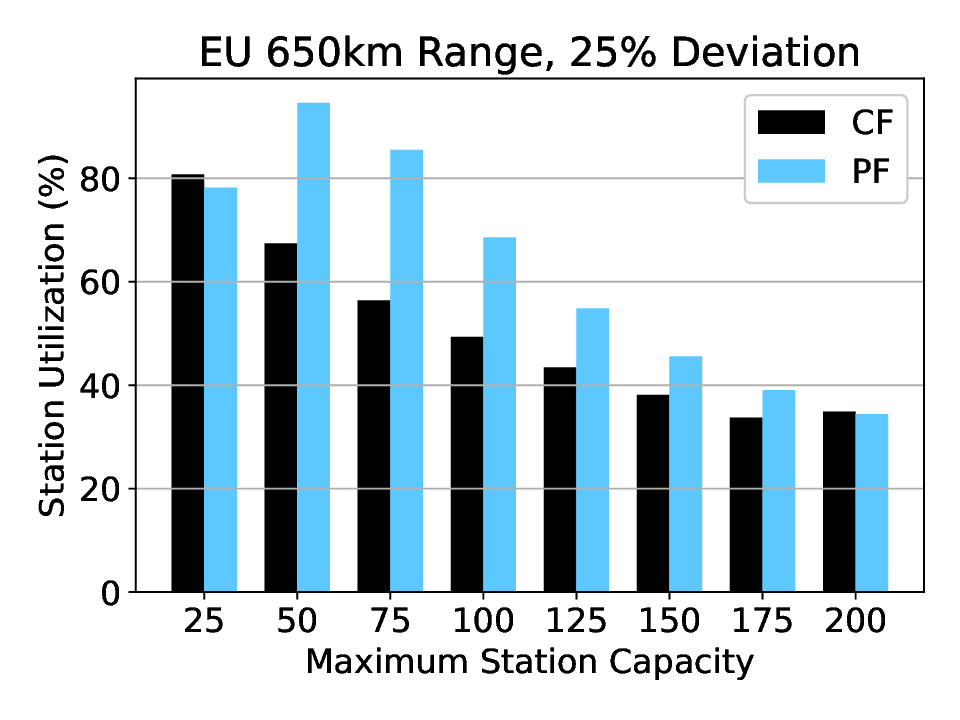} \label{fig:euutilization}}
\subfloat[Objective Value]{\includegraphics[width=.49\linewidth]{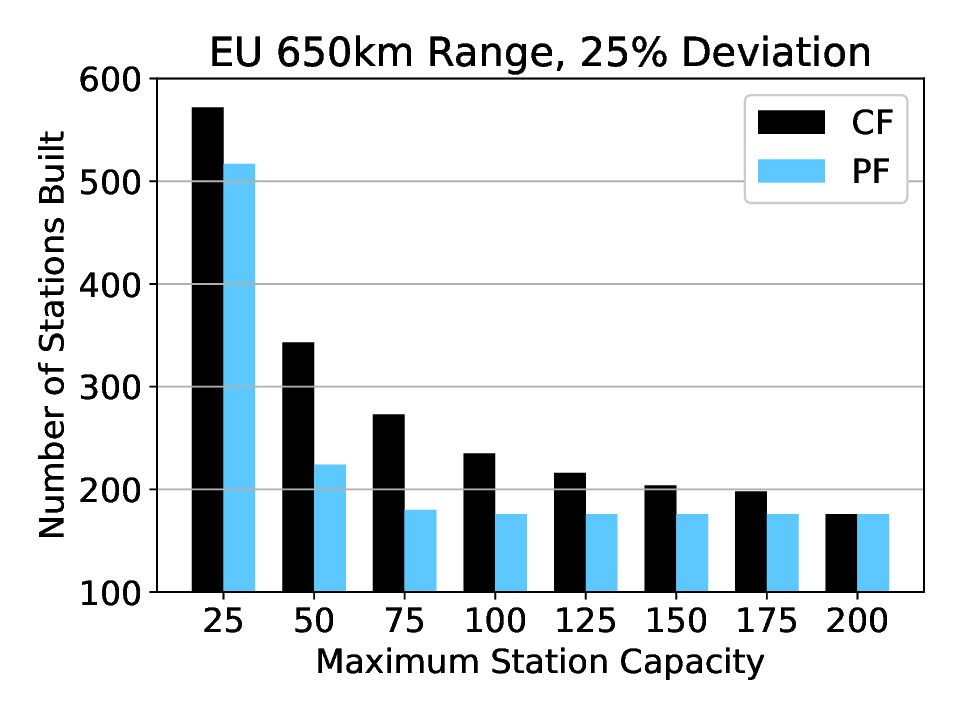}\label{fig:euobjval}}
\end{minipage}
}
\caption{Results for the EU road network\label{fig:euresults}. The plot shows results with deviation tolerance $\lambda = 25\%$ and vehicle range $\rangemax = 650$km for \eqref{eq:cut} and \eqref{eq:path}.}
\end{figure*}

For the larger Europe network instances, we assume a maximum range of $\rangemax =$ 650km.
We test deviation tolerances of $10\%$, $15\%$, $20\%$, and $25\%$, and a range of station capacities from $25$ to $200$ in increments of $25$, which amounts to a total of 32 instances.
We set the time limit of our solvers to 24 hours.

The full results are found in Table~\ref{tab:fulleu} and a plot for the case $\lambda = 25\%$ is shown in Figure~\ref{fig:euresults}.
Of the 32 instances we tested, 29 have a feasible solution.
On the one hand, \eqref{eq:path} solves 21 instances to optimality, 7 with a relative optimality gap within $2\%$, and 1 with a gap of over $27\%$.
On the other hand, \eqref{eq:cut} solves 10 instances to optimality, 5 with a relative gap within $2\%$, 14 with gaps greater than $10\%$, and fails to find any feasible solution for one instance.
Across all the EU road network instances, \eqref{eq:path} finds better solutions than \eqref{eq:cut}.
On average across the 14 instances where the objective values differ, the solutions obtained with \eqref{eq:path} have $33.40\%$ fewer stations than the solutions returned by \eqref{eq:cut}.

For all instances with restrictive capacity constraints ($\kappa < 125$), the mean optimality gap of \eqref{eq:cut} is $23.81\%$ while the mean optimality gap of \eqref{eq:path} is $2.68\%$.
Unlike for the California road network, \eqref{eq:cut} does not provide significant runtime reductions compared to \eqref{eq:path} in the presence of less restrictive capacity constraints. 
Only in 4 of 32 instances our solver for \eqref{eq:cut} terminates before the one for \eqref{eq:path}, and all four instances have the smallest deviation tolerance $\lambda = 10.0\%$.
This suggests that \eqref{eq:path} may be suited for solving large-scale instances of the \rslprc.

\begin{table}[]
\centering
\caption
{Results for the CA road network with 100km range\label{tab:fullca100}. The table lists results for  \eqref{eq:cut}/\eqref{eq:path} with vehicle range $\rangemax = 100$km. We report runtime, objective value, number of explored branch-and-bound nodes, optimality gap and average station utilization for different deviation tolerances $\lambda$ and station capacities $\kappa$. (--) denotes instances for which no feasible solution exists.}
{
\begin{tabular}{lllllll}
\toprule
$\lambda$             & $\kappa$ & Time(s)    & Obj value & \#Nodes   & Gap(\%) & Utilization(\%)  \\
\midrule
0.1\% & 25 & \textbf{72}/36000 & ---/--- & 0/2691 & ---/--- & ---/--- \\
& 50 & \textbf{36000}/\textbf{36000} & 58/\textbf{54} & 7252/1059 & 42.27/\textbf{2.0} & 93.55/96.00 \\
& 75 & \textbf{36000}/\textbf{36000} & 42/\textbf{40} & 13456/1484 & 26.06/\textbf{4.0} & 88.79/91.60 \\
& 100 & \textbf{36000}/\textbf{36000} & \textbf{34}/35 & 9207/1051 & 8.6/\textbf{8.0} & 83.29/80.17 \\
& 125 & \textbf{631}/7678 & \textbf{31}/\textbf{31} & 216/72 & \textbf{0.0}/\textbf{0.0} & 74.27/73.75 \\
& 150 & \textbf{111}/12239 & \textbf{30}/\textbf{30} & 44/194 & \textbf{0.0}/\textbf{0.0} & 64.96/64.60 \\
& 175 & \textbf{57}/4163 & \textbf{30}/\textbf{30} & 33/49 & \textbf{0.0}/\textbf{0.0} & 55.68/55.98 \\
& 200 & \textbf{25}/1645 & \textbf{29}/\textbf{29} & 8/12 & \textbf{0.0}/\textbf{0.0} & 50.03/50.29 \\
& 225 & \textbf{19}/7520 & \textbf{29}/\textbf{29} & 8/100 & \textbf{0.0}/\textbf{0.0} & 44.69/45.72 \\
& 250 & \textbf{12}/2071 & \textbf{29}/\textbf{29} & 7/12 & \textbf{0.0}/\textbf{0.0} & 39.88/40.51 \\
\midrule
5.0\% & 25 & \textbf{36000}/\textbf{36003} & 128/\textbf{103} & 1461/1829 & 75.21/\textbf{2.0} & 90.03/99.88 \\
& 50 & \textbf{36002}/\textbf{36000} & 73/\textbf{55} & 1757/948 & 62.7/\textbf{7.0} & 78.43/95.49 \\
& 75 & \textbf{36000}/\textbf{36000} & 40/\textbf{38} & 1857/566 & 32.0/\textbf{5.0} & 93.40/93.12 \\
& 100 & 36002/\textbf{28412} & 34/\textbf{31} & 1596/347 & 20.59/\textbf{0.0} & 84.21/90.68 \\
& 125 & \textbf{36000}/\textbf{36000} & 30/\textbf{29} & 1848/583 & 10.0/\textbf{4.0} & 79.49/79.78 \\
& 150 & \textbf{9721}/15032 & \textbf{27}/\textbf{27} & 529/134 & \textbf{0.0}/\textbf{0.0} & 71.85/71.58 \\
& 175 & \textbf{674}/7531 & \textbf{27}/\textbf{27} & 39/53 & \textbf{0.0}/\textbf{0.0} & 62.92/63.81 \\
& 200 & \textbf{159}/4669 & \textbf{27}/\textbf{27} & 16/28 & \textbf{0.0}/\textbf{0.0} & 55.00/54.78 \\
& 225 & \textbf{188}/2214 & \textbf{27}/\textbf{27} & 15/24 & \textbf{0.0}/\textbf{0.0} & 49.15/49.23 \\
& 250 & \textbf{103}/1798 & \textbf{27}/\textbf{27} & 8/37 & \textbf{0.0}/\textbf{0.0} & 44.74/43.82 \\
\midrule
10.0\% & 25 & 36000/\textbf{1,244} & 134/\textbf{101} & 742/1,311 & 79.95/\textbf{0.0} & 88.48/99.96 \\
& 50 & \textbf{36000}/\textbf{36000} & 73/\textbf{55} & 704/837 & 65.09/\textbf{7.0} & 80.82/93.02 \\
& 75 & \textbf{36000}/\textbf{36000} & 49/\textbf{37} & 517/901 & 47.81/\textbf{4.0} & 78.07/94.10 \\
& 100 & \textbf{36000}/\textbf{36000} & 41/\textbf{31} & 737/687 & 37.68/\textbf{7.0} & 71.12/89.61 \\
& 125 & 36000/\textbf{9606} & 31/\textbf{27} & 633/53 & 17.66/\textbf{0.0} & 74.04/85.16 \\
& 150 & \textbf{23902}/36000 & \textbf{27}/\textbf{27} & 1121/359 & \textbf{0.0}/7.0 & 71.19/74.35 \\
& 175 & \textbf{1955}/6071 & \textbf{26}/\textbf{26} & 42/21 & \textbf{0.0}/\textbf{0.0} & 65.54/65.78 \\
& 200 & \textbf{789}/3649 & \textbf{26}/\textbf{26} & 38/14 & \textbf{0.0}/\textbf{0.0} & 57.61/58.71 \\
& 225 & \textbf{668}/5237 & \textbf{26}/\textbf{26} & 32/15 & \textbf{0.0}/\textbf{0.0} & 52.65/50.91 \\
& 250 & \textbf{781}/3201 & \textbf{26}/\textbf{26} & 38/20 & \textbf{0.0}/\textbf{0.0} & 47.11/45.65 \\
\midrule
15.0\% & 25 & \textbf{36000}/\textbf{36000} & 133/\textbf{102} & 559/1917 & 81.64/\textbf{1.0} & 90.08/99.53 \\
& 50 & \textbf{36000}/\textbf{36000} & 72/\textbf{56} & 410/609 & 67.36/\textbf{10.0} & 82.53/93.50 \\
& 75 & 36000/\textbf{11496} & 52/\textbf{36} & 257/1594 & 54.76/\textbf{0.0} & 75.23/96.81 \\
& 100 & \textbf{36000}/\textbf{36000} & 41/\textbf{29} & 392/438 & 42.68/\textbf{4.0} & 72.07/94.07 \\
& 125 & 36000/\textbf{17683} & 34/\textbf{26} & 317/76 & 30.88/\textbf{0.0} & 67.69/86.52 \\
& 150 & \textbf{36000}/\textbf{36000} & 32/\textbf{25} & 340/69 & 26.56/\textbf{6.0} & 59.33/78.24 \\
& 175 & 36000/\textbf{25115} & 25/\textbf{24} & 218/47 & 4.0/\textbf{0.0} & 68.21/71.31 \\
& 200 & \textbf{2291}/5309 & \textbf{24}/\textbf{24} & 61/25 & \textbf{0.0}/\textbf{0.0} & 63.04/62.52 \\
& 225 & \textbf{1211}/6011 & \textbf{24}/\textbf{24} & 47/17 & \textbf{0.0}/\textbf{0.0} & 56.91/56.43 \\
& 250 & \textbf{3882}/4527 & \textbf{24}/\textbf{24} & 67/16 & \textbf{0.0}/\textbf{0.0} & 51.07/50.28 \\
\bottomrule
\end{tabular}
}
\end{table}

\begin{table}[]
\centering
\caption
{Results for the CA road network with 200km range\label{tab:fullca200}. The table lists results for  \eqref{eq:cut}/\eqref{eq:path} with vehicle range $\rangemax = 200$km. We report runtime, objective value, number of explored branch-and-bound nodes, optimality gap and average station utilization for different deviation tolerances $\lambda$ and station capacities $\kappa$.}
{
\begin{tabular}{lllllll}
\toprule
$\lambda$             & $\kappa$ & Time(s)    & Obj value & \#Nodes   & Gap(\%)   & Utilization(\%)\\
\midrule
0.1\% & 25 & 36003/\textbf{331} & 65/\textbf{62} & 3484/197 & 63.72/\textbf{0.0} & 99.26/100.00 \\
& 50 & 36002/\textbf{28261} & 34/\textbf{32} & 3694/1846 & 33.57/\textbf{0.0} & 97.59/98.94 \\
& 75 & 36002/\textbf{9732} & 26/\textbf{25} & 3754/1320 & 10.26/\textbf{0.0} & 88.05/88.91 \\
& 100 & \textbf{763}/826 & \textbf{23}/\textbf{23} & 97/23 & \textbf{0.0}/\textbf{0.0} & 78.87/76.52 \\
& 125 & \textbf{297}/2362 & \textbf{23}/\textbf{23} & 46/95 & \textbf{0.0}/\textbf{0.0} & 63.69/61.18 \\
& 150 & \textbf{330}/1197 & \textbf{23}/\textbf{23} & 45/46 & \textbf{0.0}/\textbf{0.0} & 52.99/51.21 \\
& 175 & \textbf{501}/1833 & \textbf{23}/\textbf{23} & 65/66 & \textbf{0.0}/\textbf{0.0} & 44.87/44.17 \\
& 200 & \textbf{225}/1743 & \textbf{23}/\textbf{23} & 29/112 & \textbf{0.0}/\textbf{0.0} & 37.65/38.22 \\
& 225 & \textbf{282}/4412 & \textbf{23}/\textbf{23} & 39/913 & \textbf{0.0}/\textbf{0.0} & 33.58/35.09 \\
& 250 & \textbf{275}/1751 & \textbf{23}/\textbf{23} & 35/156 & \textbf{0.0}/\textbf{0.0} & 29.95/31.32 \\
\midrule
5.0\% & 25 & 36002/\textbf{1042} & 64/\textbf{62} & 1856/604 & 67.12/\textbf{0.0} & 99.75/99.61 \\
& 50 & 36003/\textbf{14402} & 34/\textbf{32} & 2139/2103 & 38.24/\textbf{0.0} & 98.76/99.12 \\
& 75 & 36000/\textbf{21274} & 25/\textbf{24} & 2318/3187 & 13.06/\textbf{0.0} & 91.31/94.06 \\
& 100 & \textbf{1715}/1839 & \textbf{22}/\textbf{22} & 110/42 & \textbf{0.0}/\textbf{0.0} & 81.09/79.45 \\
& 125 & \textbf{636}/1233 & \textbf{22}/\textbf{22} & 54/81 & \textbf{0.0}/\textbf{0.0} & 66.04/64.04 \\
& 150 & \textbf{318}/911 & \textbf{22}/\textbf{22} & 39/67 & \textbf{0.0}/\textbf{0.0} & 55.09/54.03 \\
& 175 & \textbf{322}/388 & \textbf{22}/\textbf{22} & 30/25 & \textbf{0.0}/\textbf{0.0} & 46.81/46.08 \\
& 200 & \textbf{266}/1131 & \textbf{22}/\textbf{22} & 24/94 & \textbf{0.0}/\textbf{0.0} & 39.34/40.75 \\
& 225 & \textbf{273}/594 & \textbf{22}/\textbf{22} & 26/66 & \textbf{0.0}/\textbf{0.0} & 36.04/36.20 \\
& 250 & \textbf{211}/727 & \textbf{22}/\textbf{22} & 19/54 & \textbf{0.0}/\textbf{0.0} & 32.93/32.71 \\
\midrule
10.0\% & 25 & 36000/\textbf{307} & 77/\textbf{62} & 1367/161 & 73.84/\textbf{0.0} & 90.39/99.55 \\
& 50 & \textbf{36000}/\textbf{36000} & 35/\textbf{32} & 1383/3246 & 42.86/\textbf{3.0} & 93.49/97.31 \\
& 75 & 36000/\textbf{4502} & 25/\textbf{23} & 1139/1382 & 20.0/\textbf{0.0} & 94.40/94.96 \\
& 100 & 2438/\textbf{1884} & \textbf{21}/\textbf{21} & 91/89 & \textbf{0.0}/\textbf{0.0} & 85.90/84.48 \\
& 125 & \textbf{961}/969 & \textbf{21}/\textbf{21} & 42/48 & \textbf{0.0}/\textbf{0.0} & 69.60/66.74 \\
& 150 & \textbf{811}/1396 & \textbf{21}/\textbf{21} & 49/223 & \textbf{0.0}/\textbf{0.0} & 59.59/56.70 \\
& 175 & \textbf{593}/808 & \textbf{21}/\textbf{21} & 30/38 & \textbf{0.0}/\textbf{0.0} & 50.31/48.19 \\
& 200 & \textbf{419}/1718 & \textbf{21}/\textbf{21} & 33/117 & \textbf{0.0}/\textbf{0.0} & 44.14/41.86 \\
& 225 & \textbf{493}/1149 & \textbf{21}/\textbf{21} & 46/368 & \textbf{0.0}/\textbf{0.0} & 36.53/37.46 \\
& 250 & \textbf{419}/2888 & \textbf{21}/\textbf{21} & 21/244 & \textbf{0.0}/\textbf{0.0} & 32.95/33.75 \\
\midrule
15.0\% & 25 & 36008/\textbf{3250} & 65/\textbf{62} & 1017/715 & 71.57/\textbf{0.0} & 98.52/99.54 \\
& 50 & \textbf{36000}/\textbf{36000} & 34/\textbf{32} & 1004/2271 & 45.59/\textbf{3.0} & 98.18/96.75 \\
& 75 & 36006/\textbf{502} & 24/\textbf{22} & 1148/112 & 22.92/\textbf{0.0} & 94.28/96.42 \\
& 100 & \textbf{1511}/1521 & \textbf{19}/\textbf{19} & 45/30 & \textbf{0.0}/\textbf{0.0} & 91.53/89.84 \\
& 125 & \textbf{906}/1549 & \textbf{19}/\textbf{19} & 27/33 & \textbf{0.0}/\textbf{0.0} & 76.38/76.13 \\
& 150 & \textbf{376}/858 & \textbf{19}/\textbf{19} & 13/17 & \textbf{0.0}/\textbf{0.0} & 65.43/62.14 \\
& 175 & \textbf{333}/917 & \textbf{19}/\textbf{19} & 10/26 & \textbf{0.0}/\textbf{0.0} & 55.43/54.74 \\
& 200 & \textbf{171}/633 & \textbf{19}/\textbf{19} & 5/15 & \textbf{0.0}/\textbf{0.0} & 45.71/46.24 \\
& 225 & \textbf{361}/836 & \textbf{19}/\textbf{19} & 10/39 & \textbf{0.0}/\textbf{0.0} & 40.61/41.87 \\
& 250 & \textbf{121}/776 & \textbf{19}/\textbf{19} & 5/26 & \textbf{0.0}/\textbf{0.0} & 36.91/38.27 \\
\bottomrule
\end{tabular}
}
\end{table}

\begin{table}[]
\centering
\caption
{Results for the EU network\label{tab:fulleu}. The table lists results for \eqref{eq:cut}/\eqref{eq:path} with vehicle range $\rangemax = 650$km. We report runtime, objective value, number of explored branch-and-bound nodes, optimality gap and average station utilization for different deviation tolerances $\lambda$ and station capacities $\kappa$. (--) denotes instances for which no feasible solution exists.}
{
\begin{tabular}{lllllll}
\toprule
$\lambda$             & $\kappa$ & Time(s)    & Obj value & \#Nodes  & Gap(\%) & Utilization(\%)  \\
\midrule
10.0\% & 25  & --/-- & --/-- & --/-- & --/-- & --/-- \\
       & 50  & --/-- & --/-- & --/-- & --/-- & --/-- \\
       & 75  & --/-- & --/-- & --/-- & --/-- & --/-- \\
       & 100 & 218/\textbf{103} & \textbf{346}/\textbf{346} & 17/3 & \textbf{0.0}/\textbf{0.0} & 29.59/29.33\\
       & 125 & \textbf{70}/128 & \textbf{345}/\textbf{345} & 6/1 & \textbf{0.0}/\textbf{0.0} & 23.89/23.50\\
       & 150 & \textbf{55}/173 & \textbf{345}/\textbf{345} & 5/1 & \textbf{0.0}/\textbf{0.0} & 19.94/19.64\\
       & 175 & \textbf{60}/131 & \textbf{345}/\textbf{345} & 5/1 & \textbf{0.0}/\textbf{0.0} & 17.02/16.80\\
       & 200 & \textbf{21}/121 & \textbf{345}/\textbf{345} & 3/1 & \textbf{0.0}/\textbf{0.0} & 14.49/14.71\\
\midrule
15.0\% & 25  & \textbf{87,038}/\textbf{86,400} & $\infty$/\textbf{432} & -/6,543 & $\infty$/\textbf{1.12} & 0.0/95.48 \\
       & 50  & 86,400/\textbf{55,294} & 428/\textbf{273} & 17/2,909 & 40.31/\textbf{0.0}  & 52.28/81.38 \\
       & 75  & 86,400/\textbf{4,890}  & 347/\textbf{259} & 307/138   & 26.22/\textbf{0.0}  & 42.96/58.86 \\
       & 100 & 86,400/\textbf{3,642}  & \textbf{258}/\textbf{258} & 99/198   & 0.65/\textbf{0.0}  & 45.00/44.80\\
       & 125 & 26,615/\textbf{2,417}  & \textbf{258}/\textbf{258} & 143/137   & \textbf{0.0}/\textbf{0.0}  & 36.14/35.71\\
       & 150 & 5,720/\textbf{1,993}  & \textbf{258}/\textbf{258} & 62/140   & \textbf{0.0}/\textbf{0.0}  & 30.27/30.04\\
       & 175 & 8,457/\textbf{2,370}  & \textbf{258}/\textbf{258} & 117/142   & \textbf{0.0}/\textbf{0.0}  & 25.85/25.75\\
       & 200 & 6,696/\textbf{2,039}  & \textbf{258}/\textbf{258} & 68/180   & \textbf{0.0}/\textbf{0.0}  & 22.67/22.48\\
\midrule
20.0\% & 25  & \textbf{86,400}/\textbf{86,400} & 618/\textbf{415} & 3/1,550  & 64.95/\textbf{0.73} & 73.15/97.97 \\
       & 50  & \textbf{86,400}/\textbf{86,400} & 391/\textbf{244} & 16/1,334  & 45.34/\textbf{1.09} & 59.00/89.75 \\
       & 75  & \textbf{86,401}/\textbf{86,400} & 304/\textbf{218} & 15/529    & 29.70/\textbf{0.93} & 50.29/71.94 \\
       & 100 & \textbf{86,400}/\textbf{86,400} & 277/\textbf{217} & 23/1,757  & 22.92/\textbf{0.97} & 41.28/54.87\\
       & 125 & 86,400/\textbf{75,880} & 218/\textbf{217} & 208/2,871  & 1.96/\textbf{0.0} & 44.11/43.57\\
       & 150 & 86,400/\textbf{68,563} & \textbf{217}/\textbf{217} & 115/10,723 & 1.45/\textbf{0.0} & 37.02/36.47\\
       & 175 & 86,400/\textbf{45,764} & \textbf{217}/\textbf{217} & 75/2,126  & 1.51/\textbf{0.0}  & 31.71/31.25\\
       & 200 & 86,400/\textbf{22,231} & \textbf{217}/\textbf{217} & 84/1,151  & 1.28/\textbf{0.0}  & 27.80/27.37\\
\midrule
25.0\% & 25  & \textbf{86,401}/\textbf{86,400} & 572/\textbf{517} & 1/71  & 69.51/\textbf{27.45} & 80.77/78.18 \\
       & 50  & \textbf{86,400}/\textbf{86,400} & 343/\textbf{224} & 5/493 & 48.83/\textbf{1.27}  & 67.43/94.56 \\
       & 75  & \textbf{86,400}/\textbf{86,400} & 273/\textbf{180} & 7/423 & 35.78/\textbf{1.13}  & 56.40/85.51 \\
       & 100 & 86,400/\textbf{35,993} & 235/\textbf{176} & 19/138 & 23.32/\textbf{0.0}   & 49.33/68.59\\
       & 125 & 86,563/\textbf{26,818} & 216/\textbf{176} & 36/444 & 18.75/\textbf{0.0}   & 43.43/54.85\\
       & 150 & 86,456/\textbf{17,447} & 204/\textbf{176} & 116/147 & 14.05/\textbf{0.0}   & 38.14/45.58\\
       & 175 & 86,401/\textbf{27,923} & 198/\textbf{176} & 157/633 & 11.36/\textbf{0.0}   & 33.75/39.07\\
       & 200 & 26,116/\textbf{9,989}  & \textbf{176}/\textbf{176} & 37/56  & \textbf{0.0}/\textbf{0.0}   & 34.91/34.40\\
\bottomrule
\end{tabular}
}
\end{table}

\subsubsection{Improved integer separation for the cut formulation}
\label{sec:int-sep-results}

In this section we compare our improved integer separation method introduced in Section~\ref{sec:int-sep} to the approach proposed by \cite{Arslan2019}.
In the prior work the authors distinguish five variations of their branch-and-cut algorithm for solving the non-capacitated \rslpr, which differ in the separation algorithms they employ.
The version that is fastest in their evaluation only runs separation at integer solutions (refered to as ``B\&C-1'').
Here, we repeat the experiments from \cite{Arslan2019} using the non-capacitated version of \rslpr stated in \eqref{eq:cutinf} and compare our own method to B\&C-1.
Note that we use the same implementation for both methods (i.e.\ no fractional separation) and only swap the integer separation routine.

\begin{align*}
\min \quad & \sum_{v \in \stations} c_v x_v \tag{CF$\infty$} \label{eq:cutinf} \\
\text{s.t.} \quad
& \sum_{v \in \stations} x_v \geq 1 & \forall q \in \odpairs, \; S \in \Gamma_q \label{eq:cut-coverage-inf}\\
& x_v\in \{0, 1\} & \forall v \in \stations 
\end{align*}

In Table~\ref{tab:int-sep-results} we report the runtime and the number of calls to the Dijkstra shortest path algorithm on problem instances for both network graphs.
We can see that our separation method significantly reduces the number of Dijkstra calls across all instances and the highest reduction is exhibited for instances with larger subgraphs.
This is because our method determines a small initial time-separator by running Dijkstra twice instead of taking the complement of the current active stations, which gives the most benefit when there are many nodes in the subgraph. 
Similarly, our algorithm runs faster for all instances except the smallest ones, where the total runtime is negligible to begin with.
For the most challenging instances with large subgraphs our algorithm is almost 50\% faster.

From Table~\ref{tab:int-sep-results} it is also apparent that the addition of capacity constraints renders the \rslprc much harder to solver than the non-capacitated \rslpr.
On the same network graph, we can solve \eqref{eq:cutinf} in seconds while solving \eqref{eq:cut} and \eqref{eq:path} can take hours depending on the strictness of the capacity constraints.

\begin{table}[]
\centering
\caption
{Comparison of integer separation routines\label{tab:int-sep-results}. The table lists the total solver runtime and the number of Dijkstra calls for the B\&C-1 method from \cite{Arslan2019} and our integer separation method on various non-capacitated problem instances.}
{
\begin{tabular}{rrrrrrrr}
\toprule
   &                     &            & \multicolumn{2}{c}{Runtime (sec)} &  & \multicolumn{2}{c}{\# Dijkstra calls} \\ \cmidrule(lr){4-5} \cmidrule(lr){7-8} 
Network & $\lambda$ & \afv range & B\&C-1             & Ours               &  & B\&C-1       & Ours                \\
\midrule
CA & 0\%                     & 100km      & \textbf{0.30}      & 0.32               &  & 23,314       & \textbf{19,154}     \\
 &                       & 150km      & \textbf{0.33}               & \textbf{0.33}               &  & 21,022       & \textbf{15,687}     \\
        &                & 200km      & \textbf{0.31}               & \textbf{0.31}               &  & 20,439       & \textbf{14,042}     \\ \midrule
& 10\%                    & 100km      & 1.21               & \textbf{1.06}      &  & 60,153       & \textbf{28,986}     \\
 &                       & 150km      & 1.43               & \textbf{1.30}      &  & 52,023       & \textbf{27,928}     \\
  &                      & 200km      & 1.46               & \textbf{1.45}      &  & 51,960       & \textbf{30,678}     \\ \midrule
& 20\%                    & 100km      & 2.63               & \textbf{1.87}      &  & 82,107       & \textbf{34,768}     \\
  &                      & 150km      & 3.29               & \textbf{2.50}      &  & 77,922       & \textbf{34,229}     \\
 &                       & 200km      & 3.58               & \textbf{3.08}      &  & 77,954       & \textbf{41,903}     \\ \midrule
& 50\%                    & 100km      & 6.31               & \textbf{3.56}      &  & 149,155      & \textbf{51,863}     \\
 &                       & 150km      & 7.11               & \textbf{5.00}      &  & 123,247      & \textbf{47,899}     \\
 &                       & 200km      & 8.10               & \textbf{6.79}      &  & 122,365      & \textbf{62,258}     \\
                        \midrule
EU & 10\%                    & 650km      & 4.63                 & \textbf{4.53}        &  & 106,285       & \textbf{81,143}     \\
& 15\%                    &            & 28.46                & \textbf{20.89}       &  & 254,580       & \textbf{167,917}     \\
& 20\%                    &            & 117.66               & \textbf{63.99}       &  & 452,795       & \textbf{239,382}     \\ 
& 25\%                    &            & 219.07               & \textbf{115.68}      &  & 596,544       & \textbf{303,696}     \\ \bottomrule
\end{tabular}
}
\end{table}

\section{Conclusion}

In this paper, we presented the capacitated refueling station location problem with routing (\rslprc). 
The solution to this problem locates a minimal cost set of refueling stations such that \afvs can transport all feasible demand across every \od\ pair; no drivers deviate from the shortest path by more than their maximal tolerance; and no refueling station exceeds its maximal capacity.
The inclusion of capacity constraints improves the real-world applicability of the model at the expense of increasing the hardness of an already theoretically NP-hard problem.
We devised two optimization methods to solve the \rslprc and showed that our solvers can still find exact solutions for instances of real-world scope.
However, solving with capacity constraints can increase the solution time by orders of magnitude compared to the uncapacitated version.
For instances with stricter capacity constraints, we show the branch-and-price formulation finds solutions that have $20\%$ smaller optimality gap and require $30\%$ fewer refueling stations than the solutions provided by branch-and-cut.
When the capacity constraints are looser, we show the branch-and-cut formulation finds optimal solutions faster than branch-and-price for smaller and less complex instances.
In addition to our method's ability to handle capacity constraints, we presented a faster integer separation method than the one used in \cite{Arslan2019} and introduced a primal heuristic to find integer solutions faster.

Future work on the capacity aspect of the problem may be necessary to advance the applicability of the model.
For instance, the capacities may not be hard constraints, but could be considered a variable that incurs an additional (potentially non-linear) cost when increased.

\clearpage


\bibliographystyle{abbrvnat}
\bibliography{literature}

\clearpage
\appendix

\section{Primal Heuristic} \label{app:heuristic}

\begin{algorithm}
\caption{Set costs and get constrained shortest path (CSP)}\label{alg:getCSP}
\begin{algorithmic}[1]
\Procedure{getCSP}{$q$, $k$, $\bar{z}$}
\State $l_v^q \gets \infty, \; \forall v \in \stations$
\ForAll{$v \in \V_q$}
    \If{$f_q + k_v \leq \kappa_v$} \Comment{$v$ can satisfy demand from \od\ pair $q$.}
        \If{$k_v > 0$} 
            $l_v^q \gets 0$ \Comment{$v$ already built by another \od\ pair.}
        \Else
            \State $l_v^q \gets 1 - \bar{z}_v^q$ \Comment{Cost of $v$ is related to LP relaxation.}
        \EndIf
    \Else
        \State $l_v^q \gets \infty$ \Comment{$v$ cannot satisfy demand from \od\ pair $q$.}
    \EndIf
\EndFor
\State $P_q \gets \textsc{CSP}(\G_q, u_q, l_v^q)$ \Comment{Shortest path with transit time upper bound $u_q$}
\State \Return{$P_q$}
\EndProcedure
\end{algorithmic}
\end{algorithm}

\begin{algorithm}
\caption{Resolve violated capacity constraints}\label{alg:resolve}
\begin{algorithmic}[1]
\Procedure{ResolveInfeasibility}{$P_q$, $k$, $z$, $\textsc{StationQueue}$, $\textsc{O-DQueue}$}
    \ForAll{$v \in P_q \setminus \{s_q, t_q\}$}
        \If{$k_v + f_q \leq \kappa_v$}
            \State $k_v \gets k_v + f_q$ \Comment{Increase demand at $v$ with \od\ pair $q$.}
            \State $z_v^q \gets 1$ \Comment{Assign $v$ to $q$ in integer solution.}
            \State $\textsc{StationQueue}(v).\text{enqueue}(q)$
        \Else
            \While{$k_v > \kappa_v$} \Comment{Reassign \od\ pairs until capacity constraints are feasible}
                \State $q' \gets \textsc{StationQueue}(v).\text{dequeue}()$
                \ForAll{$v \in P_{q'} \setminus \{s_{q'}, t_{q'}\}$}
                    \State $k_v \gets k_v - f_{q'}$
                    \State $z_v^{q'} \gets 0$ \Comment{Reassign $v$ from $q$ in integer solution.}
                \EndFor
                \State $\textsc{O-DQueue}.\text{enqueue}(q')$
            \EndWhile
        \EndIf
    \EndFor
\EndProcedure
\end{algorithmic}
\end{algorithm}

\section{Plots for all results}
\label{app:plots-ca}

\clearpage

\begin{figure*}
\centering
{
\begin{minipage}{.8\linewidth}\centering
\subfloat[Optimality gap]{\includegraphics[width=.49\linewidth]{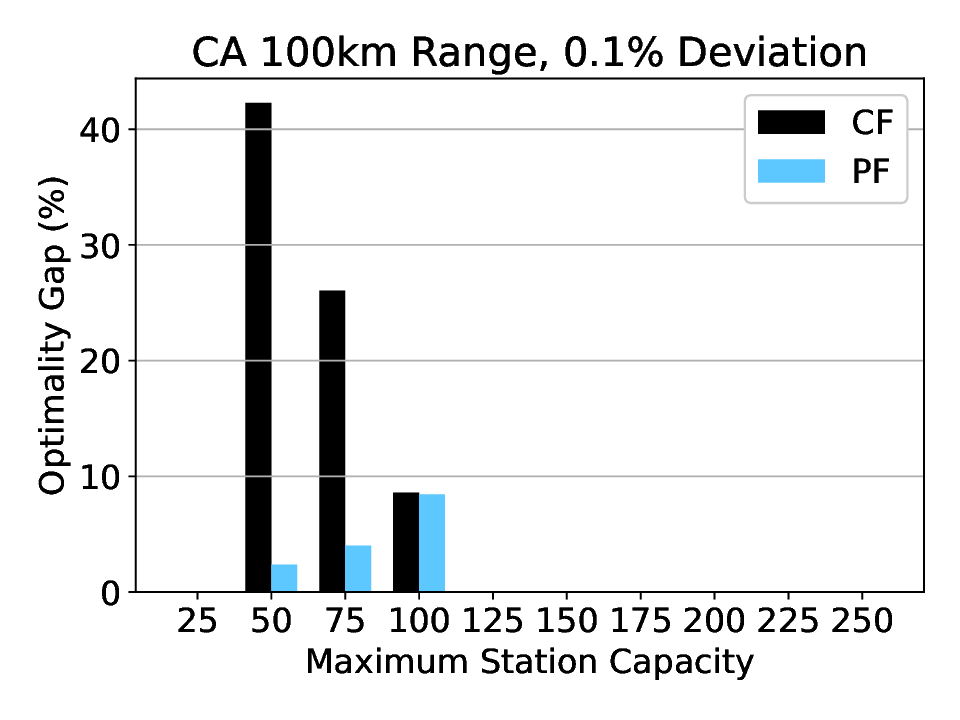} \label{fig:gap-ca_100_0}}
\subfloat[Solution time]{\includegraphics[width=.49\linewidth]{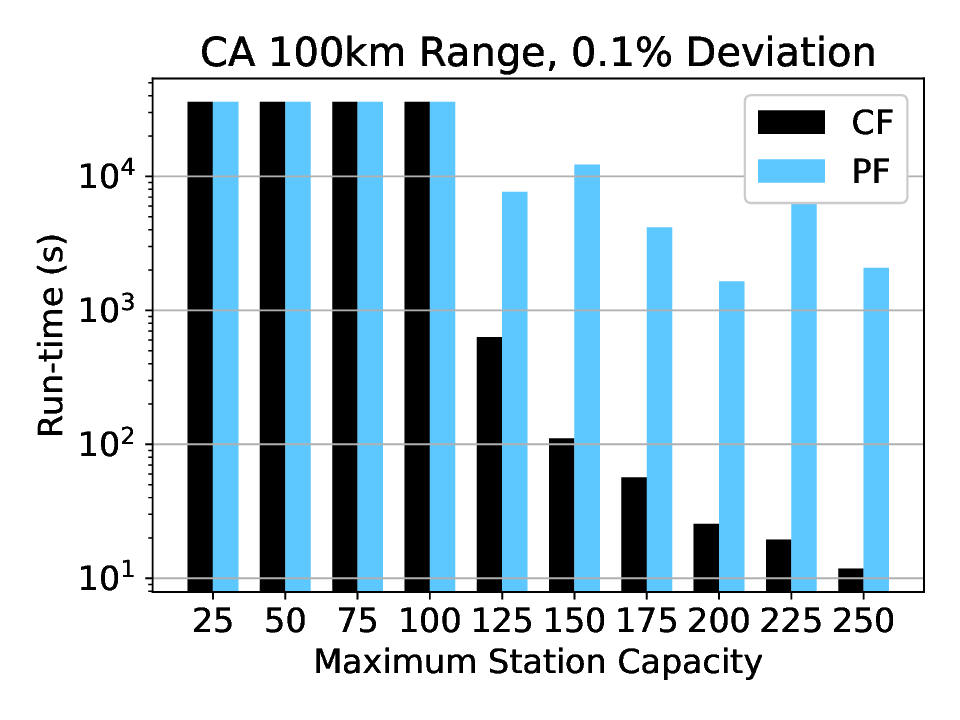}\label{fig:time-ca_100_0}}
\newline
\subfloat[Utilization]{\includegraphics[width=.49\linewidth]{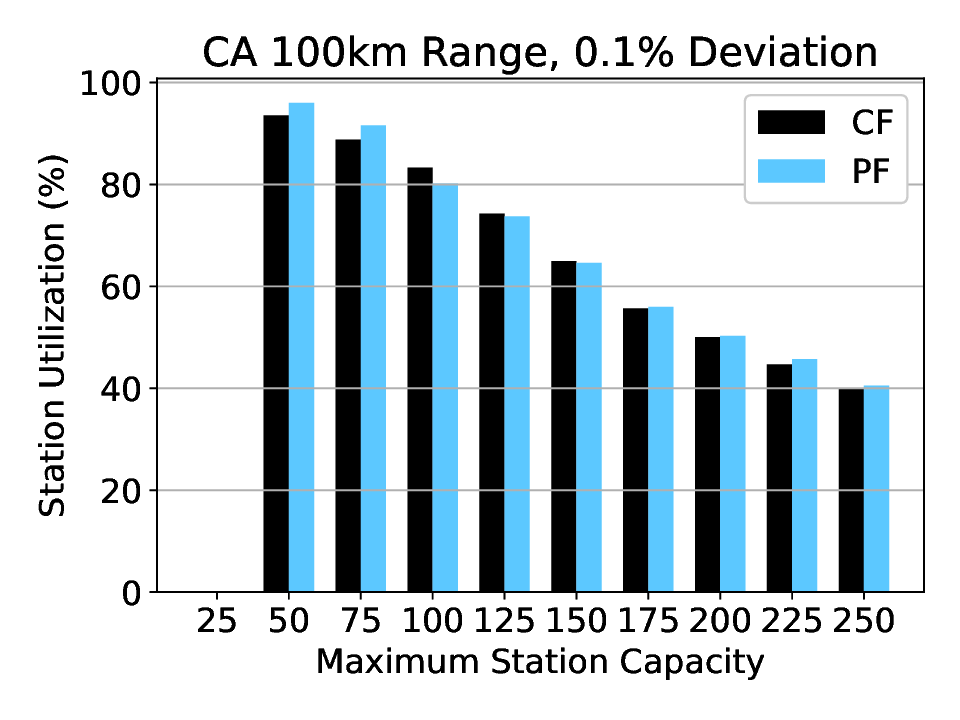} \label{fig:utilization-ca_100_0}}
\subfloat[Objective Value]{\includegraphics[width=.49\linewidth]{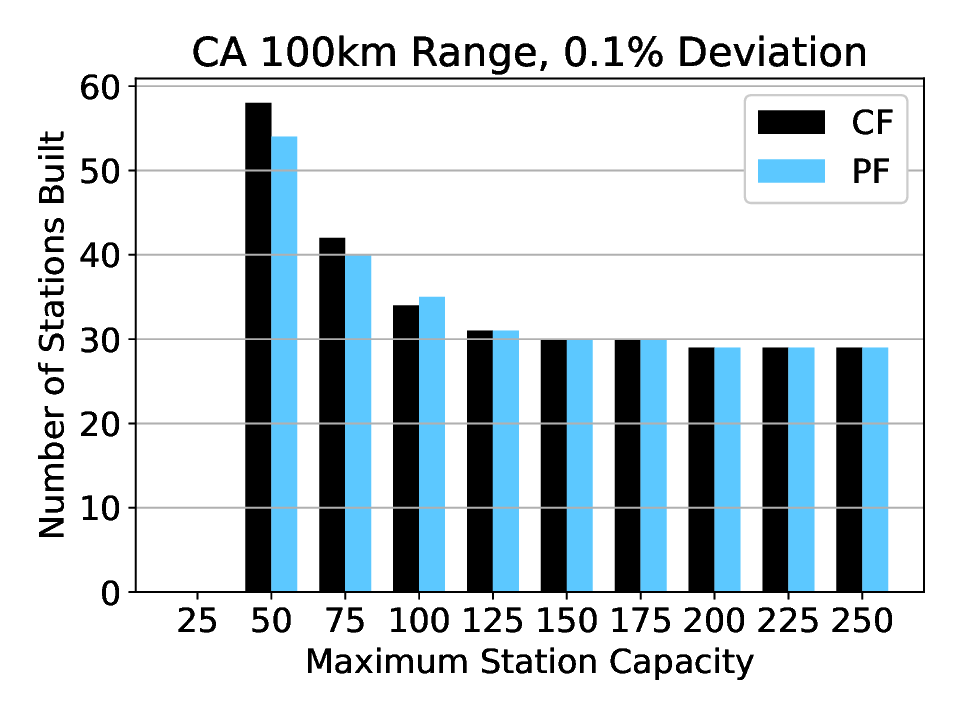}\label{fig:objval-ca_100_0}}
\end{minipage}
}
\caption{Results for the CA road network\label{fig:results_ca_100_0}. The plot shows results with deviation tolerance $\lambda = 0.1\%$ and vehicle range $\rangemax = 100$km for \eqref{eq:cut} and \eqref{eq:path}.}
\end{figure*}

\begin{figure*}
\centering
{
\begin{minipage}{.8\linewidth}\centering
\subfloat[Optimality gap]{\includegraphics[width=.49\linewidth]{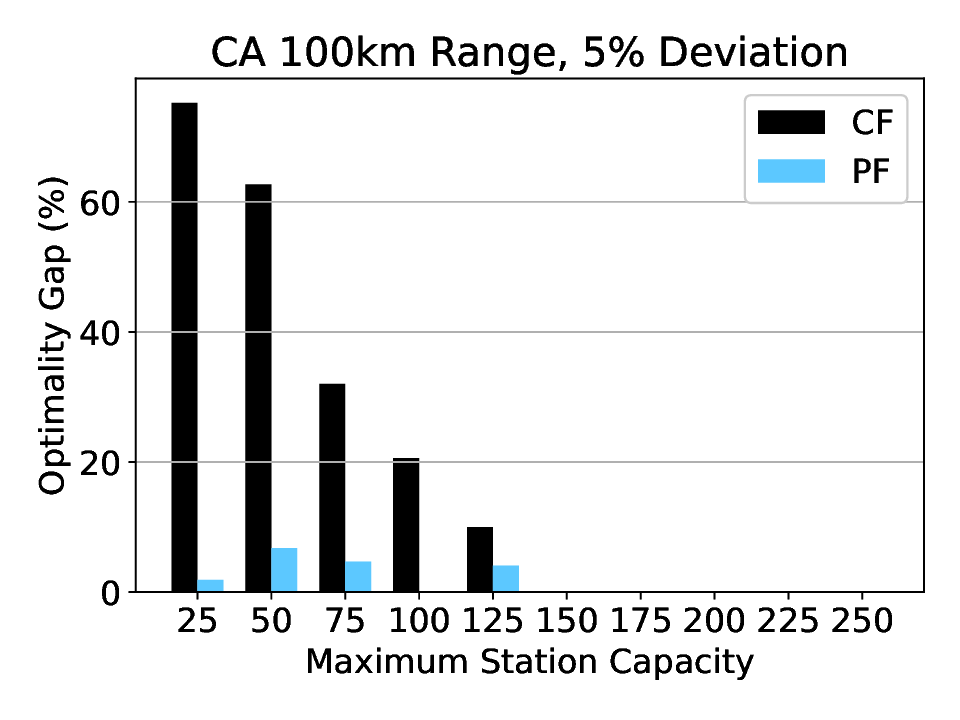} \label{fig:gap-ca_100_5}}
\subfloat[Solution time]{\includegraphics[width=.49\linewidth]{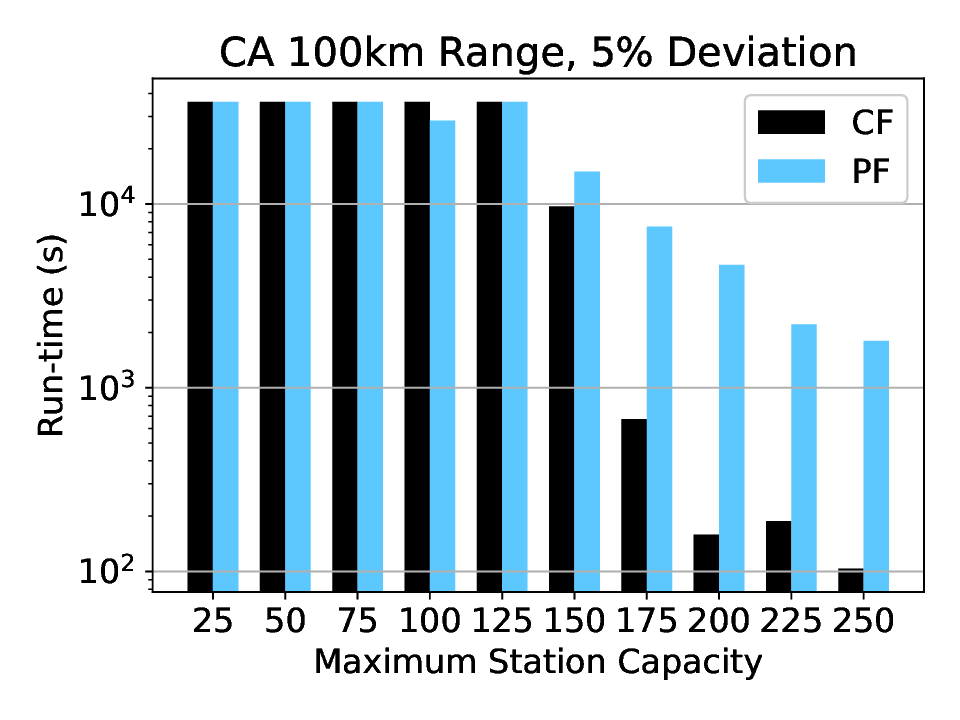}\label{fig:time-ca_100_5}}
\newline
\subfloat[Utilization]{\includegraphics[width=.49\linewidth]{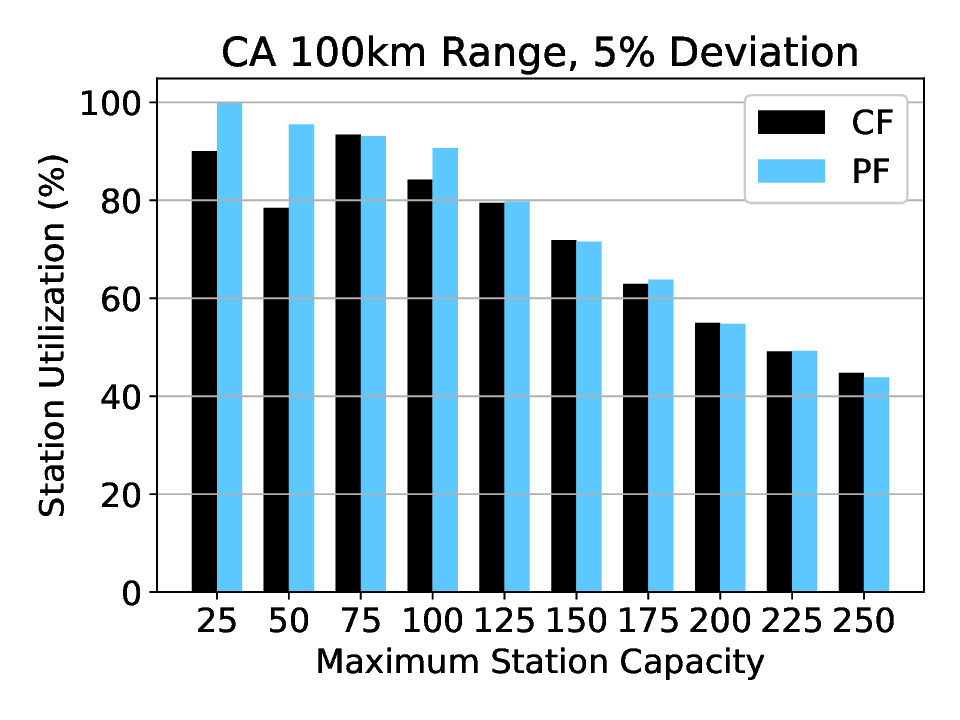} \label{fig:utilization-ca_100_5}}
\subfloat[Objective Value]{\includegraphics[width=.49\linewidth]{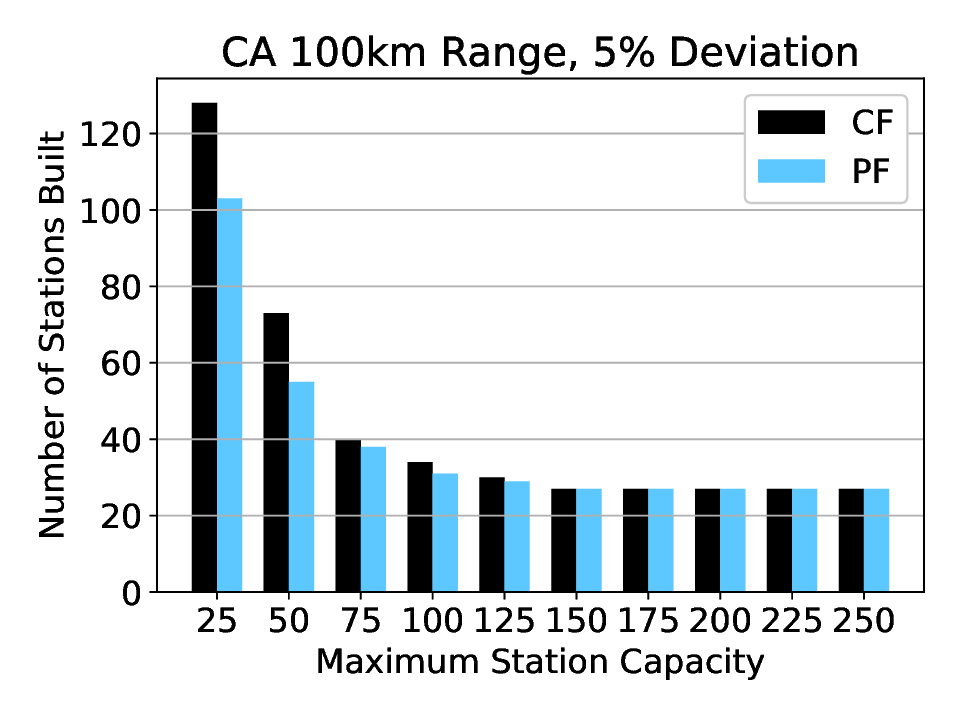}\label{fig:objval-ca_100_5}}
\end{minipage}
}
\caption{Results for the CA road network\label{fig:results_ca_100_5}. The plot shows results with deviation tolerance $\lambda = 5.0\%$ and vehicle range $\rangemax = 100$km for \eqref{eq:cut} and \eqref{eq:path}.}
\end{figure*}

\begin{figure*}
\centering
{
\begin{minipage}{.8\linewidth}\centering
\subfloat[Optimality gap]{\includegraphics[width=.49\linewidth]{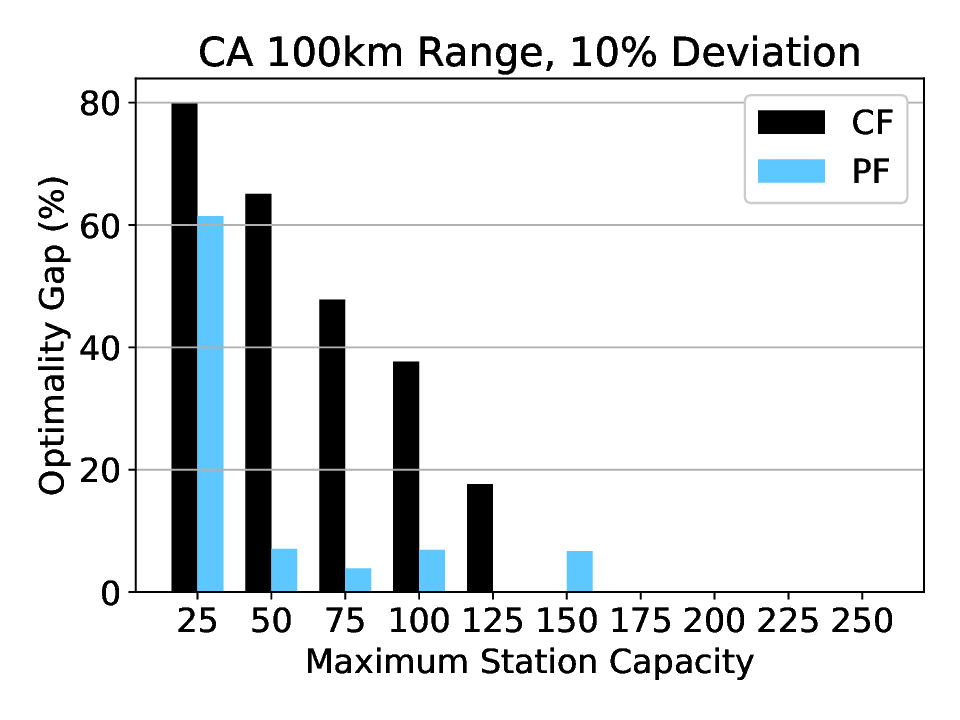} \label{fig:gap-ca_100_10}}
\subfloat[Solution time]{\includegraphics[width=.49\linewidth]{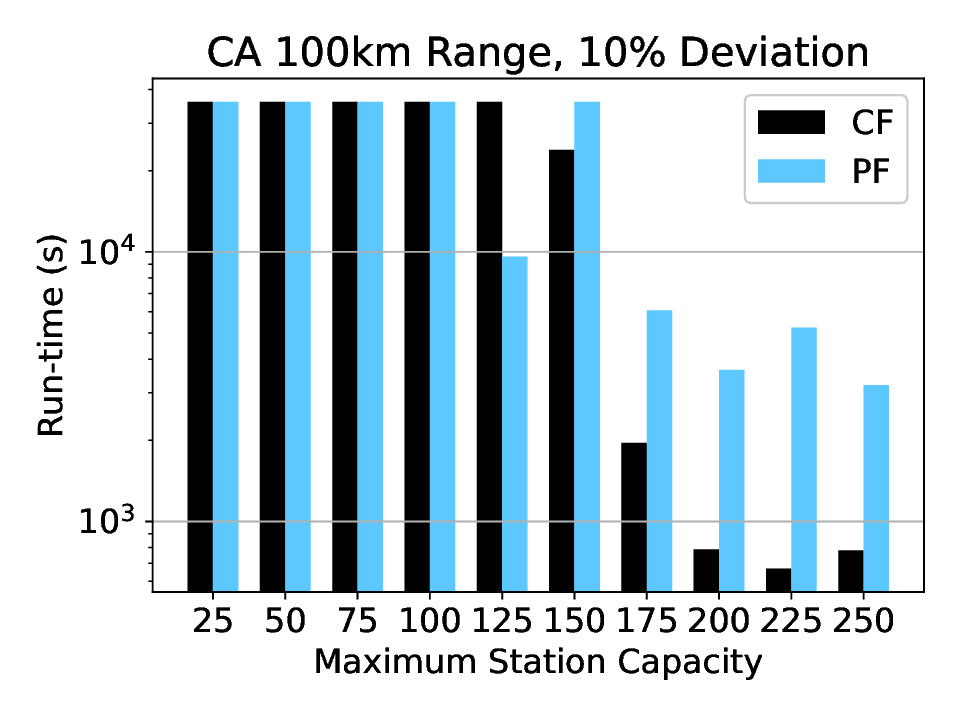}\label{fig:time-ca_100_10}}
\newline
\subfloat[Utilization]{\includegraphics[width=.49\linewidth]{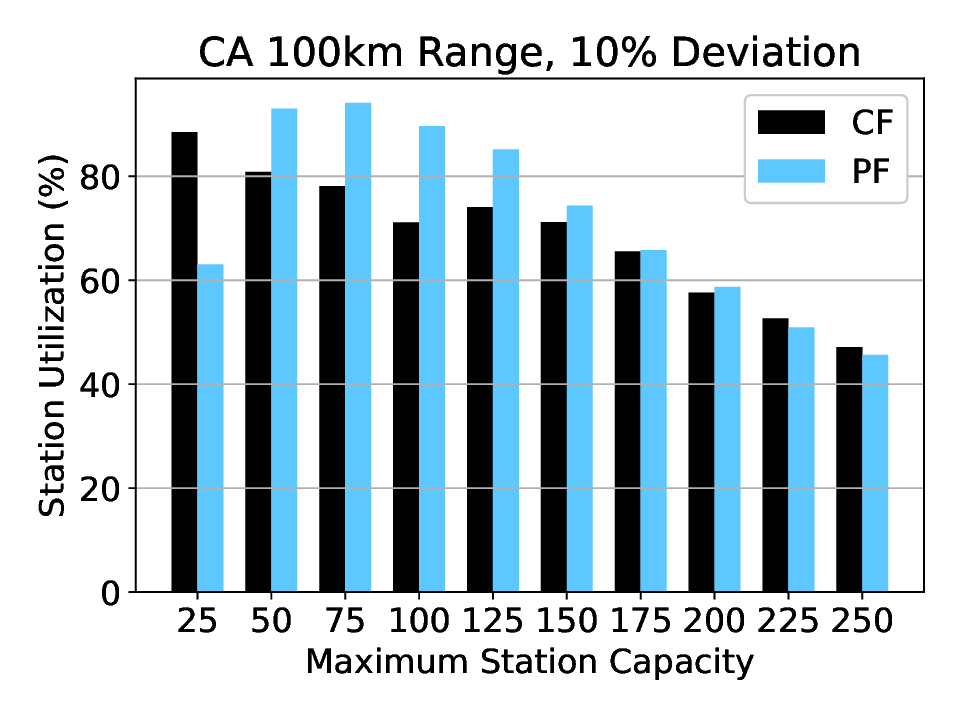} \label{fig:utilization-ca_100_10}}
\subfloat[Objective Value]{\includegraphics[width=.49\linewidth]{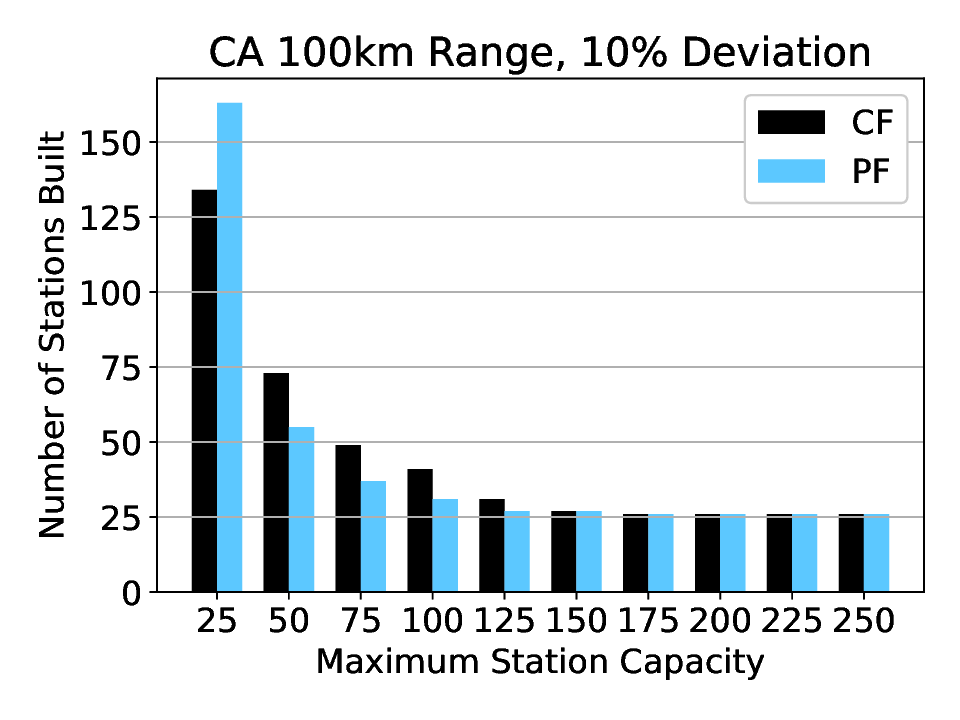}\label{fig:objval-ca_100_10}}
\end{minipage}
}
\caption{Results for the CA road network\label{fig:results_ca_100_10}. The plot shows results with deviation tolerance $\lambda = 10.0\%$ and vehicle range $\rangemax = 100$km for \eqref{eq:cut} and \eqref{eq:path}.}
\end{figure*}

\begin{figure*}
\centering
{
\begin{minipage}{.8\linewidth}\centering
\subfloat[Optimality gap]{\includegraphics[width=.49\linewidth]{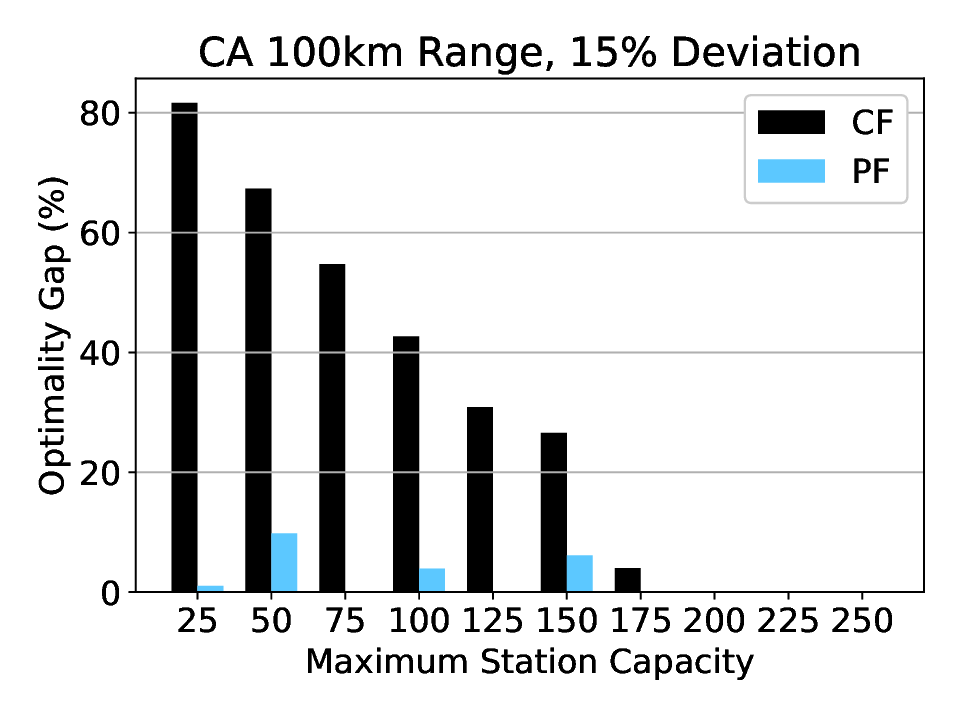} \label{fig:gap-ca_100_15}}
\subfloat[Solution time]{\includegraphics[width=.49\linewidth]{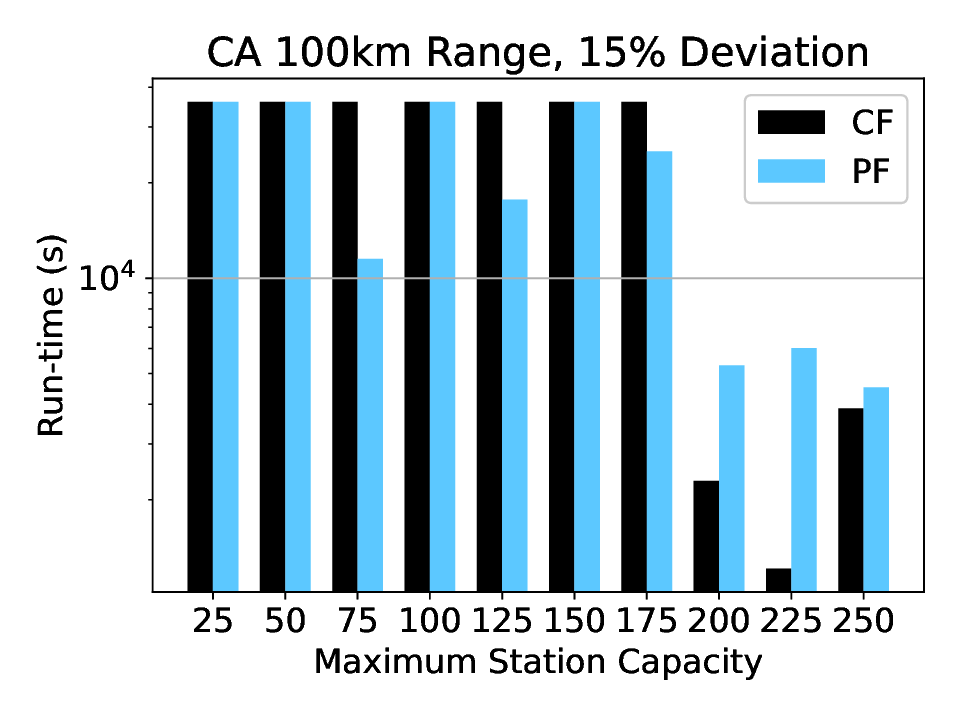}\label{fig:time-ca_100_15}}
\newline
\subfloat[Utilization]{\includegraphics[width=.49\linewidth]{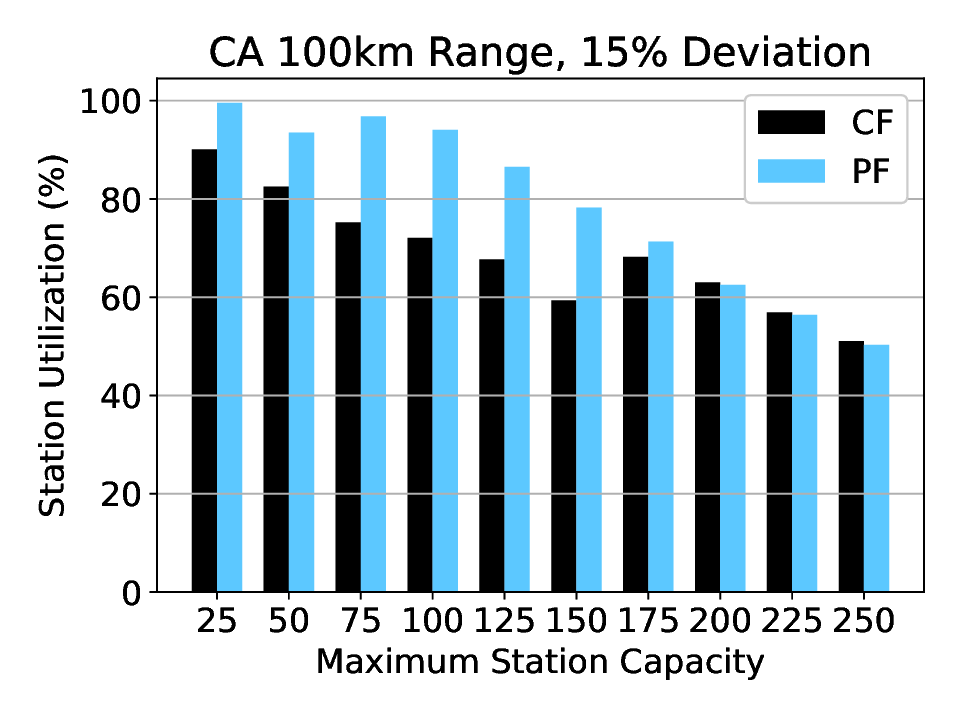} \label{fig:utilization-ca_100_15}}
\subfloat[Objective Value]{\includegraphics[width=.49\linewidth]{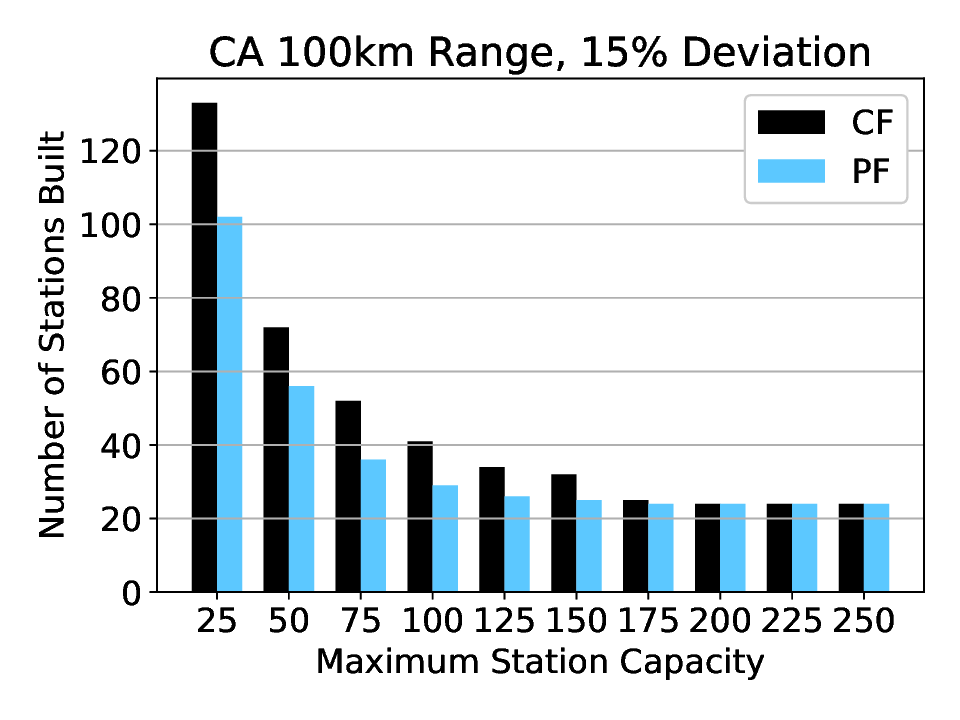}\label{fig:objval-ca_100_15}}
\end{minipage}
}
\caption{Results for the CA road network\label{fig:results_ca_100_15}. The plot shows results with deviation tolerance $\lambda = 15.0\%$ and vehicle range $\rangemax = 100$km for \eqref{eq:cut} and \eqref{eq:path}.}
\end{figure*}

\begin{figure*}
\centering
{
\begin{minipage}{.8\linewidth}\centering
\subfloat[Optimality gap]{\includegraphics[width=.49\linewidth]{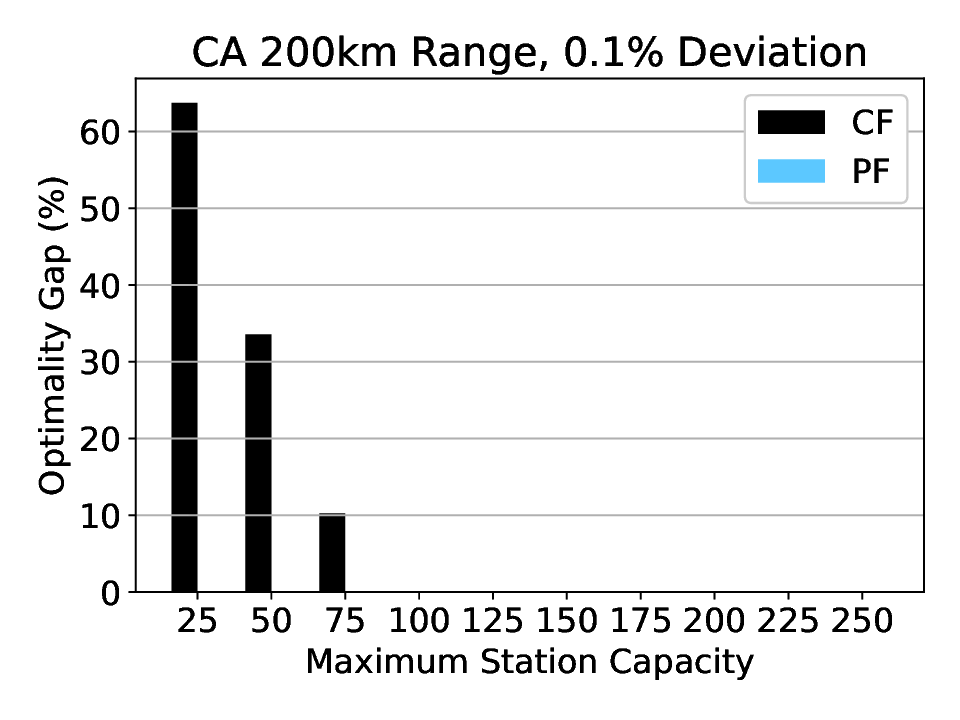} \label{fig:gap-ca_200_0}}
\subfloat[Solution time]{\includegraphics[width=.49\linewidth]{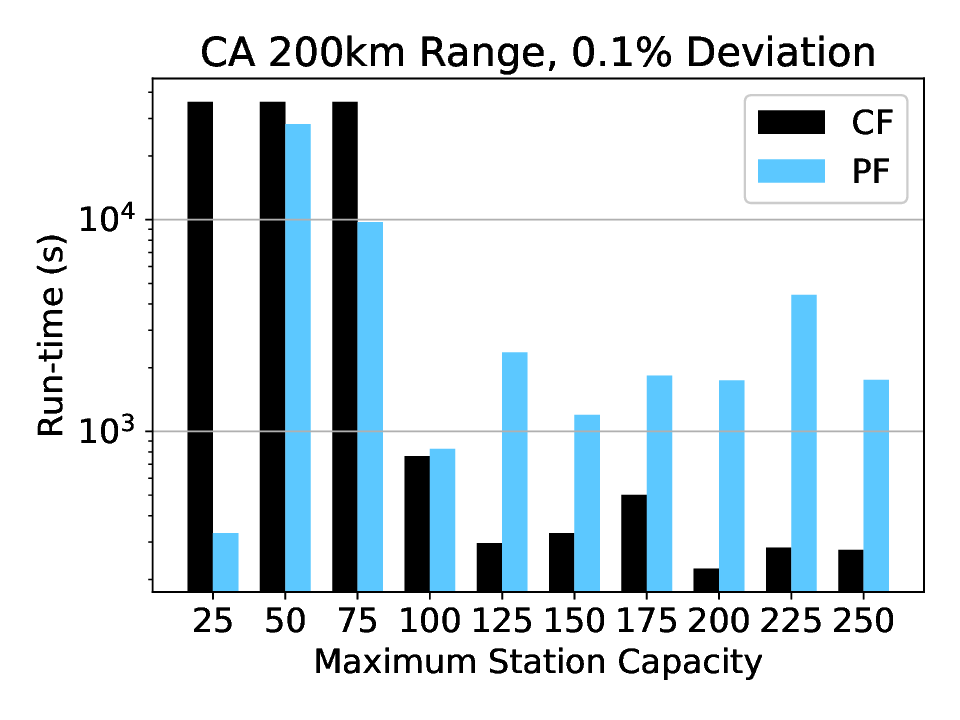}\label{fig:time-ca_200_0}}
\newline
\subfloat[Utilization]{\includegraphics[width=.49\linewidth]{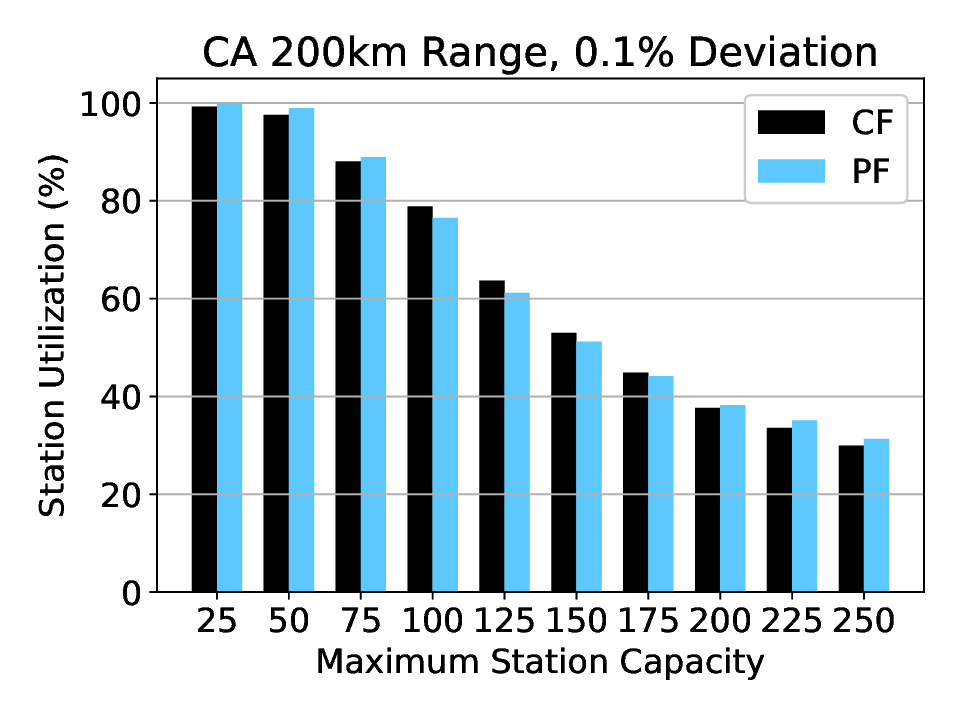} \label{fig:utilization-ca_200_0}}
\subfloat[Objective Value]{\includegraphics[width=.49\linewidth]{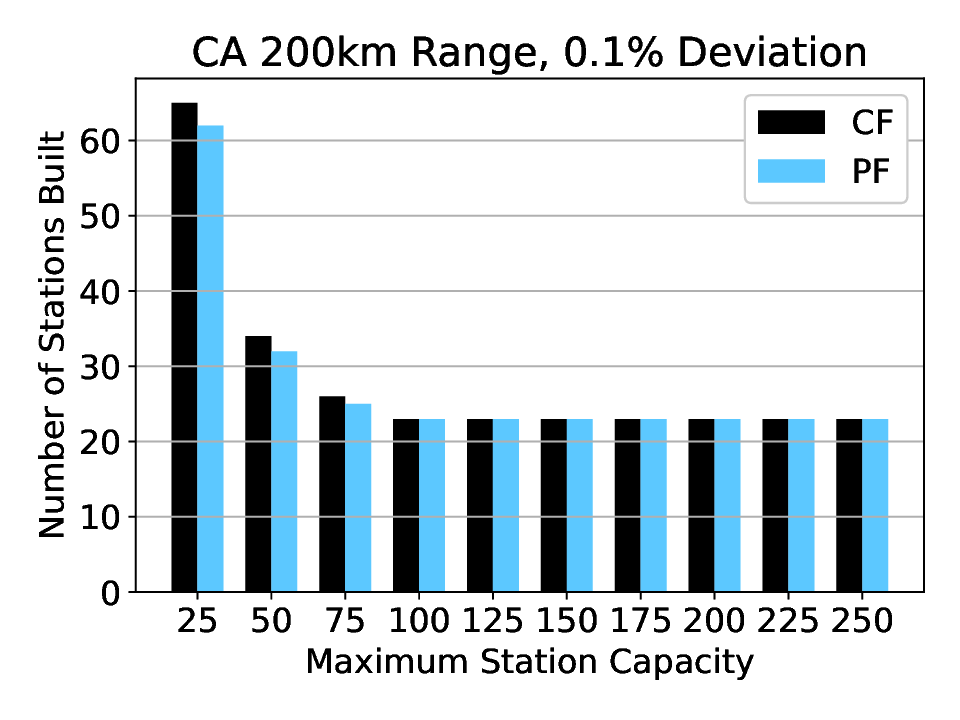}\label{fig:objval-ca_200_0}}
\newline
\end{minipage}
}
\caption{Results for the CA road network\label{fig:results_ca_200_0}. The plot shows results with deviation tolerance $\lambda = 0.1\%$ and vehicle range $\rangemax = 200$km for \eqref{eq:cut} and \eqref{eq:path}.}
\end{figure*}

\begin{figure*}
\centering
{
\begin{minipage}{.8\linewidth}\centering
\subfloat[Optimality gap]{\includegraphics[width=.49\linewidth]{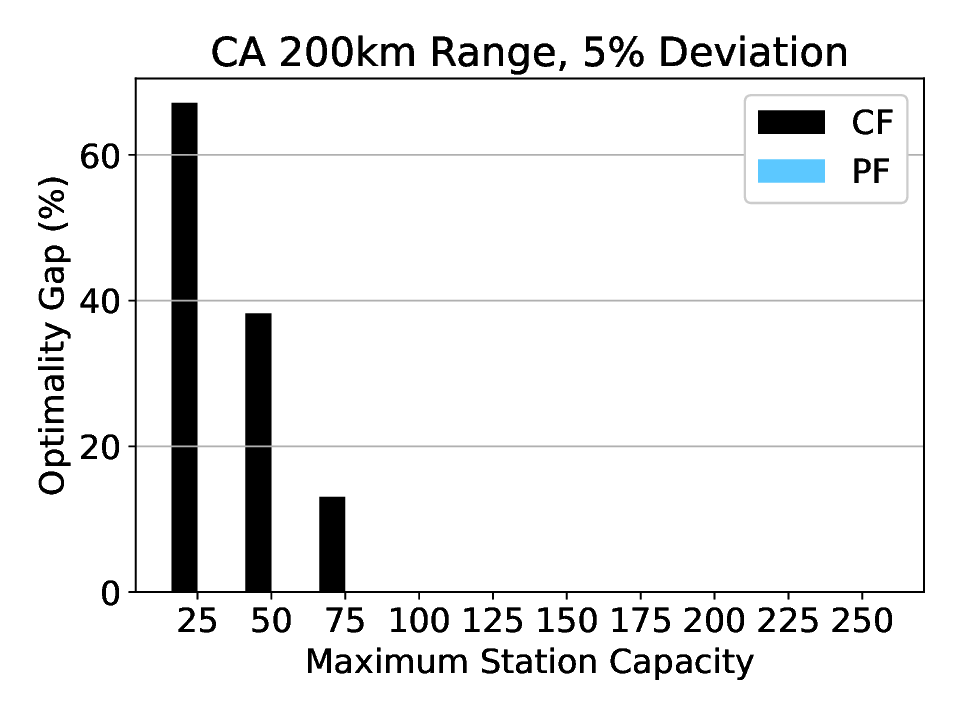} \label{fig:gap-ca_200_5}}
\subfloat[Solution time]{\includegraphics[width=.49\linewidth]{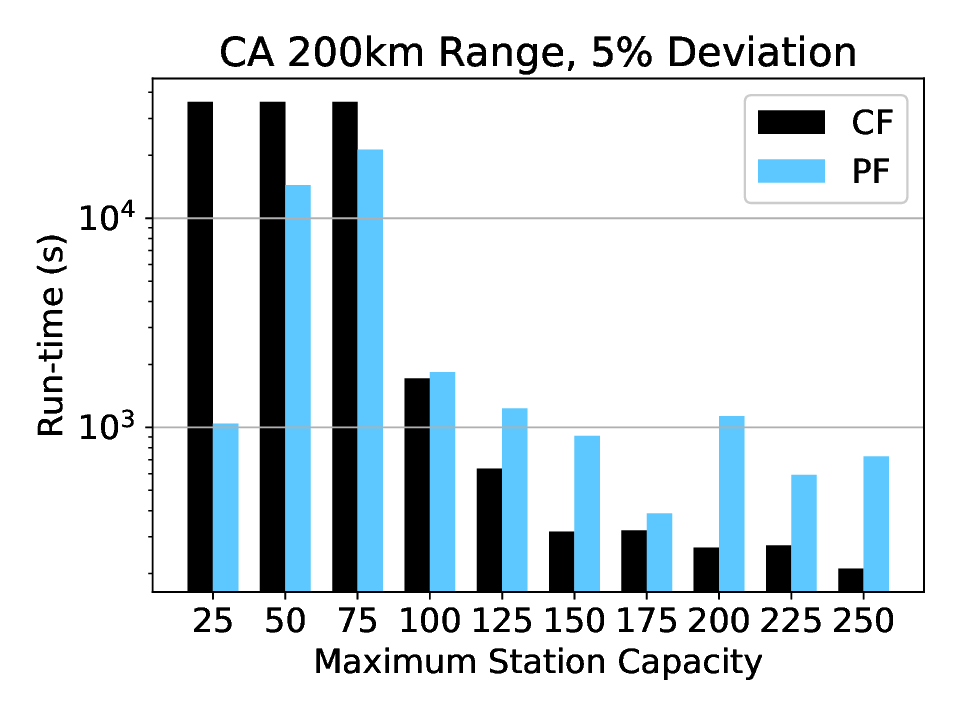}\label{fig:time-ca_200_5}}
\newline
\subfloat[Utilization]{\includegraphics[width=.49\linewidth]{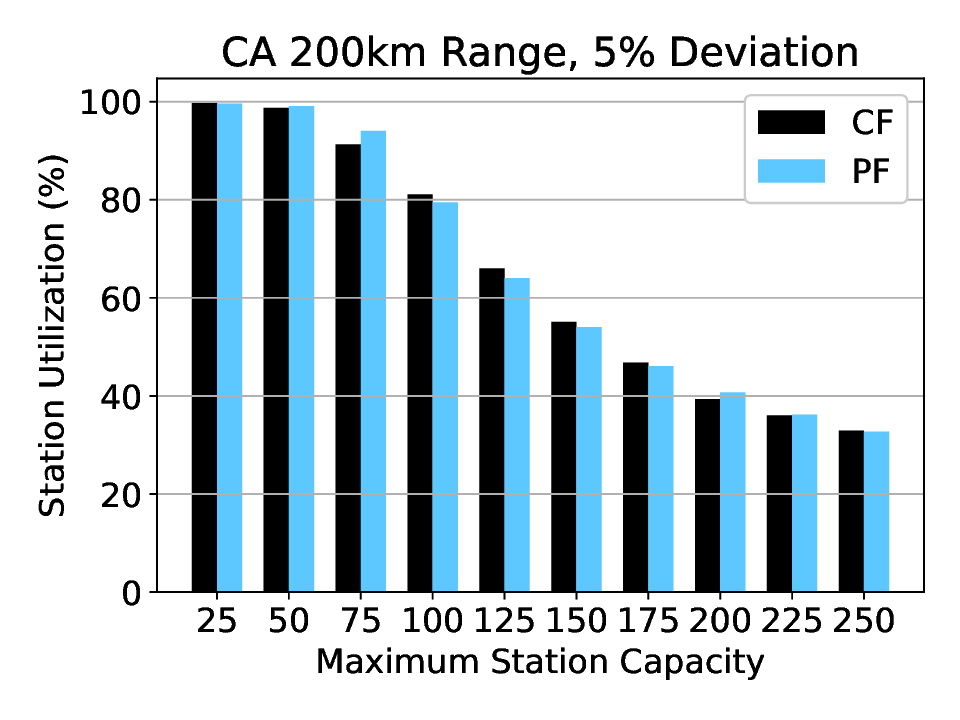} \label{fig:utilization-ca_200_5}}
\subfloat[Objective Value]{\includegraphics[width=.49\linewidth]{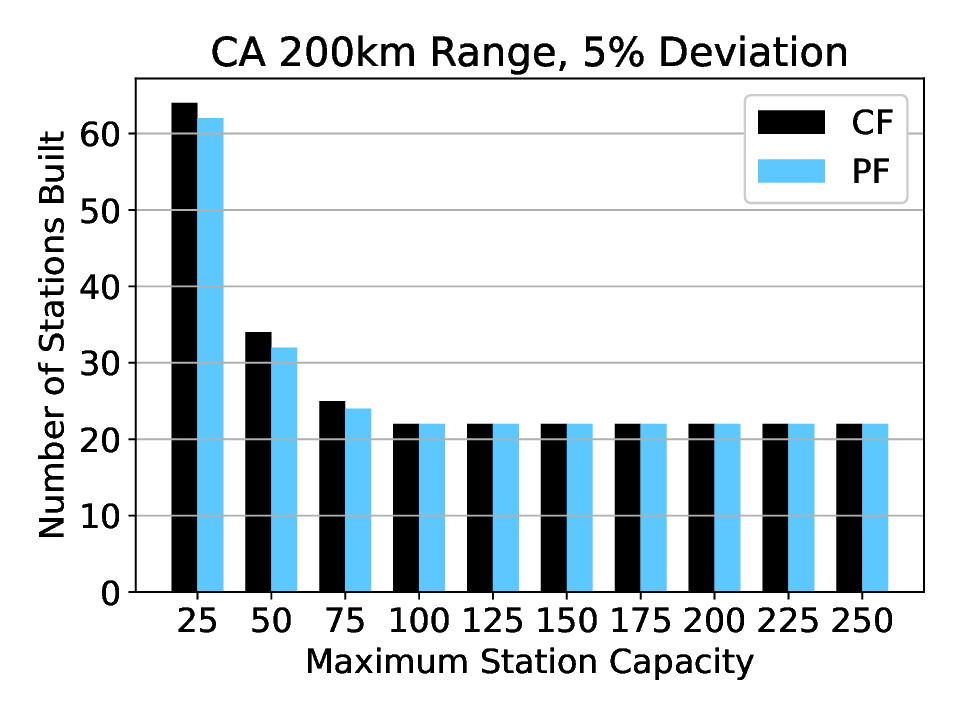}\label{fig:objval-ca_200_5}}
\end{minipage}
}
\caption{Results for the CA road network\label{fig:results_ca_200_5}. The plot shows results with deviation tolerance $\lambda = 5.0\%$ and vehicle range $\rangemax = 200$km for \eqref{eq:cut} and \eqref{eq:path}.}
\end{figure*}

\begin{figure*}
\centering
{
\begin{minipage}{.8\linewidth}\centering
\subfloat[Optimality gap]{\includegraphics[width=.49\linewidth]{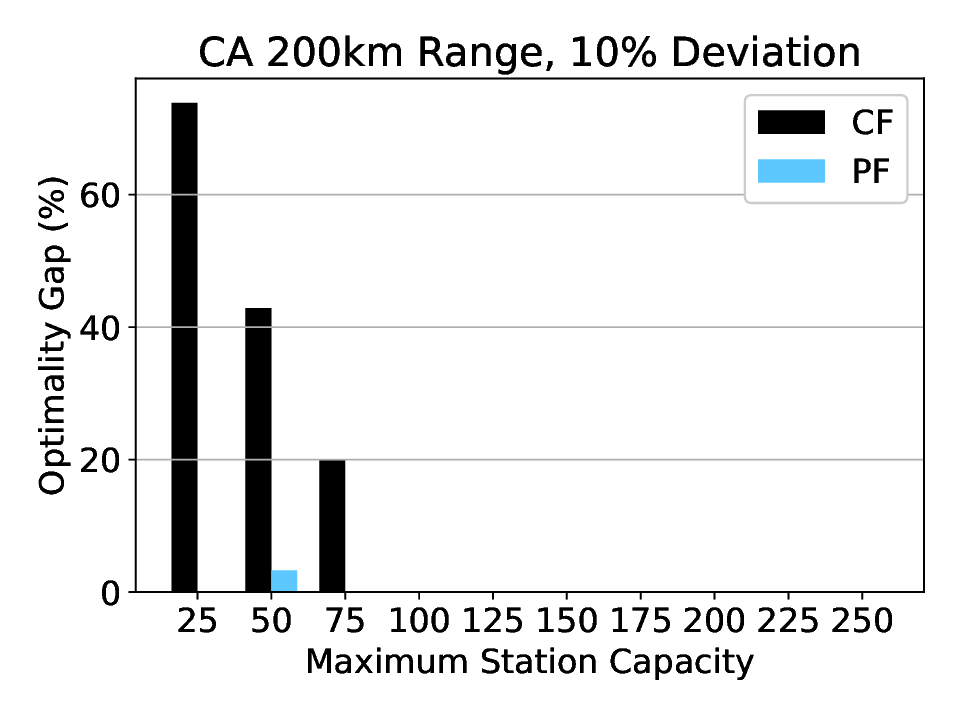} \label{fig:gap-ca_200_10}}
\subfloat[Solution time]{\includegraphics[width=.49\linewidth]{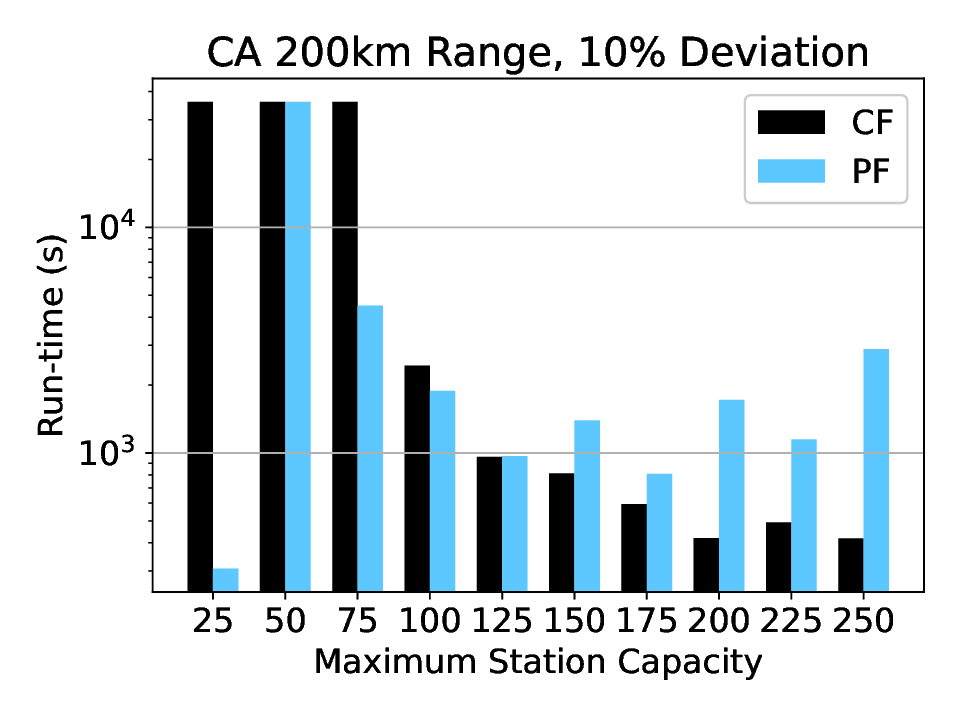}\label{fig:time-ca_200_10}}
\newline
\subfloat[Utilization]{\includegraphics[width=.49\linewidth]{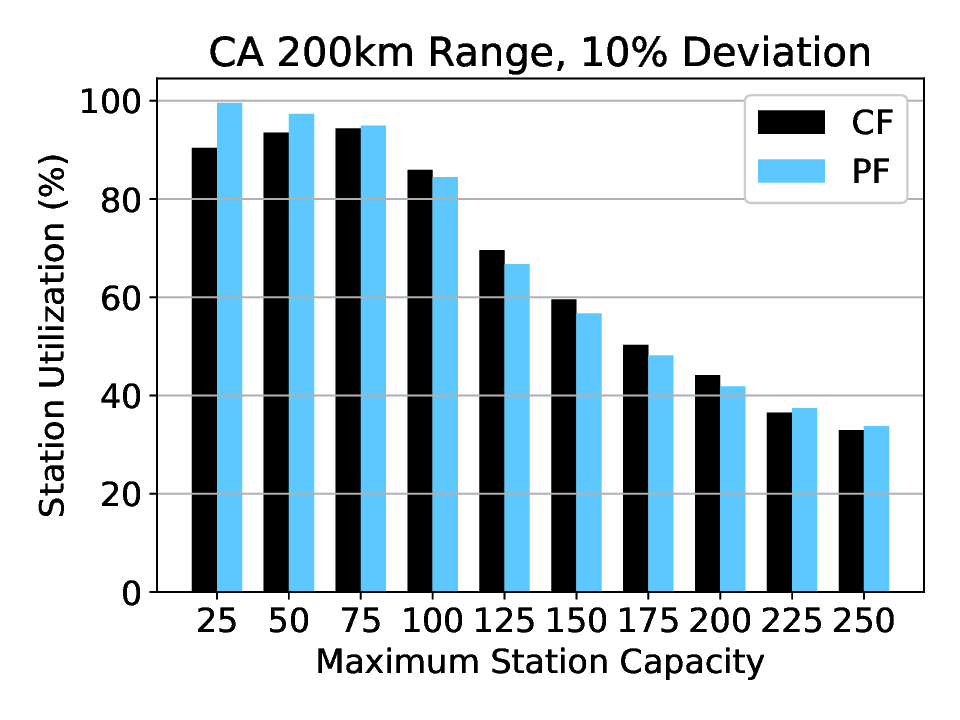} \label{fig:utilization-ca_200_10}}
\subfloat[Objective Value]{\includegraphics[width=.49\linewidth]{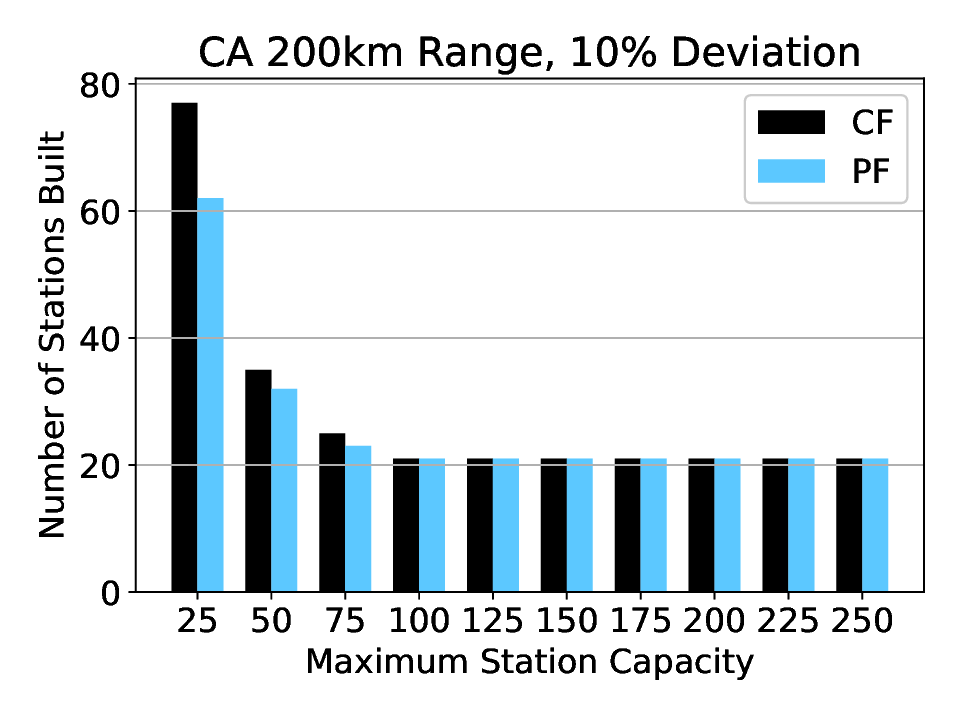}\label{fig:objval-ca_200_10}}
\end{minipage}
}
\caption{Results for the CA road network\label{fig:results_ca_200_10}. The plot shows results with deviation tolerance $\lambda = 10.0\%$ and vehicle range $\rangemax = 200$km for \eqref{eq:cut} and \eqref{eq:path}.}
\end{figure*}

\begin{figure*}
\centering
{
\begin{minipage}{.8\linewidth}\centering
\subfloat[Optimality gap]{\includegraphics[width=.49\linewidth]{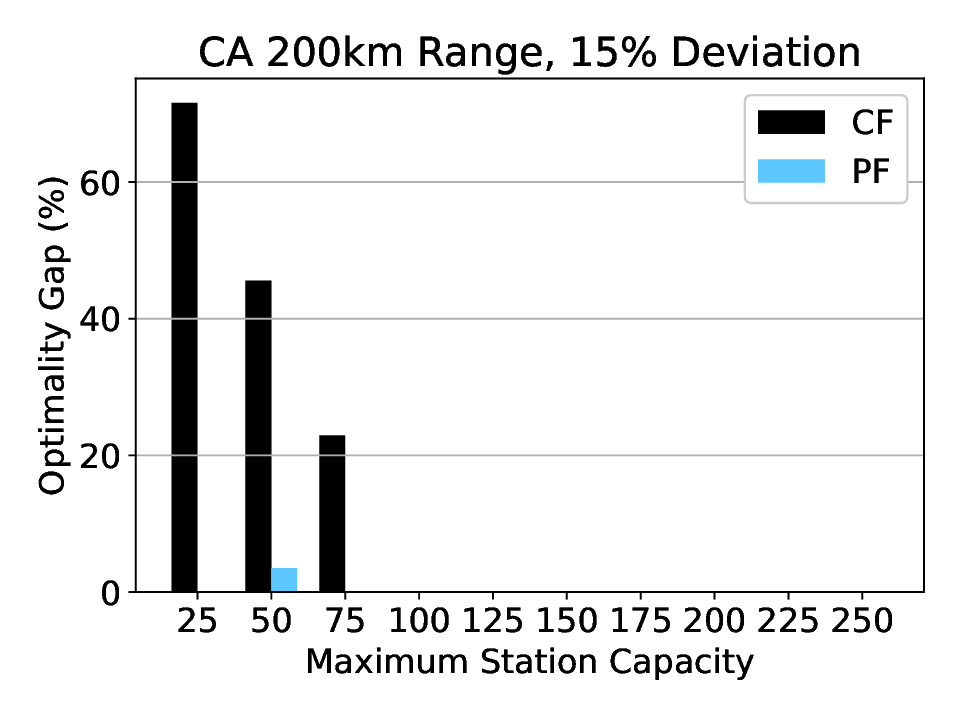} \label{fig:gap-ca_200_15}}
\subfloat[Solution time]{\includegraphics[width=.49\linewidth]{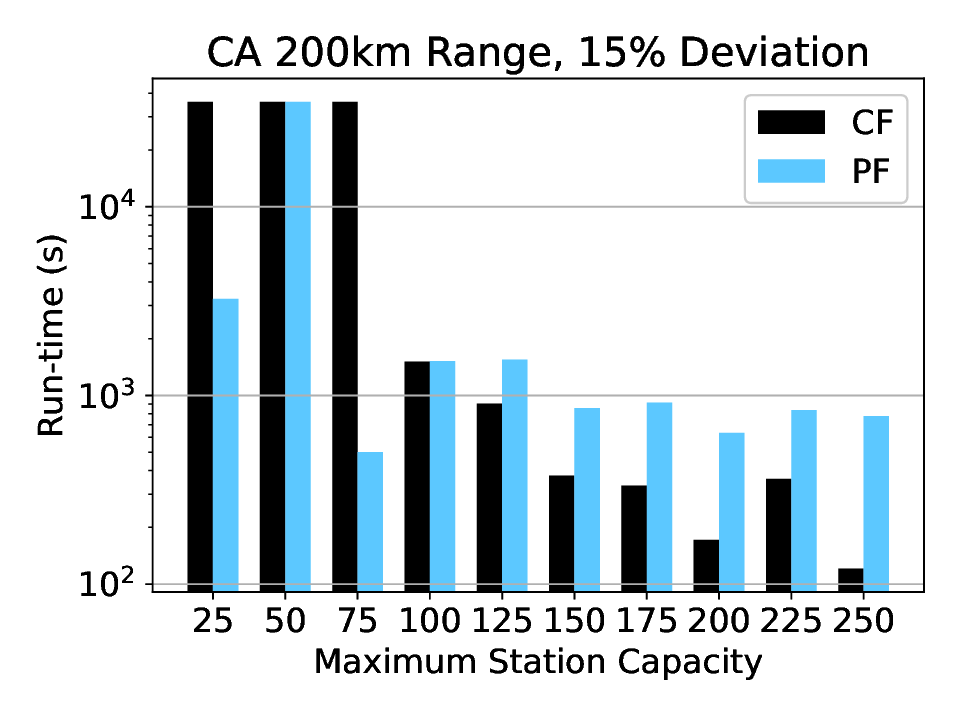}\label{fig:time-ca_200_15}}
\newline
\subfloat[Utilization]{\includegraphics[width=.49\linewidth]{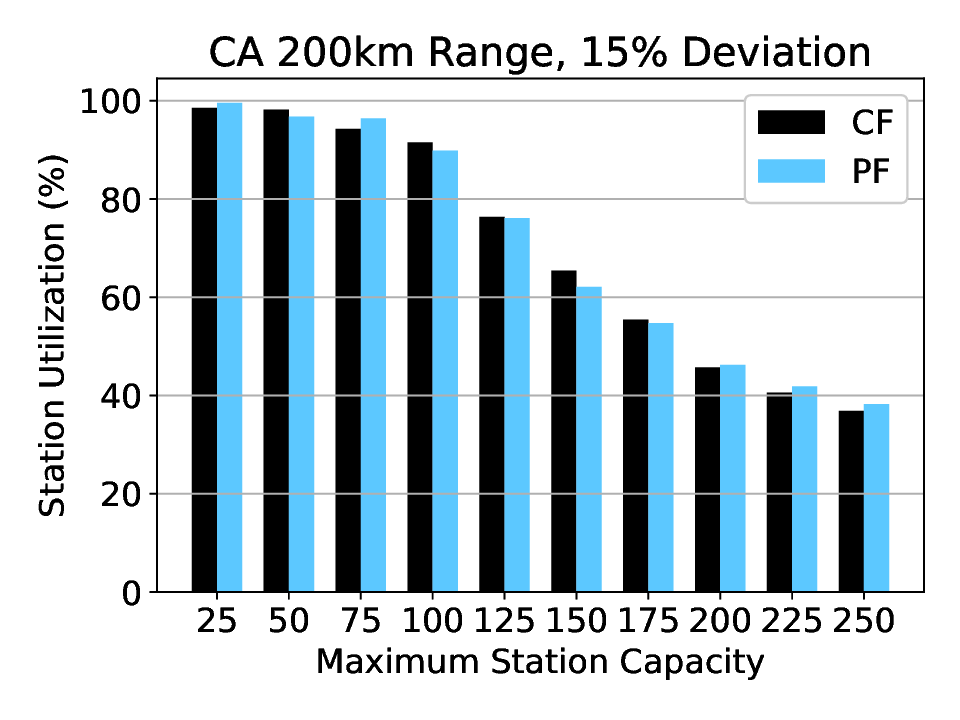} \label{fig:utilization-ca_200_15}}
\subfloat[Objective Value]{\includegraphics[width=.49\linewidth]{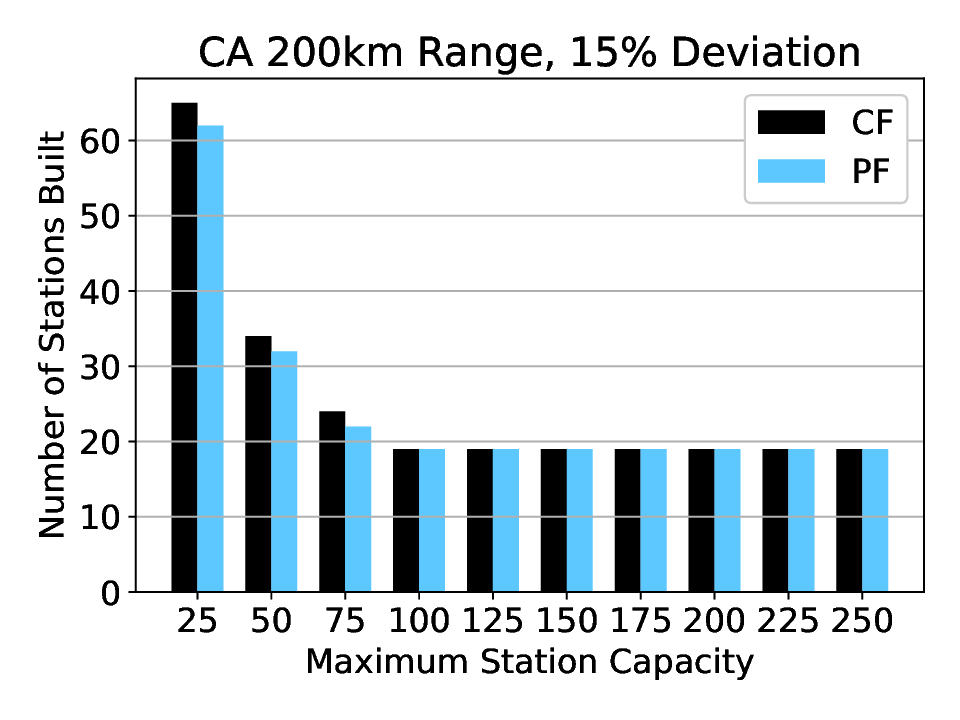}\label{fig:objval-ca_200_15}}
\end{minipage}
}
\caption{Results for the CA road network\label{fig:results_ca_200_15}. The plot shows results with deviation tolerance $\lambda = 15.0\%$ and vehicle range $\rangemax = 200$km for \eqref{eq:cut} and \eqref{eq:path}.}
\end{figure*}

\begin{figure*}[!ht]
\centering
{
\begin{minipage}{.8\linewidth}\centering
\subfloat[Optimality gap]{\includegraphics[width=.49\linewidth]{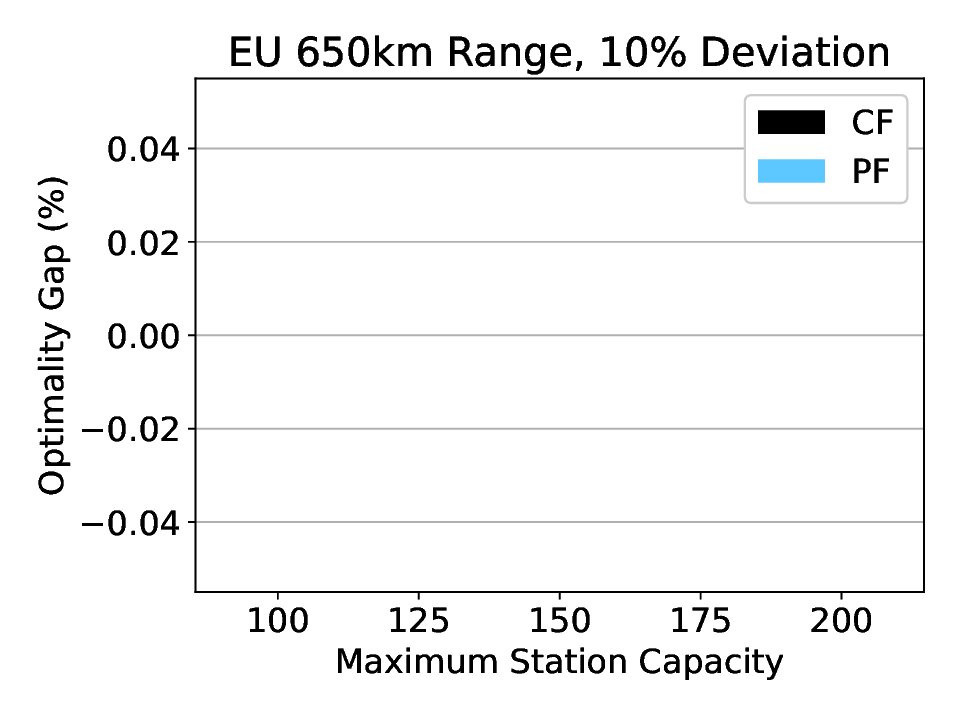} \label{fig:eugap_650_10}}
\subfloat[Solution time]{\includegraphics[width=.49\linewidth]{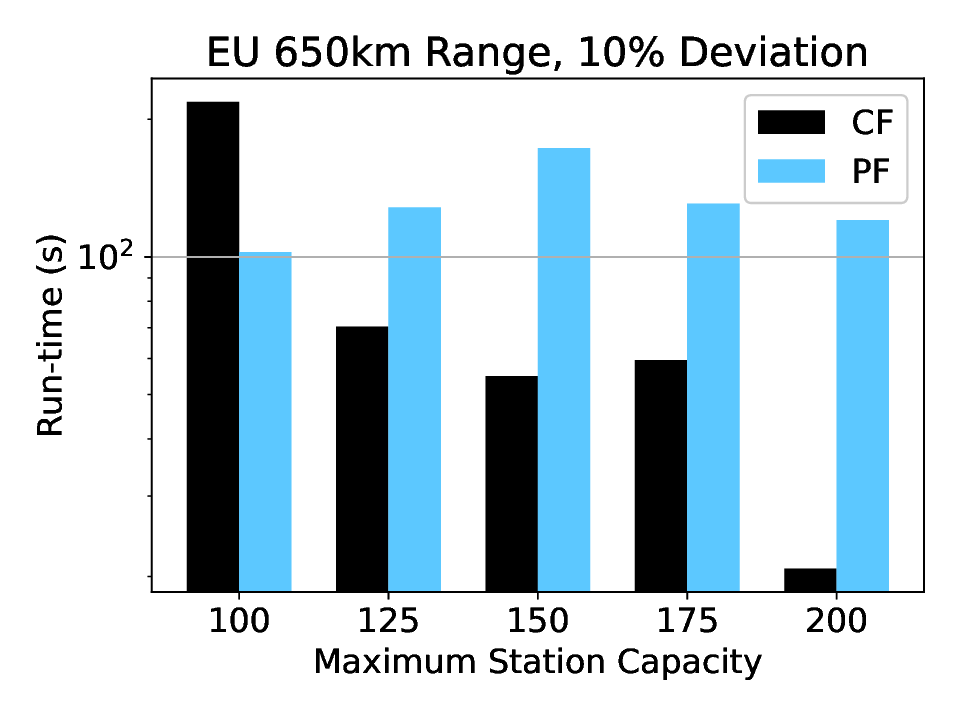}\label{fig:eutime_650_10}}
\newline
\subfloat[Utilization]{\includegraphics[width=.49\linewidth]{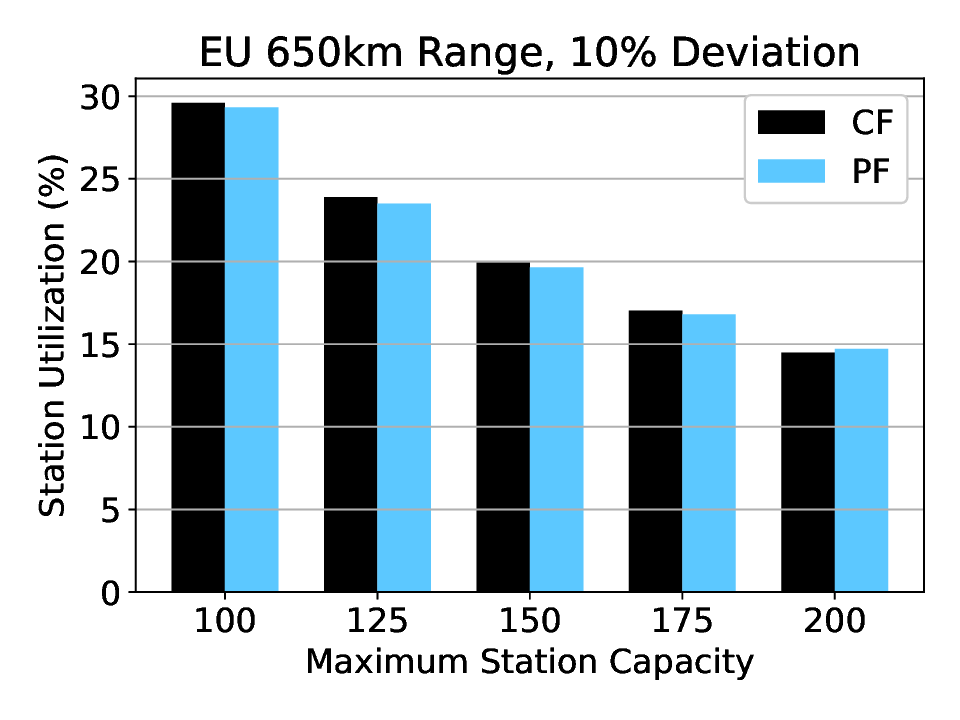} \label{fig:euutilization_650_10}}
\subfloat[Objective Value]{\includegraphics[width=.49\linewidth]{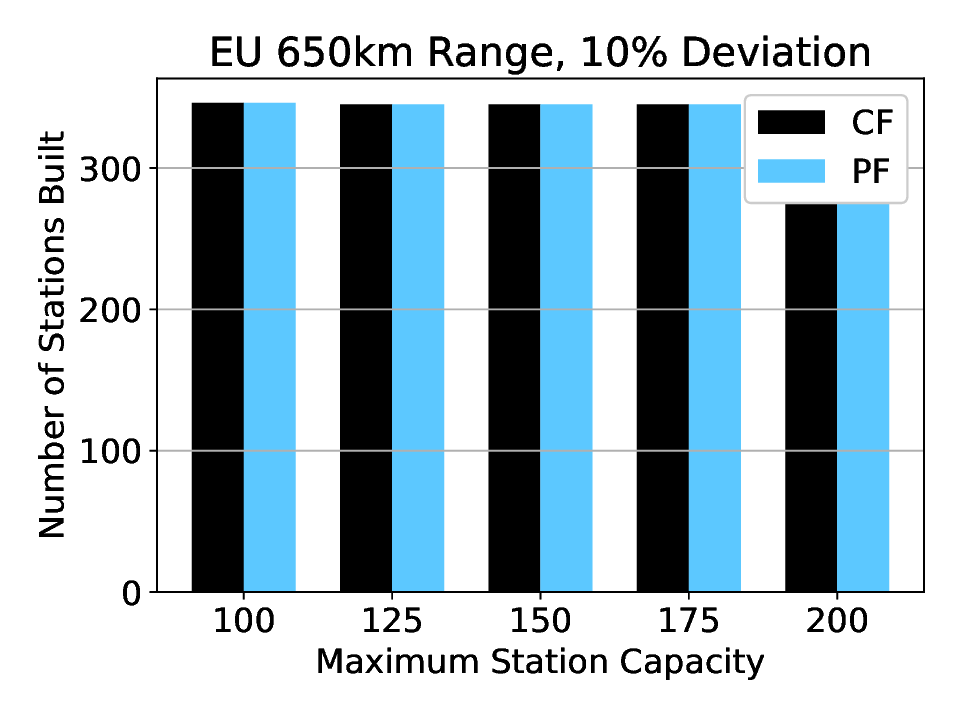}\label{fig:euobjval_650_10}}
\end{minipage}
}
\caption{Results for the EU road network\label{fig:euresults_650_10}. The plot shows results with deviation tolerance $\lambda = 10.0\%$ and vehicle range $\rangemax = 650$km for \eqref{eq:cut} and \eqref{eq:path}.}
\end{figure*}

\begin{figure*}[!ht]
\centering
{
\begin{minipage}{.8\linewidth}\centering
\subfloat[Optimality gap]{\includegraphics[width=.49\linewidth]{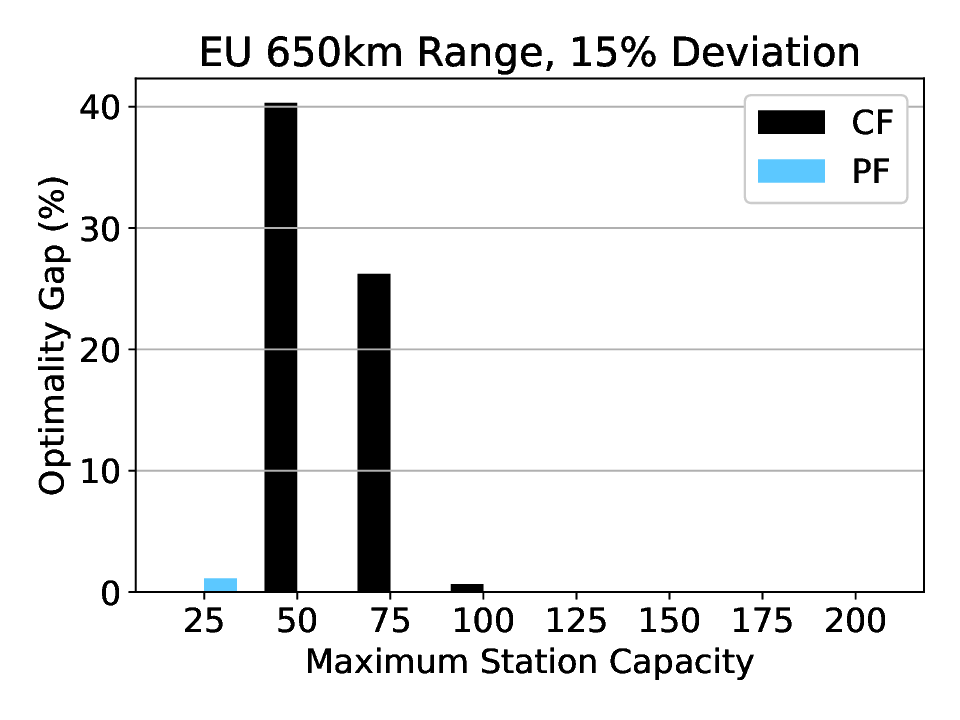} \label{fig:eugap_650_15}}
\subfloat[Solution time]{\includegraphics[width=.49\linewidth]{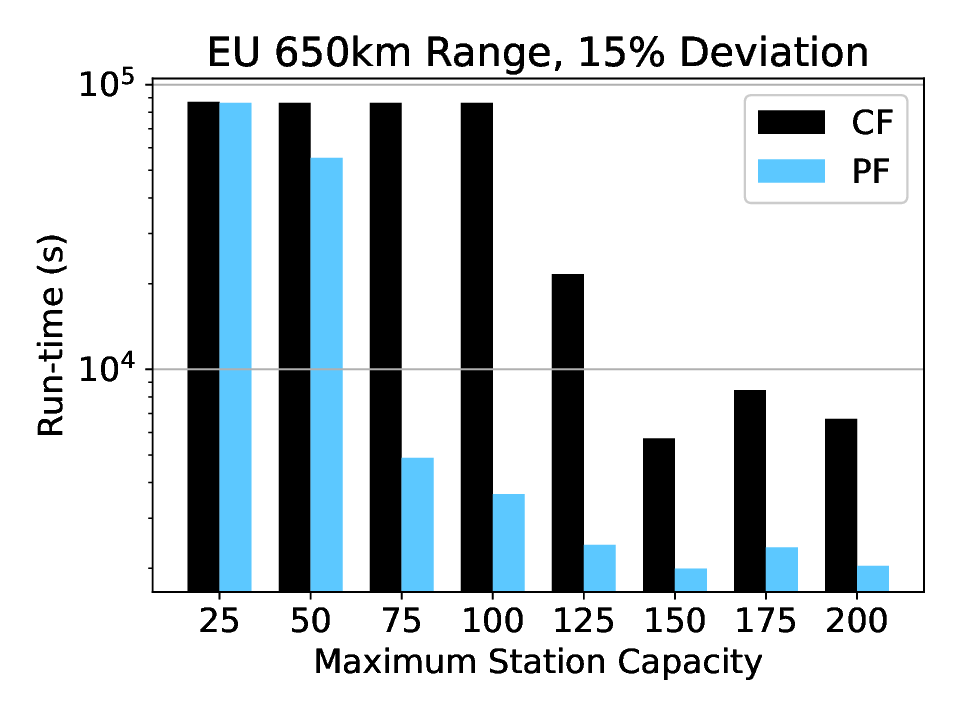}\label{fig:eutime_650_15}}
\newline
\subfloat[Utilization]{\includegraphics[width=.49\linewidth]{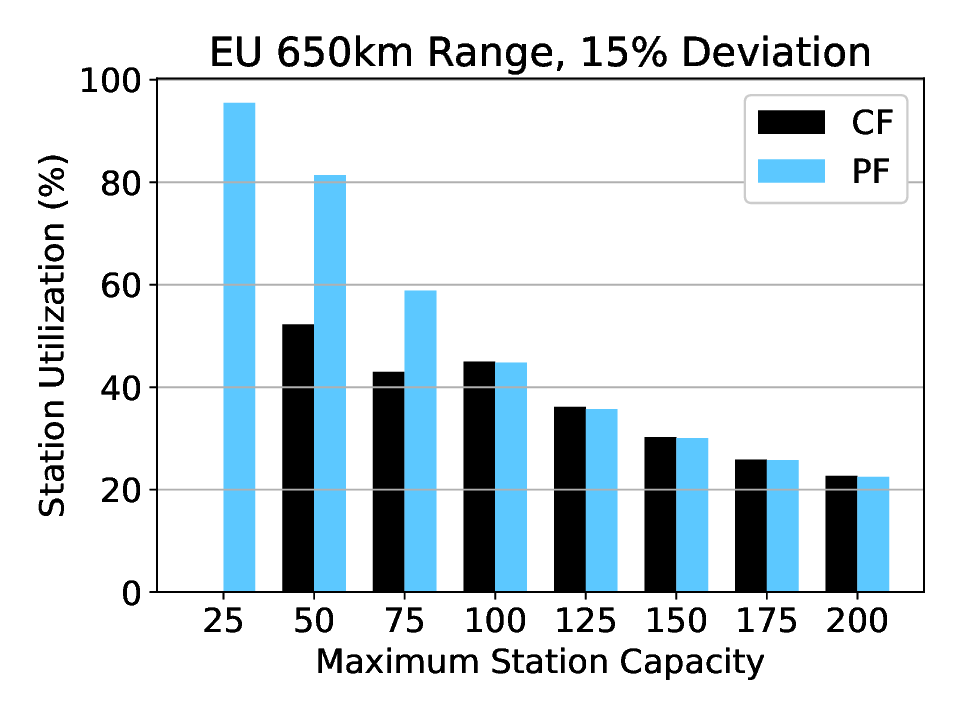} \label{fig:euutilization_650_15}}
\subfloat[Objective Value]{\includegraphics[width=.49\linewidth]{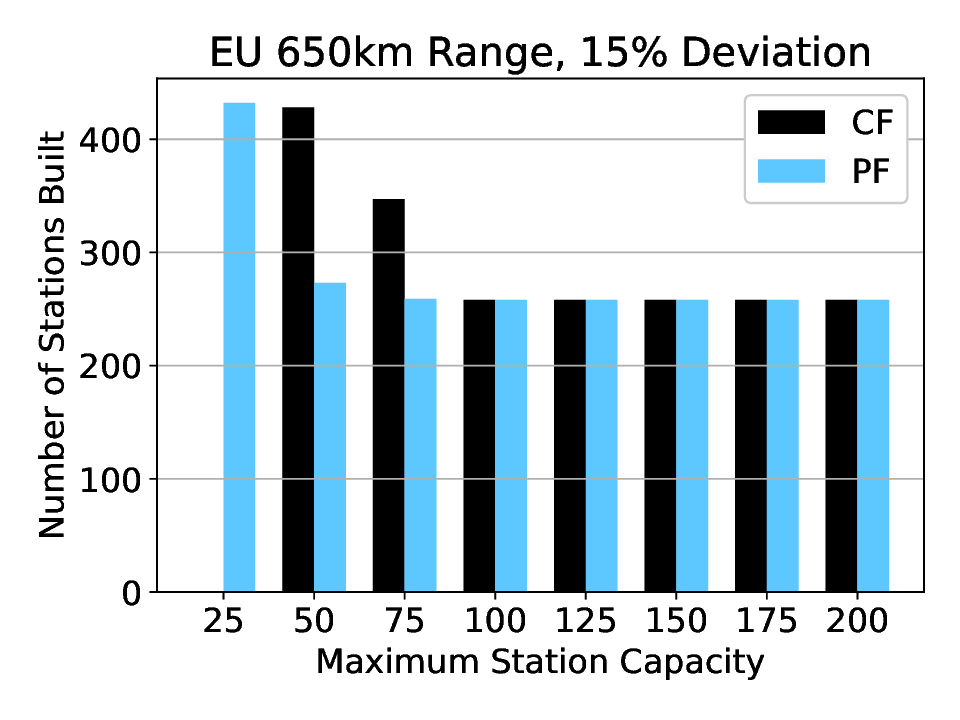}\label{fig:euobjval_650_15}}
\end{minipage}
}
\caption{Results for the EU road network\label{fig:euresults_650_15}. The plot shows results with deviation tolerance $\lambda = 15.0\%$ and vehicle range $\rangemax = 650$km for \eqref{eq:cut} and \eqref{eq:path}.}
\end{figure*}

\begin{figure*}[!ht]
\centering
{
\begin{minipage}{.8\linewidth}\centering
\subfloat[Optimality gap]{\includegraphics[width=.49\linewidth]{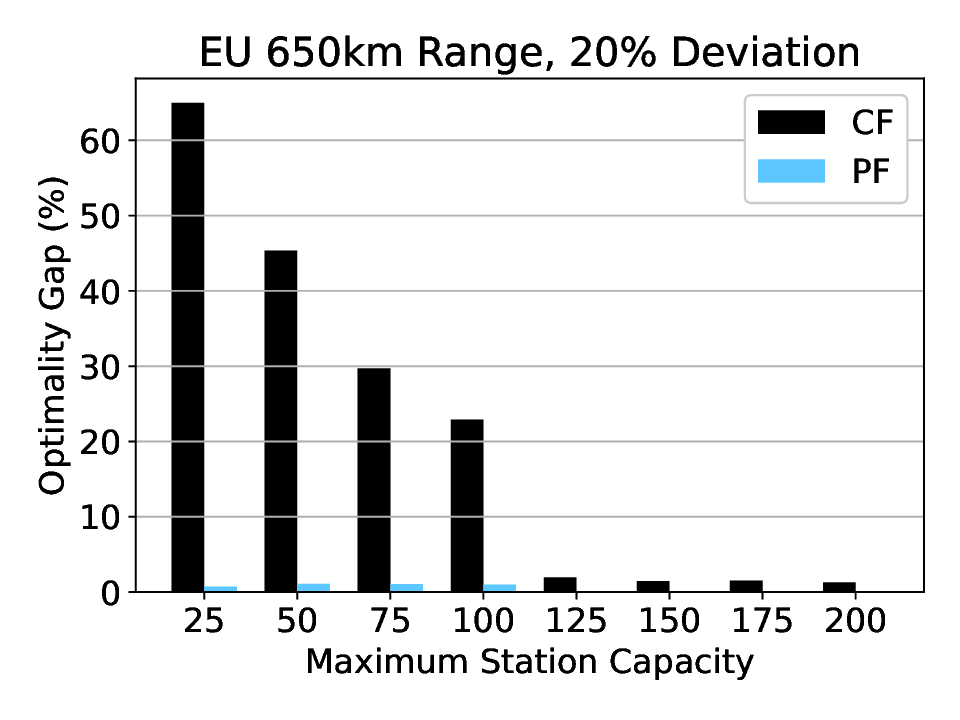} \label{fig:eugap_650_20}}
\subfloat[Solution time]{\includegraphics[width=.49\linewidth]{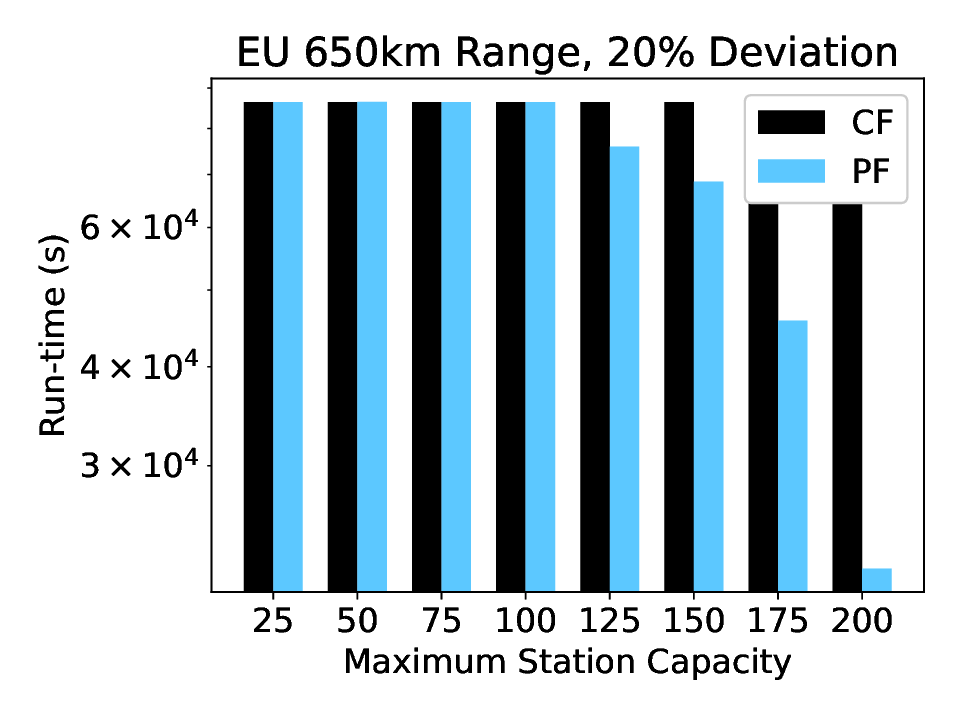}\label{fig:eutime_650_20}}
\newline
\subfloat[Utilization]{\includegraphics[width=.49\linewidth]{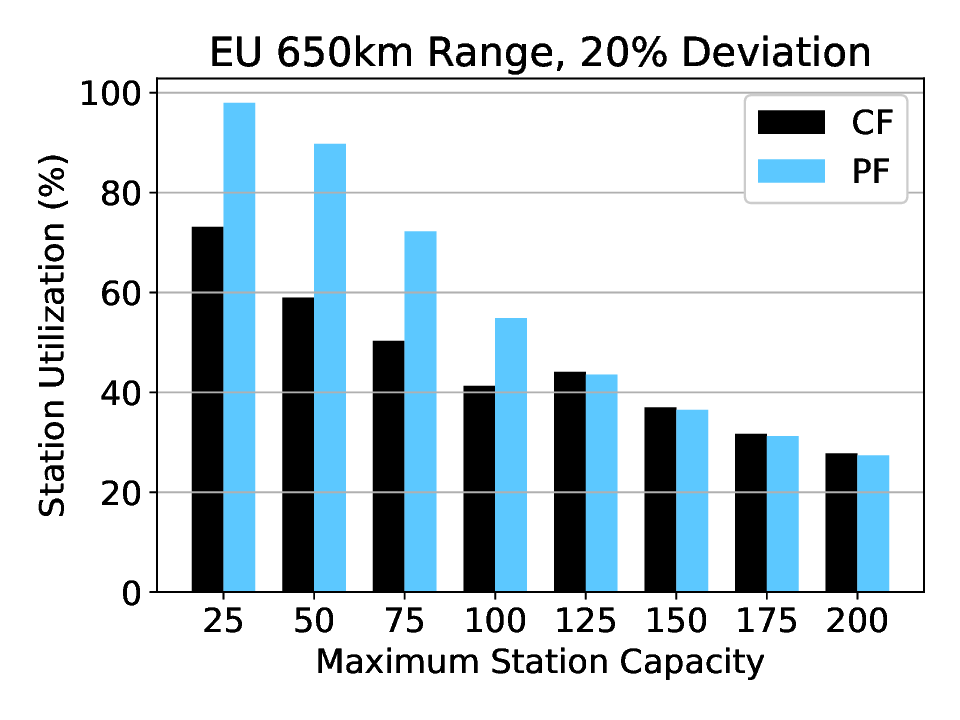} \label{fig:euutilization_650_20}}
\subfloat[Objective Value]{\includegraphics[width=.49\linewidth]{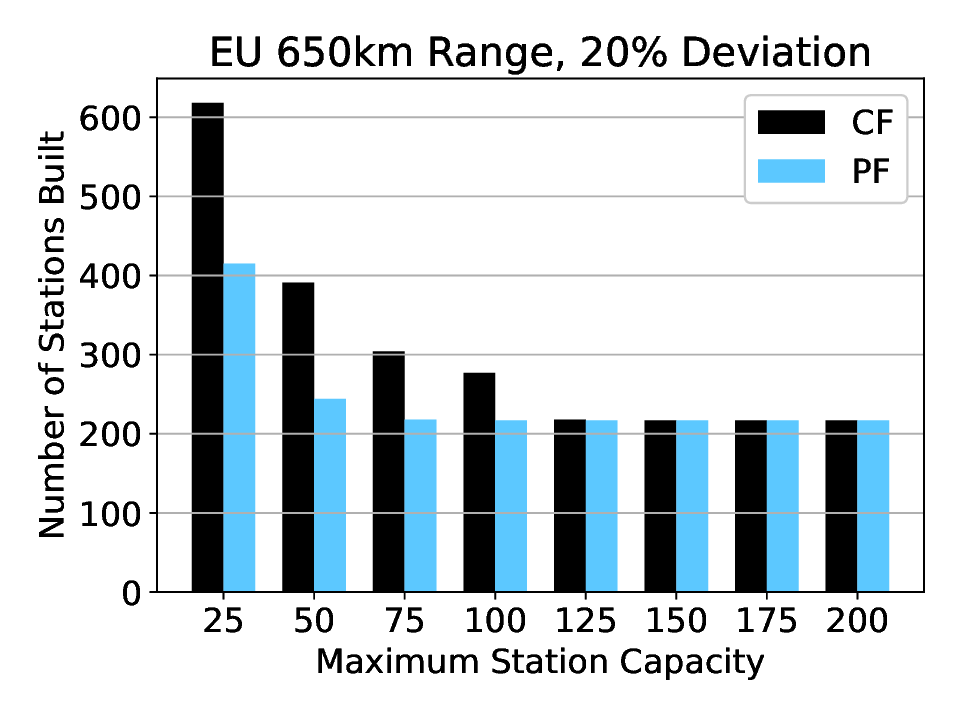}\label{fig:euobjval_650_20}}
\end{minipage}
}
\caption{Results for the EU road network\label{fig:euresults_650_20}. The plot shows results with deviation tolerance $\lambda = 20.0\%$ and vehicle range $\rangemax = 650$km for \eqref{eq:cut} and \eqref{eq:path}.}
\end{figure*}

\begin{figure*}[!ht]
\centering
{
\begin{minipage}{.8\linewidth}\centering
\subfloat[Optimality gap]{\includegraphics[width=.49\linewidth]{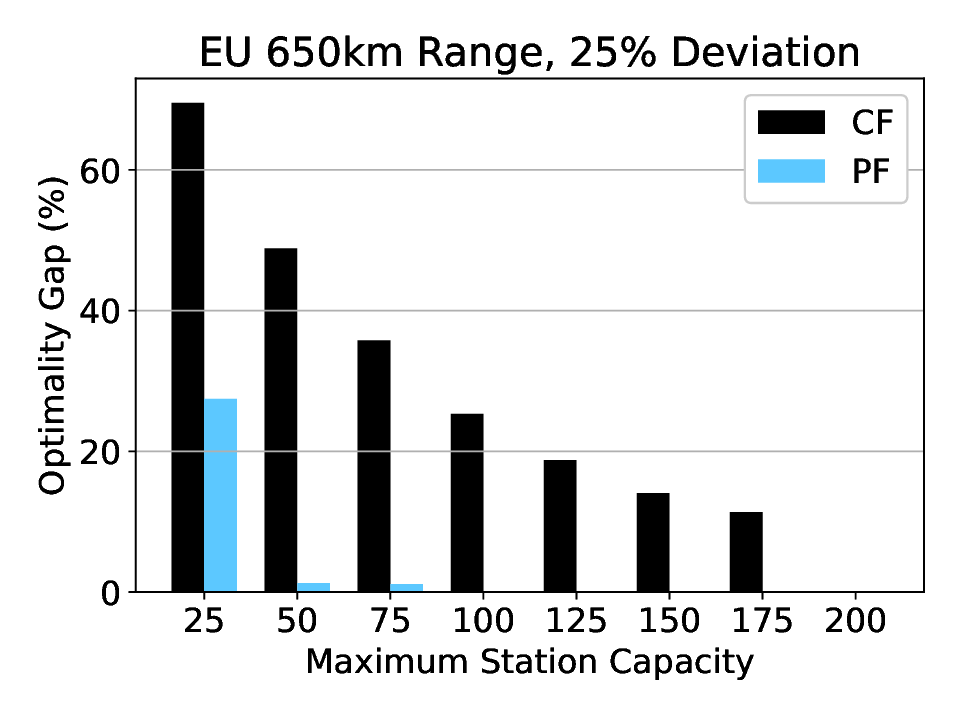} \label{fig:eugap_650_25}}
\subfloat[Solution time]{\includegraphics[width=.49\linewidth]{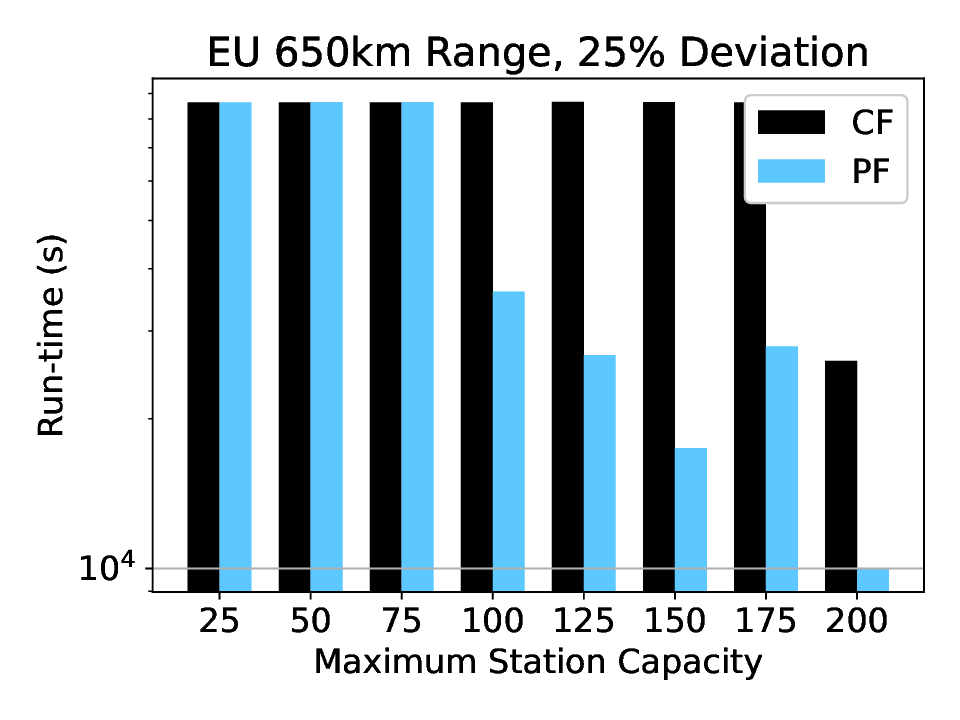}\label{fig:eutime_650_25}}
\newline
\subfloat[Utilization]{\includegraphics[width=.49\linewidth]{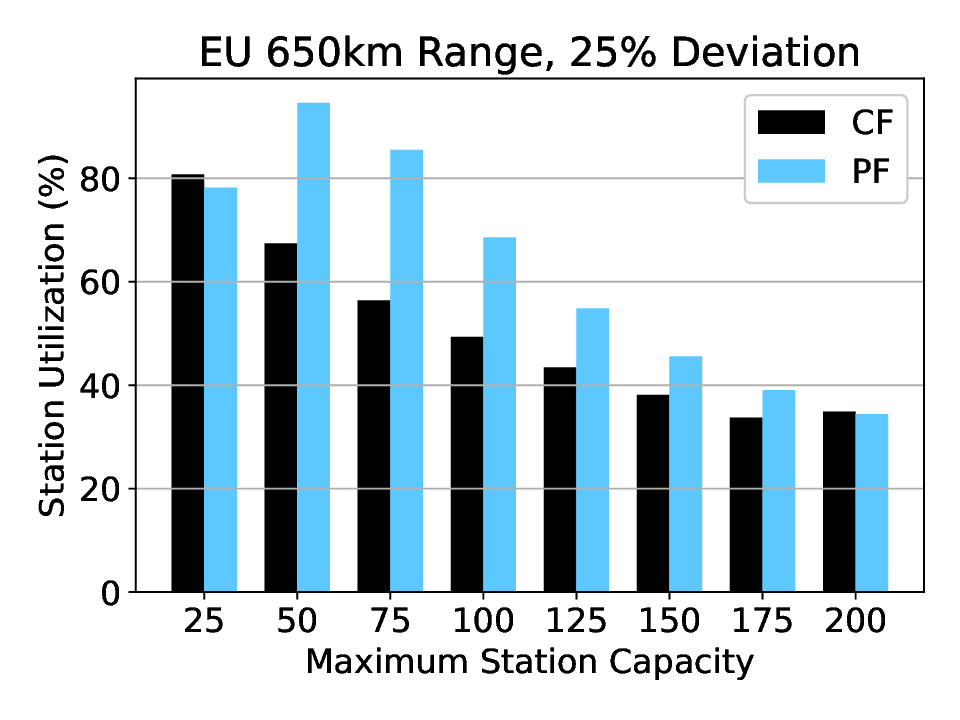} \label{fig:euutilization_650_25}}
\subfloat[Objective Value]{\includegraphics[width=.49\linewidth]{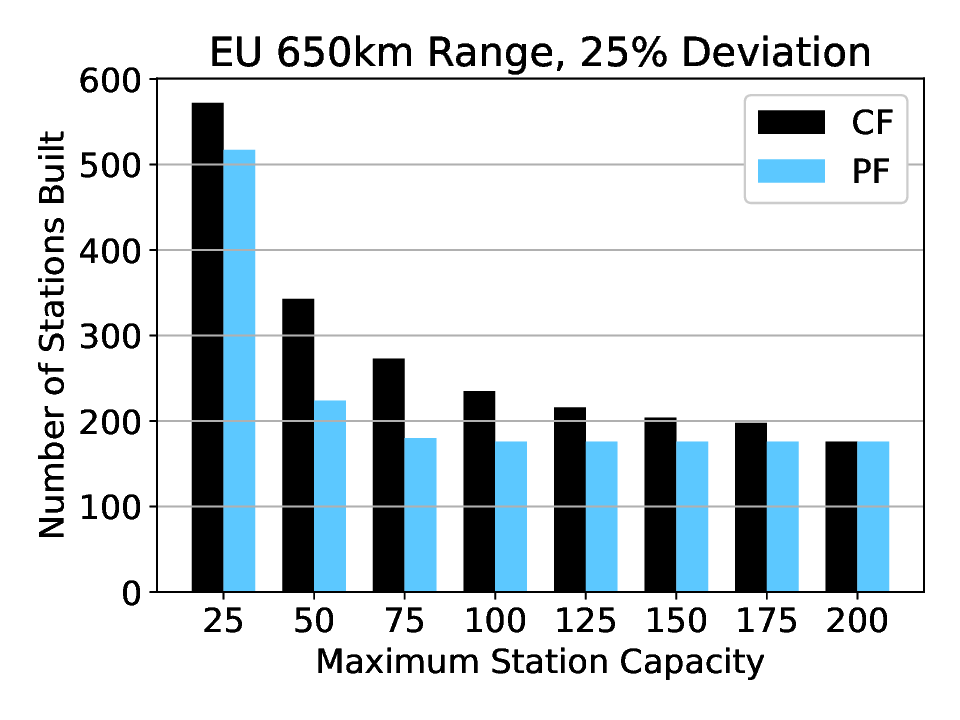}\label{fig:euobjval_650_25}}
\end{minipage}
}
\caption{Results for the EU road network\label{fig:euresults_650_25}. The plot shows results with deviation tolerance $\lambda = 25.0\%$ and vehicle range $\rangemax = 650$km for \eqref{eq:cut} and \eqref{eq:path}.}
\end{figure*}

\end{document}